\documentclass[11pt]{article}

\usepackage[utf8]{inputenc}

\usepackage{authblk}

\usepackage[a4paper,left=18.5mm,right=18.5mm,top=18.5mm,bottom=20mm]{geometry}

\usepackage[T1]{fontenc}
\usepackage{float}
\usepackage{xcolor}
\usepackage[maxnames=99,style=trad-plain,backend=biber,sortcites=true]{biblatex}
\DeclareRedundantLanguages{English}{english}

\usepackage{amsmath,amssymb,amsthm,bm,bbm}
\usepackage{stmaryrd}
\usepackage{mathtools}
\usepackage{accents}
\usepackage{IEEEtrantools}

\usepackage{enumitem}

\usepackage[ruled,vlined]{algorithm2e}
\usepackage{caption}
\usepackage{booktabs}

\usepackage{graphicx}
\usepackage{subfigure}
\usepackage{comment}

\newcommand{\R}{\mathbb R}

\newcommand{\N}{\mathbb N}

\newcommand{\dd}{\mathop{}\!\mathrm{d}}

\newcommand{\TV}{\textnormal{TV}}

\renewcommand{\d}{\mathrm{d}}
\newcommand{\BMO}{\mathrm{BMO}}

\DeclareMathOperator*{\argmin}{\arg\!\min}

\DeclareMathOperator{\osc}{osc}

\newtheorem{thm}{Theorem}[section]
\newtheorem{lem}[thm]{Lemma}
\newtheorem{cor}[thm]{Corollary}
\newtheorem{prop}[thm]{Proposition}
\theoremstyle{definition}

\newtheorem{assu}[thm]{Assumption}
\newtheorem*{assu*}{Assumption}
\theoremstyle{remark}
\newtheorem{rem}[thm]{Remark}
\newtheorem*{rem*}{Remark}
\newtheorem{exm}[thm]{Example}

\renewcommand{\leq}{\leqslant}
\renewcommand{\le}{\leqslant}
\renewcommand{\geq}{\geqslant}
\renewcommand{\ge}{\geqslant}

\newcommand{\dbE}{\mathbb{E}}
\newcommand{\dbF}{\mathbb{F}}

\newcommand{\dbP}{\mathbb{P}}
\newcommand{\dbR}{\mathbb{R}}
\newcommand{\dbQ}{\mathbb{Q}}

\newcommand{\cA}{\mathcal{A}}

\newcommand{\cF}{\mathcal{F}}

\newcommand{\cP}{\mathcal{P}}
\newcommand{\cT}{\mathcal{T}}

\newcommand{\cH}{\mathcal{H}}

\usepackage{microtype}
\usepackage{hyperref}
\hypersetup{colorlinks=true,linkcolor=blue,citecolor=blue,urlcolor=blue,
 pdftitle={Policy gradient descent as a mirror flow on path-space: a geometric approach to well-posedness and convergence},
 pdfsubject={Newton flows, natural policy gradients, particle approximation and spatial regularity}}
\usepackage{bookmark}
\allowdisplaybreaks[2]
\usepackage{etoolbox}
\AtBeginDocument{  \setlength{\abovedisplayskip}{5pt plus 1pt minus 1pt}  \setlength{\belowdisplayskip}{5pt plus 1pt minus 1pt}  \setlength{\abovedisplayshortskip}{2pt plus 1pt}  \setlength{\belowdisplayshortskip}{3pt plus 1pt minus 1pt}  \setlength{\jot}{2.4pt}}
\usepackage{titlesec}
\usepackage{needspace}
\titlespacing*{\section}{0pt}{2.2ex plus .6ex minus .2ex}{.9ex}
\titlespacing*{\subsection}{0pt}{1.9ex plus .5ex minus .2ex}{.7ex}
\titlespacing*{\subsubsection}{0pt}{1.5ex plus .4ex minus .2ex}{.5ex}
\titlespacing*{\paragraph}{0pt}{1.2ex plus .3ex minus .1ex}{.7em}
\setlist{topsep=3pt,itemsep=1pt,parsep=0pt,partopsep=0pt}

\AtEveryBibitem{\color{black}}
\defbibheading{bibintoc}[\refname]{\section*{#1}\phantomsection\addcontentsline{toc}{section}{#1}}
\newcommand{\revkeepstep}{\par\Needspace{4\baselineskip}}
\newcommand{\readingguide}[1]{\par\addvspace{2pt}\noindent#1\par}
\newcommand{\proofguide}[1]{\readingguide{#1}}
\newcommand{\revstep}[2]{\par\addvspace{3pt}\noindent\textbf{Step~#1: #2.}\enspace\ignorespaces}
\newcommand{\revsubstep}[2]{\par\addvspace{2pt}\noindent\textit{Substep~#1: #2.}\enspace\ignorespaces}
\newcommand{\revlead}[1]{\par\addvspace{2pt}\noindent\textit{#1.}\enspace\ignorespaces}

\title{Convex geometry of diffusion laws \\ and gradient flows for stochastic control}
\author[1]{Yupeng Bai}
\author[2]{Louis-Pierre Chaintron}
\author[1]{Zhenjie Ren}
\author[3]{Songbo Wang}

\affil[1]{LaMME, Université Evry Paris-Saclay, France}
\affil[2]{École Polytechnique Fédérale de Lausanne, Institute of Mathematics, Lausanne,
Switzerland}
\affil[3]{Laboratoire J.A. Dieudonné,
Université Côte d'Azur, France}
\date{}

\makeatletter
\let\PathEntropyDisplayLabel\label@in@display
\def\label@in@display#1{  \ifst@rred\else
    \if@eqnsw\PathEntropyDisplayLabel{#1}\fi
  \fi}
\makeatother

\begin{document}
\maketitle
\vspace{-24pt}
\renewcommand{\d}{\mathrm{d}}

\begin{abstract}
We study gradient-type dynamics for weak stochastic control through the laws of controlled diffusions on path space. Although the objective is generally nonconvex in the drift, a uniformly convex running cost induces a strictly convex functional of the corresponding path law and hence a natural Bregman geometry. We construct the associated implicit proximal scheme and prove its well-posedness and convergence to a continuous-time mirror flow.
When the mirror cost coincides with the running cost and the law-dependent potential is linear, the flow becomes affine in centered entropy coordinates. This yields a variational construction of the resulting Newton flow, uniform stability estimates, and exponential convergence to the optimizer. For a general uniformly convex mirror cost and a nonlinear convex potential of the law, we construct the natural-gradient flow globally in learning time using BMO estimates for quadratic BSDEs, and establish quantitative convergence and a first-order discretization error.
We further analyze interacting finite-particle approximations, proving exponential convergence to the finite-particle optimum and an \(O(N^{-1})\) normalized discrepancy from the mean-field optimum. In the Markovian setting, reflection coupling yields uniform spatial regularity and exponential convergence of the feedback controls on \(\mathbb{R}^d\), including for nonquadratic state-dependent Hamiltonians.
\end{abstract}

{\small\tableofcontents}
\clearpage

\section{Introduction}\label{sec-introduction}

On the canonical space $\Omega = C([0,T],\R^d)$, for some $T >0$ and $d \geq 1$, endowed with the augmented canonical filtration $\mathbb{F} = (\mathcal{F}_t)_{0 \leq t \leq T}$ and the Wiener measure $\dbP$ , we consider the stochastic control problem of minimizing the objective function
\begin{equation} \label{eq:IntroStoProb}
\mathcal{J}(\beta) := \mathbb{E}^{\dbQ^\beta} \bigg[ \int_0^T \ell ( t, X, \beta_t ) \d t \bigg]+ F ( \dbQ^\beta ) =: H_\ell ( \dbQ^\beta ) + F ( \dbQ^\beta ), 
\end{equation} 
where $X := ( X_t )_{0 \leq t \le T}$ denotes the canonical process, $\beta := ( \beta_t )_{0 \leq t \leq T}$ is a progressively measurable control process, and $\dbQ^\beta$ is the path measure such that
\[ W^\beta_t := X_t - \int_0^t \beta_r \d r, \quad 0 \leq t \leq T, \]
is a Brownian motion under $\dbQ^\beta$. 
A standard choice for $F$ is $F( \dbQ^\beta ) = \mathbb{E}^{\dbQ^\beta} [ g( X_T) ]$ for a terminal cost $g$, which, when the running cost is Markovian, reduces \eqref{eq:IntroStoProb} to a standard Markov control problem. However, the general setting \eqref{eq:IntroStoProb} allows for further path-dependency and mean-field running costs as in \cite{backhoff2020nonexponential}.

To find a minimizer $\beta^*$ of $\mathcal{J}$, a natural algorithm is the formal gradient descent 
\begin{equation} \label{eq:genGD}
\partial_s \beta^s = - \delta_\beta \mathcal{J}(\beta^s),
\end{equation}
for a learning time $s > 0$,
where the $L^2(\dbQ^\beta\otimes\d t)$ gradient $\delta_\beta \mathcal{J}$ satisfies
\[ \frac{\d}{\d \varepsilon} \bigg\vert_{\varepsilon = 0} \mathcal{J} ( \beta + \varepsilon \eta ) = \mathbb{E}^{\dbQ^\beta} \bigg[ \int_0^T [\delta_\beta \mathcal{J} (\beta)]_t \cdot \eta_t \d t \bigg], \]
for any admissible progressive perturbation $\eta := (\eta_t)_{0 \leq t \leq T}$. 
The well-posedness and convergence analysis of \eqref{eq:genGD} is non-trivial as the objective function $\mathcal{J}$ is not convex with respect to $\beta$ in general \cite{Fazel2018,KerimkulovHoward2020,SethiConstrained2026}.

\paragraph{Policy gradient descent as a mirror flow on path space.}
For admissible drifts whose stochastic exponentials are true martingales, the Cameron--Martin--Girsanov formula asserts that
\begin{equation} \label{eq:IntroCameron}
\frac{\d \dbQ^\beta}{\d \dbP}(X)
=
\exp\bigg[
\int_0^T \beta_t\cdot \d X_t
-\frac12\int_0^T|\beta_t|^2\,\d t
\bigg]
=:D(\beta),
\qquad \dbP\text{-a.s.},
\end{equation}
where $D(\beta)$ denotes the classical Doléans--Dade exponential \cite[Chapter~3]{karatzas1991brownian}.
Writing $\mathcal J(\beta)\equiv J(\dbQ^\beta)$, a perturbation
$\beta+\varepsilon\eta$ formally gives
\[
\left.
\frac{\d}{\d\varepsilon}
\right|_{\varepsilon=0}
\mathcal J(\beta+\varepsilon\eta)
=
\dbE^{\dbQ^\beta}
\bigg[
\delta_{\dbQ}J(\dbQ^\beta,X)
\int_0^T\eta_t\cdot\d W_t^\beta
\bigg] = \dbE^{\dbQ^\beta}
\bigg[ \int_0^T Z_t^{\beta}\cdot \eta_t \d t \bigg], 
\]
where $Z^{\beta}$ denotes the martingale-representation integrand of
$\delta_{\dbQ}J(\dbQ^\beta,X)$ under $\dbQ^\beta$, namely
\begin{equation} \label{eq:IntroMart}
\delta_{\dbQ}J(\dbQ^\beta,X)
=
\dbE^{\dbQ^\beta}
[
\delta_{\dbQ}J(\dbQ^\beta,X)
]
+
\int_0^T Z_t^{\beta}\cdot\d W_t^\beta, \quad \dbQ^\beta\text{-a.s.}
\end{equation}
Consequently, $\delta_\beta\mathcal J(\beta)=Z^\beta$ and the policy-gradient flow \eqref{eq:genGD} reads
\begin{equation}\label{eq:intro-beta-gradient}
\partial_s\beta^s=-Z^{\beta^s}.
\end{equation}

The key motivation of the present work is the observation that this equation can be written intrinsically in terms
of the corresponding path laws. 
Let $H(\dbQ\mid\dbP):=
\dbE^{\dbQ}
[
\log\tfrac{\d\dbQ}{\d\dbP}
]$ denote the usual relative entropy with respect to the Wiener measure.
From Girsanov's theorem, we notice that $H(\dbQ^\beta\mid\dbP) = H_\ell (\dbQ^\beta)$ as defined in \eqref{eq:IntroStoProb} for the quadratic cost $\ell (t,X,\beta) = \tfrac{1}{2} \vert \beta \vert^2$ \cite{BoueDupuis1998}.  
From \eqref{eq:IntroCameron}, up to an irrelevant additive constant,
\begin{equation*}\label{eq:intro-relative-entropy-derivative}
\delta_{\dbQ}H(\dbQ^\beta,X)
=
\log\frac{\d\dbQ^\beta}{\d\dbP}(X)
=
\int_0^T\beta_t\cdot\d X_t
-\frac12\int_0^T|\beta_t|^2\,\d t,
\qquad \dbP\text{-a.s.}
\end{equation*}
Hence, along any sufficiently regular curve $s\mapsto\beta^s$,
\[
\frac{\d}{\d s}
\delta_{\dbQ}H(\dbQ^{\beta^s},X)
=
\int_0^T
\partial_s\beta_t^s\cdot
\bigl(\d X_t-\beta_t^s\,\d t\bigr)
=
\int_0^T
\partial_s\beta_t^s\cdot\d W_t^{\beta^s}.
\]
For the policy gradient flow \eqref{eq:intro-beta-gradient}, the
martingale representation  \eqref{eq:IntroMart} therefore yields
\begin{equation}\label{eq:intro-KL-mirror-flow}
\frac{\d}{\d s}
\delta_{\dbQ}H(\dbQ^{\beta^s},X)
=
-\delta_{\dbQ}J(\dbQ^{\beta^s},X)
+
\dbE^{\dbQ^{\beta^s}}
[
\delta_{\dbQ}J(\dbQ^{\beta^s},X)
].
\end{equation}
The last term is a normalizing constant, reflecting the fact that
linear derivatives on the space of probability measures are defined
only modulo additive constants. 
Thus \eqref{eq:intro-KL-mirror-flow} is
precisely the mirror flow for $J$ on path space with mirror map $\dbQ \mapsto H(\dbQ\mid\dbP)$.
We refer to \cite{Beck2017,AlvarezBolteBrahic2004} for the standard theory of mirror flows.

\emph{Ordinary gradient descent with respect to the drift $\beta$ can
therefore be viewed equivalently as relative-entropy mirror descent
for the associated path law.} This elementary observation is the
starting point for the more general path-space mirror flows considered
below.

\paragraph{Why the path-law viewpoint is useful.}
The previous observation is classical in spirit, but it suggests a different way of organizing the control problem. We do not regard the geometry as an external preconditioner imposed on a parametrized family of policies. Instead, the weak formulation already associates with every admissible control a probability law on trajectories, and the running cost itself generates a convex functional on this space of laws. Its derivative therefore provides a natural dual coordinate for the control problem, while its Bregman divergence measures the cost of moving from one controlled trajectory law to another. The point of the present work is to exploit this path-space geometry directly, including for nonquadratic running costs, path-dependent controls and nonlinear functionals of the law.

\paragraph{A general proximal scheme.}
We now replace relative entropy by a generalized path entropy generated
by a convex function $\psi$ of the drift. For a controlled law
$\dbQ^\beta$, set
\begin{equation*}\label{eq:intro-Hpsi}
H_\psi(\dbQ^\beta)
:=
\dbE^{\dbQ^\beta}
\left[
\int_0^T\psi(t,X,\beta_t)\,\d t
\right].
\end{equation*}
Its linear derivative is formally given by
\[
\delta_{\dbQ}H_\psi(\dbQ^\beta,X)
=
\int_0^T\psi(t,X,\beta_t)\,\d t
+
\int_0^T
\nabla_b\psi(t,X,\beta_t)\cdot\d W_t^\beta,
\]
and the corresponding Bregman divergence \cite{Teboulle1992} is
\[
H_\psi(\dbQ^\beta\mid\dbQ^{\beta'})
:=
H_\psi(\dbQ^\beta)-H_\psi(\dbQ^{\beta'})
-\big\langle
\delta_{\dbQ}H_\psi(\dbQ^{\beta'}),
\dbQ^\beta-\dbQ^{\beta'}
\big\rangle
\nonumber =
\dbE^{\dbQ^\beta}
\bigg[
\int_0^T
d_\psi
\bigl(
t,X,\beta_t,\beta'_t
\bigr)\,\d t
\bigg],
\label{eq:intro-Hpsi-bregman}
\]
where
\begin{equation} \label{eq:introBreg}
d_\psi(t,X,b,b')
:=
\psi(t,X,b)-\psi(t,X,b')
-\nabla_b\psi(t,X,b')\cdot(b-b').
\end{equation}
For $\psi(b)=|b|^2/2$, this reduces to the usual relative entropy
between the corresponding controlled path measures.

A natural implicit approximation of the associated mirror flow \cite{Rockafellar1976,Teboulle1992,AlvarezBolteBrahic2004} is
therefore 
\begin{equation}\label{eq:intro-JKO}
\dbQ^{h,n+1}
\in
\argmin_{\dbQ}
\left\{
J(\dbQ)
+
\frac1h
H_\psi(\dbQ\mid\dbQ^{h,n})
\right\}.
\end{equation}
Equivalently, in drift coordinates,
\[
\beta^{h,n+1}
\in
\argmin_{\beta}
\left\{
\mathcal J(\beta)
+
\frac1h
\dbE^{\dbQ^\beta}
\left[
\int_0^T
d_\psi
\bigl(
t,X,\beta_t,\beta_t^{h,n}
\bigr)\,\d t
\right]
\right\}.
\]
The existence and uniqueness of the discrete scheme are established in
Proposition~\ref{prop-bregman-proximal}, while its continuous-learning-time limit
is analyzed in Theorems~\ref{thm-newton-structure} and~\ref{thm-natural-structure} below.

\paragraph{Convexity, duality and the Newton flow.}
The principal advantage of moving from the drift $\beta$ to the law
$\dbQ^\beta$ is that $
\dbQ\mapsto H_\ell(\dbQ)$
is \emph{strictly convex} on the finite-cost domain of equivalent Wiener laws considered below as soon as the running cost is strictly convex in
the drift. 
This non-trivial fact can be deduced from the observation that for admissible progressive controls $\beta^0$, $\beta^1$, we have
\[
(1-\theta)\dbQ^{\beta^0}+\theta\dbQ^{\beta^1} = \dbQ^{\beta^\theta},
\qquad 0\le\theta\le1 ,
\]
for
\begin{equation*}\label{eq:intro-mixture-drift}
\beta_t^\theta
:=
\frac{
(1-\theta)D_t^0\beta_t^0+\theta D_t^1\beta_t^1
}{
(1-\theta)D_t^0+\theta D_t^1
},
\end{equation*}
using the notation $D_t^i
:=
\dbE
[
\tfrac{\d\dbQ^{\beta^i}}{\d\dbP}
\vert
\mathcal F_t
]$, $i \in \{0,1\}$. See Lemma~\ref{lem-H-convex} below for a complete rigorous derivation.

Thus $H_\ell(\dbQ\mid\dbQ')$ is a genuine Bregman divergence associated
with a strictly convex function on path space. 
The crucial point is not only that \(H_\ell\) is convex: strict convexity turns its derivative into a \emph{coordinate on controlled path laws}.
Indeed,
\[
\dbQ\mapsto\delta_{\dbQ}H_\ell(\dbQ)
\]
is injective, modulo the irrelevant additive constants in the linear
derivative. It may therefore be regarded as a change of coordinates
on the set of admissible path measures. This is the elementary
information-geometric idea underlying the paper: rather than using the
law $\dbQ$ itself as coordinate, we use the slope of the convex entropy
functional at $\dbQ$.
We refer to the standard textbooks \cite{Amari2016,AyJostLeSchwachhofer2017} for more context about information geometry.

This idea becomes especially effective when the mirror map coincides with
the nonlinear part of the objective. Suppose for a moment that $
F(\dbQ)=\langle f,\dbQ\rangle$ 
is linear and choose $\psi=\ell$. Then $
J(\dbQ)=H_\ell(\dbQ)+\langle f,\dbQ\rangle$,
and the mirror equation becomes
\[
\frac{\d}{\d s}
\delta_{\dbQ}H_\ell(\dbQ^s)
=
-\delta_{\dbQ}H_\ell(\dbQ^s)-f+c_s,
\]
for a scalar normalization $c_s$. After centering under the fixed Wiener measure, if
\[
\overline\rho^s
:=
\delta_{\dbQ}H_\ell(\dbQ^s)
-
\dbE^\dbP [
\delta_{\dbQ}H_\ell(\dbQ^s)
],
\qquad
\overline f:=f-\dbE^\dbP [f],
\]
we simply obtain
\begin{equation}\label{eq:intro-Newton-linear}
\partial_s\overline\rho^s
=
-\overline\rho^s-\overline f.
\end{equation}
This is where the geometric viewpoint becomes effective rather than merely interpretative: in the original drift coordinate the policy dynamics is nonlinear, whereas it becomes linear in the dual entropy coordinate.
A nonlinear stochastic-control problem has therefore become a linear
ODE after passing to the entropy coordinate.
Since the linear term $F$ has zero Hessian, the Hessian of $J$ is the
Hessian of $H_\ell$. Thus using $H_\ell$ as mirror map amounts formally
to preconditioning the gradient of $J$ by its own curvature, exactly as
in a standard \emph{Newton flow} \cite{AlvarezBolteBrahic2004}.
The role of the geometry is therefore stronger than a change of metric used to precondition gradient descent: the path-entropy coordinate linearizes the entire learning dynamics.
As usual, an exponential rate of convergence follows for the Newton dynamics. The rigorous construction
and the corresponding convergence estimates are given in
Theorem~\ref{thm-newton-structure} and Corollary~\ref{cor-common-flow-decay}.

\paragraph{The metric behind the Newton flow.}
The above construction admits a simple metric interpretation. Given two
progressive perturbations $\eta$ and $\zeta$, let
\[
U^\eta
:=
\left.
\frac{\d}{\d\varepsilon}
\right|_{\varepsilon=0}
\dbQ^{\beta+\varepsilon\eta},
\qquad
U^\zeta
:=
\left.
\frac{\d}{\d\varepsilon}
\right|_{\varepsilon=0}
\dbQ^{\beta+\varepsilon\zeta}.
\]
For the quadratic entropy, the local metric generated by the Bregman
divergence is
\[
g_{\dbQ^\beta}(U^\eta,U^\zeta)
:=
-
\left.
\frac{\partial^2}
{\partial\varepsilon\,\partial\varepsilon'}
\right|_{\varepsilon=\varepsilon'=0}
H (
\dbQ^{\beta+\varepsilon\eta}
\,\vert\,
\dbQ^{\beta+\varepsilon'\zeta}
)
\nonumber =
\dbE^{\dbQ^\beta}
\bigg[
\int_0^T
\eta_t\cdot\zeta_t\,\d t
\bigg].
\]
This is the Fisher--Rao metric on the corresponding family of path
measures \cite{AyJostLeSchwachhofer2017,nielsen2020elementary}. In drift coordinates, its quadratic form is simply the
$L^2(\dbQ^\beta\otimes\d t)$ norm of the drift perturbation.

For a general mirror map $H_\psi$, the same formal computation gives instead
the Hessian metric
\begin{equation}\label{eq:intro-general-metric}
g^\psi_{\dbQ^\beta}(U^\eta,U^\zeta)
=
\dbE^{\dbQ^\beta}
\left[
\int_0^T
\eta_t^\top
\nabla_b^2\psi(t,X,\beta_t)
\zeta_t\,\d t
\right].
\end{equation}
Thus the choice of $\psi$ specifies how expensive each infinitesimal
direction of policy change should be. The quadratic case treats all
directions with the Euclidean metric, while a nonquadratic $\psi$
weights directions according to its local curvature.

\paragraph{Mean-field, natural gradient descent and BSDE.}
We now return to the general objective \eqref{eq:IntroStoProb}, in which
$F$ may be nonlinear in the path law. Under the convexity assumptions in Assumptions~\ref{ass-rc} and~\ref{ass-pe} below, the discrete scheme \eqref{eq:intro-JKO} still admits a
continuous-time limit, but the entropy coordinate no longer satisfies
the linear equation \eqref{eq:intro-Newton-linear}.
Formally, the limiting mirror equation reads
\begin{equation*}\label{eq:intro-general-mirror}
\frac{\d}{\d s}
\delta_{\dbQ}H_\psi(\dbQ^s)
=
-\delta_{\dbQ}J(\dbQ^s)+c_s.
\end{equation*}
Equivalently, if $
q^s_t:=\nabla_b\psi(t,X,\beta_t^s)$,
then we can use the martingale representation theorem to get
\begin{equation}\label{eq:intro-q-natural}
\partial_s q_t^s
=
-[\delta_\beta\mathcal J(\beta^s)]_t.
\end{equation}
In the original drift coordinate this becomes
\begin{equation*}\label{eq:intro-natural-beta}
\partial_s\beta_t^s
=
-
\bigl[
\nabla_b^2\psi(t,X,\beta_t^s)
\bigr]^{-1}
[\delta_\beta\mathcal J(\beta^s)]_t.
\end{equation*}
Consequently,
\[
g^\psi_{\dbQ^s}
\bigl(
\partial_s\dbQ^s,U
\bigr)
=
-
\bigl\langle\delta_\dbQ J(\dbQ^s,\cdot),U\bigr\rangle,
\]
for every admissible signed tangent measure $U$, or formally
\[
\partial_s\dbQ^s
=
-\operatorname{grad}_{g^\psi} J (\dbQ^s).
\]
This is precisely the natural-gradient interpretation of the mirror
flow \cite{Amari1998,RaskuttiMukherjee2015}: ordinary gradient descent is recovered when $\psi(t,X,b)=|b|^2/2$,
whereas a general $\psi$ preconditions the policy gradient by the
curvature naturally induced by the path-space Bregman divergence.

In the BSDE formulation used below, the policy gradient decomposes as
\[
[\delta_\beta\mathcal J(\beta^s)]_t
=
\nabla_b\ell(t,X,\beta_t^s)+Z_t^s,
\]
where $Z^s$ represents the effect of changing the present policy on
future and mean-field costs \cite{KerimkulovHoward2020,KerimkulovMirror2025}. Hence \eqref{eq:intro-q-natural} becomes
\[
\partial_s\nabla_b\psi(t,X,\beta_t^s)
=
-\nabla_b\ell(t,X,\beta_t^s)-Z_t^s.
\]
The rigorous well-posedness and discretization estimates for this
natural-gradient flow are established in Theorem~\ref{thm-natural-structure},
with convergence to the minimiser in Corollary~\ref{cor-common-flow-decay}.
The convergence estimates control both
$H_\psi(\dbQ^*\mid\dbQ^s)$ and $H_\psi(\dbQ^s\mid\dbQ^*)$.
These orientations are not interchangeable: the first is directly
contracted by the mirror scheme, whereas the second is bounded through
$\lambda_{\ell|\psi}H_\psi(\dbQ^s\mid\dbQ^*)
\le J(\dbQ^s)-J(\dbQ^*)$.
Both bounds depend on $H_\psi(\dbQ^*\mid\dbQ^0)$; no initial bound in the opposite orientation is required.

A key difficulty is that the geometry itself depends on the evolving
law, while the BSDE terminal condition contains the law derivative
$\delta F(\dbQ^s,X)$. One must therefore control simultaneously the
change of measure, the martingale integrand $Z^s$, and the inverse map
from the dual coordinate $\nabla\psi(\beta^s)$ back to an admissible
drift. The BMO estimates developed below provide precisely the
stability needed to close this argument globally in learning time \cite{kazamaki1994continuous,ChikvinidzeMania2014}.

\paragraph{Finite-particle approximation.}
The path-law formulation also makes it natural to consider a finite-particle
approximation of the learning dynamics. Given \(N\) trajectories
\(X^1,\ldots,X^N\), we replace the law entering the mean-field potential
\(F\) by the empirical path measure
\[
    \mu^N
    :=
    \frac{1}{N}\sum_{i=1}^N \delta_{X^i},
\]
and consider the corresponding stochastic control problem on
\(\Omega^N\). This yields an interacting \(N\)-particle Bregman proximal
scheme and, in the continuous-learning-time limit, a finite-particle
mirror flow. For mean-field optimization regularized by Fisher information,
\cite{CCRW2026} develop a related minimizing-movement scheme based on
relative entropy and propose a Monte Carlo particle approximation of
the resulting mean-field Schr\"odinger dynamics.

For every fixed \(N\), the product path space inherits the convexity and
BMO structure of the mean-field problem. Theorem~\ref{thm-particle-flow}
shows that the finite-particle proximal scheme is well posed, converges
to a continuous-learning-time flow, and that this flow converges
exponentially to the minimizer \(\dbQ_N^*\) of the \(N\)-particle
objective. The constants obtained at this stage may depend on \(N\).
The more delicate question is how the finite-particle flow compares with
the mean-field optimizer \(\dbQ^*\). Even if the particles are independent
with common law \(\dbQ^*\), the empirical potential does not in general
reproduce the limiting potential exactly:
\[
    \dbE^{(\dbQ^*)^{\otimes N}}
    \left[
        F(\mu^N)
    \right]
    \neq
    F(\dbQ^*).
\]
This leads naturally to the empirical-potential defect
\[
    \varepsilon_N
    :=
    N
    \big(
        \dbE^{(\dbQ^*)^{\otimes N}}
        [
            F(\mu^N)
        ]
        -
        F(\dbQ^*)
    \big).
\]
Under our assumptions on \(F\), this quantity is bounded uniformly in
\(N\). Combining this empirical comparison with the Bregman contraction
of the finite-particle flow gives, for product initial data,
\[
    \frac{1}{N}
    H_\psi^N
    \left(
        (\dbQ^*)^{\otimes N}
        \mid
        \dbQ_s^N
    \right)
    \leq
    e^{-\lambda_{\ell|\psi}s}
    H_\psi
    \left(
        \dbQ^*
        \mid
        \dbQ_0
    \right)
    +
    \frac{C_{\mathrm{emp}}}
         {\lambda_{\ell|\psi}N}.
\]
For the forward orientation, the estimate for $s>0$ is
\[
 \frac1N H_\psi^N\bigl(\dbQ_s^N\mid(\dbQ^*)^{\otimes N}\bigr)
 \le
 \frac{H_\psi(\dbQ^*\mid\dbQ_0)}{e^{\lambda_{\ell|\psi}s}-1}
 +\frac{\varepsilon_N}{\lambda_{\ell|\psi}N}.
\]
It follows from the objective comparison
$\lambda_{\ell|\psi}H_\psi^N(\dbQ_s^N\mid(\dbQ^*)^{\otimes N})
\le J_N(\dbQ_s^N)-NJ(\dbQ^*)$, rather than from reversing the
preceding divergence. Projecting the joint drift onto each coordinate
filtration also shows that the sum of the one-particle forward
divergences is bounded by this joint forward divergence.
This provides control of the average marginal error without assuming
that the particles are exchangeable; the projection argument is given in Lemma~\ref{lem-particle-marginal-projection}.
Thus the two approximation mechanisms are separated: the optimization
error decreases exponentially in the learning time \(s\), while the
empirical approximation leaves a residual error of order \(N^{-1}\)
after normalization in either orientation; see
Theorem~\ref{thm-particle-long-time} and
Corollary~\ref{cor-particle-product-initial}.

We emphasize that this is an equilibrium comparison with the mean-field
optimizer, rather than a propagation-of-chaos result for \(\dbQ_s^N\)
at every fixed learning time. In the related settings of overdamped
and kinetic mean-field Langevin dynamics, \cite{CRW2025} and
\cite{CLRW2024}, respectively, obtain exponential relaxation towards
the mean-field minimum up to an \(O(N^{-1})\) error in normalized
free energy, and further establish uniform-in-time propagation of chaos.
Establishing such a dynamical particle limit for the present
path-space mirror flow, preferably with estimates uniform jointly
in \(N\) and \(s\), is a natural direction for further work.

\paragraph{The Markov case.}
The particular choice
\[
F(\dbQ)
=
\dbE^{\dbQ}[g(X_T)]
\]
corresponds to the classical finite-horizon problem with terminal cost
$g$ \cite{FlemingSoner}. When the running cost is Markovian $
\ell(t,X,b) \equiv \ell(t,X_t,b)$
and the initial policy is itself Markovian, then the Newton flow remains in
the class of feedback controls $
\beta_t^s=\beta^s(t,X_t)$.
Let
\[
v^s(t,x)
=
\dbE_{t,x}^{\dbQ^{\beta^s}}
\bigg[
g(X_T)
+
\int_t^T
\ell(r,X_r,\beta^s(r,X_r))\,\d r
\bigg]
\]
denote the value obtained by keeping the current policy $\beta^s$.
Then $v^s$ solves the policy-evaluation equation \cite{FlemingSoner,KerimkulovHoward2020}
\begin{equation*}\label{eq:intro-policy-evaluation}
\partial_t v^s
+
\frac12\Delta v^s
+
\beta^s\cdot\nabla_xv^s
+
\ell(t,x,\beta^s)
=
0,
\qquad
v^s(T,x)=g(x),
\end{equation*}
while the Newton improvement equation becomes
\begin{equation}\label{eq:intro-policy-improvement}
\partial_s
\nabla_b\ell(t,x,\beta^s(t,x))
=
-\nabla_b\ell(t,x,\beta^s(t,x))
-\nabla_xv^s(t,x).
\end{equation}
Thus the learning dynamics alternates continuously between policy
evaluation in physical time $t$ and policy improvement in learning
time $s$. At equilibrium $s = +\infty$, this becomes
\[
\nabla_b\ell(t,x,\beta^*(t,x))
+
\nabla_xv^*(t,x)
=
0,
\]
or equivalently, using the Hamiltonian $H(t,x,p)
:=
\sup_{b\in\R^d}
\{
-b\cdot p-\ell(t,x,b)
\}$, this reads
\[
\beta^*(t,x)
=
-D_pH(t,x,\nabla_xv^*(t,x)).
\]
Substituting this identity into the policy-evaluation equation recovers
the classical Hamilton--Jacobi--Bellman equation \cite{FlemingSoner}
\begin{equation*}\label{eq:intro-HJB}
\partial_t v^*
+
\frac12\Delta v^*
-
H(t,x,\nabla_xv^*)
=
0,
\qquad
v^*(T,x)=g(x).
\end{equation*}

A natural question is whether the spatial regularity of the initial
feedback and of the data is propagated uniformly along the entire
learning flow, and whether convergence at the level of path laws can
be upgraded to convergence of the feedbacks themselves. On the whole
space $\R^d$, this requires estimates which do not rely on compactness
or periodic boundary conditions. We address this question using reflection
coupling \cite{LindvallRogers1986,Eberle2016,Conforti2023}. The resulting bounds control the spatial Lipschitz norm of
the policy-evaluation function without differentiating the feedback
drift, and then allow one to bootstrap the regularity of
\eqref{eq:intro-policy-improvement}. This yields uniform regularity and
exponential convergence of the feedbacks for general nonquadratic
Hamiltonians; see Theorem~\ref{thm-markov-regularity} and
{Remark}~\ref{rem-markov-smooth}.

\paragraph{Perspectives.}\label{sec-intro-perspectives}
A first direction is to turn the BMO stability estimates into error bounds for inexact policy evaluation, combining time discretisation with regression or particle approximations of the BSDE. A second is to establish particle limits at each fixed learning time, rather than only comparison with the mean-field equilibrium, and to balance the particle, evaluation and optimisation errors in one implementable algorithm. The freedom to choose $\psi$ also suggests adaptive preconditioning, provided the relative-convexity and admissibility bounds can be kept uniform. Finally, extending the construction to constrained drifts, nonconstant diffusion matrices and weaker growth assumptions, and using dissipativity to control the dependence on the physical horizon, would broaden its scope. These questions require additional estimates beyond those proved here.

\subsection{Outline of the article}\label{sec-outline}

Section~\ref{section-pgd} gathers the assumptions and states the main
results. Theorems~\ref{thm-newton-structure} and~\ref{thm-natural-structure}
separate the two construction mechanisms, while
Corollary~\ref{cor-common-flow-decay} states their common initial-divergence
bounds in both orientations once. For the Newton flow, the static control family reconstructs
admissible laws from the affine entropy coordinate and yields stronger
uniform-in-time estimates. The general case instead requires nonlinear
BMO/BSDE stability in the adapted coordinate $q=\nabla\psi(\beta)$.

Section~\ref{sec-foundations} develops the common analytic tools:
BMO change-of-measure estimates, BSDE stability, path-entropy calculus
and well-posedness of the Bregman proximal step. The objective-coercivity lemma bounds the forward divergence by the objective gap.
Section~\ref{sec-flow-proofs} then constructs the two flows. The Newton
proof uses static stability; the general proof combines a BMO bootstrap,
local construction, consistency between the penalised and evaluation
BSDEs, and global continuation. Section~\ref{sec-common-flow-decay-proof}
then verifies, separately for the Newton flow and the general flow, the
limits needed to pass the discrete Bregman contraction and objective
estimate of Proposition~\ref{prop-bregman-proximal} to continuous
learning time. For the Newton flow this passage uses the affine entropy
coordinate and does not require the upper-Hessian assumption used in
the general-flow continuity argument. Forward-divergence
control then follows from the objective gap.

Section~\ref{sec-finite-particle-proofs} transfers the general-flow
construction and the resulting initial-divergence estimates to each
fixed product canonical space. The empirical-potential
estimate controls the discrepancy from the product mean-field optimum.
Combining it with the three-point Bregman identity gives contraction with the reference law in the first argument. Retaining the objective term and summing with geometric weights gives an objective bound, which also controls the forward divergence. Passing the objective and fixed-first-argument estimates to continuous learning time proves Theorem~\ref{thm-particle-long-time}, and product
initialization yields Corollary~\ref{cor-particle-product-initial}.
A final coordinate-filtration projection argument bounds the sum of the
one-particle forward divergences by the joint forward divergence.
Finally, Section~\ref{sec-reflection-proof} proves the Markov regularity
results. Restarted static control problems first bound the feedbacks;
reflection coupling then estimates policy evaluation without
differentiating the drift, and the improvement equation propagates
regularity and yields feedback convergence. 

\subsection{Comparison with existing literature}\label{sec-literature}
\paragraph{Convex optimisation and Bregman proximal methods.}
The discrete construction belongs to the proximal-point tradition: Rockafellar \cite{Rockafellar1976} studies implicit minimisation with a quadratic movement penalty, while Teboulle \cite{Teboulle1992} develops entropic proximal mappings using Bregman distances and related divergences. Beck and Teboulle \cite{BeckTeboulle2003} analyse mirror descent with a linearised objective; our step instead minimises the full objective plus the Bregman penalty. This distinction matters at nonzero step size, although the associated first-order evolutions are closely related. Alvarez, Bolte and Brahic \cite{AlvarezBolteBrahic2004} establish convergence and dual descriptions for Hessian-metric flows, and Lu, Freund and Nesterov \cite{LuFreundNesterov2018} develop relative smoothness and relative strong convexity. The Hessian comparison between $\ell$ and $\psi$ plays the corresponding role here. The new issue is not the proximal architecture itself, but its realisation on admissible diffusion laws with path-dependent, potentially unbounded drift costs.

Aubin-Frankowski, Korba and L\'eger \cite{aubin2022mirror} prove convergence of mirror descent on measure spaces under relative smoothness and convexity, with applications to Sinkhorn and expectation maximisation. Aubin-Frankowski, Sodini and Stefanelli \cite{aubin2025evolution} treat evolution variational inequalities for general costs, including Bregman divergences, and their approximation by implicit and explicit schemes. These works place the present construction in an established measure-valued optimisation framework. They do not, by themselves, provide the BMO estimates, law-dependent BSDE stability or reconstruction of a diffusion law from its entropy derivative needed in this paper.

\paragraph{Entropy, weak control and non-exponential large deviations.}
The relation between quadratic drift energy and path-space relative entropy is classical; Bou\'e and Dupuis \cite{BoueDupuis1998} express exponential Brownian functionals through stochastic control, and Lehec \cite{Lehec2013} studies the corresponding entropy representation. Lacker \cite{lacker2020non} replaces entropy in Sanov's variational formulation by a general convex penalty. Backhoff-Veraguas, Lacker and Tangpi \cite{backhoff2020nonexponential} specialise this programme to Wiener space and convex drift costs, relating non-exponential Sanov and Schilder principles to BSDEs, Schr\"odinger problems and mean-field control. This is a direct motivation for considering $H_\ell(\dbQ)+F(\dbQ)$ rather than only quadratic entropy plus a terminal expectation. Their Sanov theorem treats a deterministic integrand $g(t,b)$ and an approximation by chopped and rescaled pieces of one controlled path. It is not a theorem about the interacting policy flow studied here. For our state-independent drift-cost subclass, suitable additive path costs may be incorporated into $F$; an arbitrary path-dependent integrand requires further analysis. Our question is instead how to minimise the weak-control functional by an admissible, quantitatively stable learning dynamics.

\paragraph{Information geometry and natural policy gradients.}
Amari \cite{Amari1998} introduces natural-gradient preconditioning through statistical geometry, and Kakade \cite{Kakade2002} applies this principle to policy optimisation. The differential-geometric background, including Fisher metrics and dual coordinates, is developed in \cite{Amari2016,AyJostLeSchwachhofer2017}; Raskutti and Mukherjee \cite{RaskuttiMukherjee2015} explain the dual-geometric relation between mirror descent and natural gradient. Here the statistical object is the law of the entire trajectory. The quadratic mirror gives Fisher--Rao on path-law scores, whereas a general $\psi$ yields the weighted drift metric \eqref{eq:intro-general-metric}; this need not equal a Fisher metric on state-conditioned action distributions.

Fazel, Ge, Kakade and Mesbahi \cite{Fazel2018} prove global convergence of policy-gradient methods for the linear-quadratic regulator, despite nonconvexity in policy parameters. M\"uller, \c{C}ayc\i{} and Mont\'ufar \cite{Mueller2025} analyse Fisher--Rao flows for linear programmes and state-action natural policy gradients. Kerimkulov, Leahy, \v{S}i\v{s}ka, Szpruch and Zhang \cite{KerimkulovFR2026} establish well-posedness, exponential convergence and perturbation stability of a Fisher--Rao flow for entropy-regularised Markov decision processes on Polish spaces. Thus neither natural-policy dynamics nor their exponential convergence is new in itself. Our contribution uses convexity in a diffusion path law, allows non-Markov costs and nonlinear convex law potentials, and identifies an affine Newton evolution in centred path-entropy coordinates.

\paragraph{Policy improvement, BSDEs and regularised controls.}
Kerimkulov, \v{S}i\v{s}ka and Szpruch \cite{KerimkulovHoward2020} prove exponential convergence and stability for Howard improvement and gradient iteration in controlled diffusions using BSDEs. Their evaluation--greedy-improvement step differs from our implicit minimisation over path laws. For the Newton flow, greedy improvement is a unit explicit step in $q=\nabla\ell(\beta)$, whereas our proximal step uses the BSDE of the penalised control problem. Sethi and \v{S}i\v{s}ka \cite{SethiMSA2023} identify a gradient-flow limit of a modified method of successive approximations, with a first-order discretisation estimate; the formulation is based on the Pontryagin system rather than convexity in the trajectory law.

\v{S}i\v{s}ka and Szpruch \cite{SiskaSzpruch2024} construct decreasing gradient flows for entropy-regularised, measure-valued controls and obtain exponential convergence under sufficient convexity. Kerimkulov, \v{S}i\v{s}ka, Szpruch and Zhang \cite{KerimkulovMirror2025} study mirror descent for measure-valued controls with general regularisers, using relative convexity and BSDE estimates in BMO. These estimates are particularly close in spirit to our stability analysis, but their mirror variable is an action measure, not a full diffusion law. Sethi, \v{S}i\v{s}ka and Zhang \cite{SethiAnnealing2025} treat entropy annealing for policy mirror descent in exit-time control, while Sethi and \v{S}i\v{s}ka \cite{SethiConstrained2026} study bounded convex action sets using performance-difference identities and Sobolev regularity. In our setting the mirror cost may coincide with the original running cost; it is not necessarily an additional exploration regulariser. This flexibility is accompanied by explicit assumptions on drift integrability, uniform convexity and the diffusion coefficient.

\paragraph{Particle approximation and spatial regularity.}
The particle results control learning-time convergence toward the mean-field optimum up to an empirical-potential defect. They do not establish propagation of chaos at each fixed learning time. Lacker's entropy hierarchies \cite{LackerHier} and Chaintron's quasi-continuity method \cite{Chaintron2024} concern, respectively, quantitative chaos and large-deviation/fluctuation limits for interacting diffusions in physical time; these are relevant directions for strengthening the present particle analysis, not consequences of our defective estimate.

For the Markov results, Fleming and Soner \cite{FlemingSoner} provide the control/HJB framework, and Porretta and Priola \cite{PorrettaPriolaGlobalLip} establish global Lipschitz regularisation for parabolic equations. Lindvall and Rogers \cite{LindvallRogers1986} construct reflection couplings for multidimensional diffusions. Eberle \cite{Eberle2016} derives quantitative contraction in adapted distances, and Conforti \cite{Conforti2023} applies reflection to controlled diffusions and HJB turnpikes. We use finite-horizon bounded-drift smoothing, without assuming dissipativity or asserting a contraction uniform in the physical horizon. 

Taken together, the results connect a well-posed Bregman proximal scheme on diffusion laws with an affine entropy-coordinate Newton flow, a nonlinear natural-gradient flow controlled in BMO, and quantitative consequences for particle comparison and Markov feedbacks. The novelty lies in this path-law construction and its stochastic stability analysis, rather than in a new abstract definition of mirror descent or natural gradient.

\subsection{Notations}\label{sec-notation}
\begin{itemize}[leftmargin=1.7em,itemsep=2pt]
\item \textbf{Times and paths.} $T>0$ and $t\in[0,T]$ are the physical horizon and time; $s\ge0$, $h>0$ and $n\in\N$ are learning time, step size and iteration index. $\Omega=C([0,T],\R^d)$, $X$ is the coordinate process, $\dbP$ is Wiener measure with $X_0=0$, $\dbF=(\cF_t)$ its augmented filtration, and $\cT$ the $[0,T]$-valued stopping times.
\item \textbf{Laws and pairings.} $\cP(\Omega)$ is the space of probability laws; $\langle k,\dbQ\rangle=\int k\,\dd\dbQ$, $\dbE^{\dbQ}$ denotes expectation and $\dbE=\dbE^{\dbP}$. The symbols $\sim$ and $\Rightarrow$ denote equivalence and weak convergence of laws. Our total-variation convention is $\|\mu-\nu\|_{\TV}=\sup_{\|k\|_\infty\le1}|\langle k,\mu-\nu\rangle|$.
\item \textbf{Drifts and densities.} $\beta$ is a progressive drift, $W^\beta=X-\int_0^\cdot\beta_r\,\dd r$ and $D_t(\beta)=\exp(\int_0^t\beta_r\,\dd X_r-\frac12\int_0^t|\beta_r|^2\,\dd r)$. Write $D(\beta)=D_T(\beta)$ and $\dbQ^\beta=D(\beta)\dbP$ when this is a probability law. The admissible classes $\overline{\cA}$, $\cA$ and $\cA_\varphi$ are defined below.
\item \textbf{Costs and derivatives.} $\ell$ is the running cost, $\psi$ the mirror cost and $\varphi$ either one. Unless specified otherwise, $\nabla$ and $\nabla^2$ act on the drift argument. $F$ is a functional of laws and $J=H_\ell+F$. The derivatives $\delta F\equiv\delta_\dbQ F$ and $\delta H_\varphi\equiv\delta_\dbQ H_\varphi$ are defined modulo constants; $H_\varphi(\dbQ\mid\dbQ')$ is the Bregman divergence. $H$ with two law arguments denotes ordinary relative entropy. A star denotes an optimiser, and a bar centres under $\dbP$.
\item \textbf{Spaces and constants.} $\BMO(\dbQ)$ consists of progressive integrands with bounded conditional quadratic energy; $\BMO=\BMO(\dbP)$. $\mathcal S^\infty$ denotes bounded adapted continuous processes; $L^p=L^p(\dbP)$ and $L^2_0(\dbP)$ is its centred Hilbert subspace. Conjugate exponents satisfy $1/p+1/q=1$. Constants $C$, $C_S$ and $C_{S,N}$ may depend on the stated data and on $T$; their subscripts indicate additional dependence.
\item \textbf{Flows, particles and spatial norms.} $\dbQ^{h,s}=\dbQ^{h,\lfloor s/h\rfloor}$, with the same convention for drifts; $q^s=\nabla\psi(\beta^s)$. The full policy gradient $Z^{\beta^s}$ in the introduction equals $\nabla\ell(\beta^s)+Z^s$, with $Z^s$ the policy-evaluation BSDE integrand. Superscript $N$ denotes the joint particle setting and $\mu^N=N^{-1}\sum_i\delta_{X^i}$. For Markov data, $\operatorname{Lip}$ and $\osc$ are the spatial Lipschitz seminorm and oscillation, and $C_b^m(\R^d)$ has norm given by bounded spatial derivatives up to order $m$. Spatial norms are uniform in $t$ only when so stated. The convex Hamiltonian $H(t,x,p)=-\cH_\ell(t,x,p)$ has three arguments and is distinct from relative entropy.
\end{itemize}

\section{Statement of the main results}\label{section-pgd}
We first collect the assumptions. They are scoped explicitly: the Newton
results use the linear-potential hypothesis and $\psi=\ell$; the general
natural flow uses the upper-Hessian hypothesis instead; the particle
and Markov assumptions are invoked only in their corresponding results.
The admissible classes and their product versions are defined below.
The physical horizon $T$ is fixed throughout.

\begin{assu}[Assumptions on the running cost and Bregman potential]
\label{ass-rc}
We assume:
\begin{enumerate}
\item For every $b\in\dbR^d$, the mappings \((t,\omega)\mapsto \ell(t,\omega,b), \qquad (t,\omega)\mapsto \psi(t,\omega,b)\) are $\dbF$-progressively measurable;
\item For every
$(t,\omega)\in[0,T]\times\Omega$, the mappings $b\mapsto\ell(t,\omega,b)$, $ b\mapsto\psi(t,\omega,b)$ are $\mathcal B(\dbR^d)$-measurable and belong to $C^2(\dbR^d)$, and there exist constants $c_\ell,c_\psi>0$, independent of
$(t,\omega)$, such that
\[
c_\ell I_d
\le
\nabla^2\ell(t,\omega,b),
\qquad
c_\psi I_d
\le
\nabla^2\psi(t,\omega,b),
\]
for every \((t,\omega,b)\in[0,T]\times\Omega\times\dbR^d;\) \item There exists a constant \(\lambda_{\ell|\psi}>0\) such that $\forall\,
(t,\omega,b)
\in
[0,T]\times\Omega\times\dbR^d$,
\[
\lambda_{\ell|\psi}
\nabla^2\psi(t,\omega,b)
\le
\nabla^2\ell(t,\omega,b)
.
\]
Equivalently, for every $b,b'\in\dbR^d$,
\[
\ell(t,\omega,b) -\ell(t,\omega,b') -\nabla\ell(t,\omega,b') \cdot(b-b') \ge \lambda_{\ell|\psi} \Bigl( \psi(t,\omega,b) -\psi(t,\omega,b') -\nabla\psi(t,\omega,b') \cdot(b-b') \Bigr).
\]

\item There exist two constant $L_\ell,L_\psi>0$ such that,
\[
|\nabla\ell(t,\omega,b)|\le L_\ell(1+|\nabla\psi(t,\omega,b)|),
\qquad
|\nabla\psi(t,\omega,b)|\le L_\psi(1+|\nabla\ell(t,\omega,b)|)
.\]
\end{enumerate}
\end{assu}
These hypotheses are used in the Hamiltonian minimisation and the strict-convexity argument. They connect a pointwise cost in the drift to a convex functional of laws.

\begin{assu}[Assumptions on the potential energy]\label{ass-pe}
The potential energy functional $F$ satisfies the following conditions:
\begin{enumerate}
\item
The functional $F$ is convex on $\cP(\Omega)$, i.e., for every
$\dbQ,\dbQ'\in\cP(\Omega)$ and $\theta\in[0,1]$,
\[
F\bigl((1-\theta)\dbQ+\theta\dbQ'\bigr)
\le
(1-\theta)F(\dbQ)+\theta F(\dbQ').
\]
\item
The functional $F$ is linearly differentiable on $\cP(\Omega)$, in the
following sense: there exists a
$\mathcal B(\cP(\Omega))\otimes\mathcal B(\Omega)$-measurable function
\[
    \delta F:\cP(\Omega)\times\Omega\longrightarrow\dbR
\]
such that, for every $\dbQ,\dbQ'\in\cP(\Omega)$,
\[
    \lim_{\varepsilon\downarrow 0}
    \frac{
        F\bigl((1-\varepsilon)\dbQ+\varepsilon\dbQ'\bigr)-F(\dbQ)
    }{\varepsilon}
    =
    \int_\Omega
        \delta F(\dbQ,x)\,
        (\dbQ'-\dbQ)(\dd x).
\]
Moreover, there exists a constant $C_F>0$, independent of $\dbQ$, such that
\[
    \left\|\delta F(\dbQ,\cdot)\right\|_\infty
    \le C_F,
    \qquad
    \forall\,\dbQ\in\cP(\Omega).
\]
\item Its linear derivative $\delta F(\dbQ,x)$ is uniformly
continuous with respect to $\dbQ$ endowed with the weak topology.
More precisely, for every sequence
$(\dbQ^n)_{n\ge 1}\subset\cP(\Omega)$ converging weakly to some
$\dbQ\in\cP(\Omega)$,
\[
    \sup_{x\in\Omega}
    \left|
        \delta F(\dbQ^n,x)-\delta F(\dbQ,x)
    \right|
    \longrightarrow 0.
\]

\item
Its linear derivative $\delta F(\dbQ,x)$ is Lipschitz continuous with
respect to $\dbQ$ in the total variation topology. Namely, there exists
a constant $L_F>0$ such that, for all
$\dbQ,\dbQ'\in\cP(\Omega)$,
\[
    \sup_{x\in\Omega}
    \left|
        \delta F(\dbQ,x)-\delta F(\dbQ',x)
    \right|
    \le
    L_F\left\|\dbQ-\dbQ'\right\|_{\mathrm{TV}}.
\]
\end{enumerate}
\end{assu}
The bounded derivative gives an admissible BSDE terminal condition, while convexity turns the law-dependent BSDE fixed point into a global minimiser.

\begin{assu}[Initial state conditions]
\label{ass-ini}
Let $\beta^0$ be the initial control and let 
\[
\frac{\dd \dbQ^0}{\dd \dbP}:=D(\beta^0)
\]
be the initial probability measure. We assume that $(\dbQ^0,\beta^0)\in\cA_\ell\cap\cA_\psi$ and,
\[
\max_{\varphi\in\{\ell,\psi\}}\Bigl\{
|\beta^0|_{\varphi-\BMO},
\|\nabla\varphi(0)\|_\infty,
\|\varphi(0)\|_\infty
\Bigr\}
\le C_0.
\]
\end{assu}
The initial conditional bounds start the iteration; they are propagated in the Newton and general-flow estimates.

\begin{assu}[Linear potential energy]
\label{ass-pe-linear}
The functional $F$ is linear in $\dbQ$.
Consequently, its linear derivative
$\delta F(\dbQ,x)$ is independent of $\dbQ$.
We write \(f:=\delta F,\) and assume that there exists a constant $C_F>0$ such that
\[
\|f\|_\infty\le C_F.
\]
\end{assu}
The affine terminal potential is the extra ingredient behind the linear entropy-coordinate evolution. It is not imposed on the general natural-gradient result.

\begin{assu}[Quadratic growth of the running cost and Bregman potential]
\label{ass-rc-quadratic}
For every $(t,\omega)\in[0,T]\times\Omega$, the mappings \(b\mapsto\ell(t,\omega,b) \), \( b\mapsto\psi(t,\omega,b)\) are twice differentiable. Moreover, there exist constants
$C_\ell,C_\psi>0$, independent of $(t,\omega)$, such that
\[
\nabla^2\ell(t,\omega,b)
\le
C_\ell I_d,
\qquad
\nabla^2\psi(t,\omega,b)
\le
C_\psi I_d,
\]
for every \((t,\omega,b) \in [0,T]\times\Omega\times\dbR^d,\) where the inequalities are understood in the sense of
symmetric matrices.
\end{assu}
The upper bounds are used in the local BMO flow and consistency proof. The Newton flow is treated separately and does not require this additional assumption.

\begin{assu}[Initial condition for the particle system]
\label{ass-particle-initial}

For every fixed $N\geq1$, let \((\dbQ_0^N,\boldsymbol\beta_0^N) \in \mathcal A_\ell^N \cap \mathcal A_\psi^N\) be the initial state. Moreover, there exists a finite constant $C_{0,N}$ such that
\begin{align*}
    \max\Bigl\{
        &
        |\boldsymbol\beta_0^N|_{\ell^N-\BMO},
        \;
        |\boldsymbol\beta_0^N|_{\psi^N-\BMO},
        \;
        |\mathbf 0|_{\ell^N-\BMO},
        \;
        |\mathbf 0|_{\psi^N-\BMO},
        \notag\\
        &
        \|
            \nabla
            \ell^N(\mathbf 0)
        \|_\infty,
        \;
        \|
            \nabla
            \psi^N(\mathbf 0)
        \|_\infty,
        \;
        \|
            \ell^N(\mathbf 0)
        \|_\infty,
        \;
        \|
            \psi^N(\mathbf 0)
        \|_\infty
    \Bigr\}
    \leq
    C_{0,N}.
    \label{eq-particle-initial-bound}
\end{align*}
No uniformity of $C_{0,N}$ with respect to $N$ is required.

In particular, if
\[
    \dbQ_0^N
    =
    \dbQ_0^{\otimes N},
\]
and
\[
    \boldsymbol\beta_0^N
    =
    \bigl(
        \beta_0(X^1),\ldots,\beta_0(X^N)
    \bigr),
\]
where $(\dbQ_0,\beta_0)$ satisfies
Assumption~\ref{ass-ini}, then
Assumption~\ref{ass-particle-initial} holds for every fixed $N$.
\end{assu}
This is a fixed-$N$ admissibility condition. The particle-flow proof allows its local constants to depend on $N$.

\begin{assu}[Whole-space Markov Hamiltonian]\label{ass-markov}
For Section~\ref{sec-main-regularity} only, take $\psi=\ell$,
$\ell(t,\omega,b)=\ell(t,\omega_t,b)$, and
$F(\dbQ)=\langle g(X_T),\dbQ\rangle$.
The standing lower bound $\nabla_b^2\ell\ge c_\ell I_d$ remains in force.
Define the convex Hamiltonian by
\[
 H(t,x,p):=\sup_{b\in\R^d}\{-b\cdot p-\ell(t,x,b)\}
          =-\mathcal H_\ell(t,x,p).
\]
Assume the following bounds uniformly in $t\in[0,T]$ and $x\in\R^d$.
\begin{enumerate}[label=(\roman*),leftmargin=2em]
\item $g$ is bounded, Lipschitz and semiconcave. The Hamiltonian is
continuous in $t$ and $C^2$ in $(x,p)$, with the derivatives locally
uniformly continuous in $t$. For some $c,C>0$,
\[
 -C+c|p|^2\le H(t,x,p)\le C+c^{-1}|p|^2,
 \qquad |D_xH(t,x,p)|\le C(1+|p|).
\]
\item For every $R<\infty$ there are $C_R<\infty$ and
$0<\theta_R\le\Theta_R<\infty$ such that, when $|p|\le R$,
\[
 |D^2_{xx}H|+|D^2_{xp}H|\le C_R,
 \qquad \theta_R I_d\le D^2_{pp}H\le\Theta_R I_d.
\]
\item The initial feedback is specified by a bounded mirror
coordinate $q^0\in C([0,T];C_b^2(\R^d;\R^d))$:
\[
 \beta^0(t,x)=(\nabla_b\ell(t,x,\cdot))^{-1}(q^0(t,x))
            =-D_pH(t,x,-q^0(t,x)).
\]
Thus $\beta^0$ need not vanish. The inverse and the bounds in
(i)--(ii) imply the initial admissibility hypotheses above.
\item The higher-order conclusion of Remark~\ref{rem-markov-smooth}
at order $m$ additionally assumes bounded spatial derivatives of $g$
and $q^0$ up to order $m+3$, and bounded $(x,p)$ derivatives of $H$
up to order $m+3$ on every set $|p|\le R$, uniformly in $(t,x)$.
These derivatives are continuous, with the same local time-continuity
as in (i). This item is not needed for the low-order theorem.
\end{enumerate}
\end{assu}

The {Markov proof} first derives a uniform feedback bound from these assumptions, before using a coupling constant that depends on that bound.

\paragraph{Admissibility and generalised entropy.}
We work on the canonical Wiener space described in Section~\ref{sec-notation}.
A pair $(\dbQ,\beta)$ belongs to $\overline{\cA}$ when $\dbQ\sim\dbP$,
$\beta$ is progressive, $W^\beta=X-\int_0^\cdot\beta_t\,\dd t$ is a
$(\dbQ,\dbF)$-Brownian motion, and
\[
 H_\ell(\dbQ):=\dbE^{\dbQ}\!\int_0^T\ell(t,X,\beta_t)\,\dd t<\infty.
\]
The canonical drift is unique up to $\dd t\otimes\dd\dbQ$-null sets.
Conversely, a progressive $\beta$ whose stochastic exponential
$D(\beta)$ has expectation one defines $\dbQ^\beta=D(\beta)\dbP$;
this pair is in $\overline{\cA}$ provided its $\ell$-cost is finite.
We write $\ell(\beta_t)$ for $\ell(t,X,\beta_t)$, and use the same
convention for $\psi$ and their drift derivatives.

For a progressive integrand $\alpha$ and $\varphi\in\{\ell,\psi\}$, set
\begin{align*}
 \|\alpha\|_{\BMO(\dbQ)}
  &:=\sup_{\tau\in\cT}\left\|
       \dbE^{\dbQ}\!\left[\int_\tau^T|\alpha_t|^2\,\dd t
                         \,\middle|\,\cF_\tau\right]^{1/2}\right\|_\infty,\\
 |\alpha|_{\varphi-\BMO(\dbQ)}
  &:=\|\alpha\|_{\BMO(\dbQ)}
      +\|\nabla\varphi(\alpha)\|_{\BMO(\dbQ)}
      +\|\sqrt{\varphi^+(\alpha)}\|_{\BMO(\dbQ)}.
\end{align*}
The subscript $\dbQ$ is omitted for $\dbQ=\dbP$.
We write $\cA=\{(\dbQ,\beta)\in\overline{\cA}:\beta\in\BMO\}$ and
$\cA_\varphi=\{(\dbQ,\beta)\in\cA:|\beta|_{\varphi-\BMO}<\infty\}$.
Define
\[
 H_\varphi(\dbQ\mid\dbQ')
 :=\dbE^{\dbQ}\!\int_0^T
  [\varphi(\beta_t)-\varphi(\beta'_t)
          -\nabla\varphi(\beta'_t)\cdot(\beta_t-\beta'_t)]\,\dd t.
\]
On the regular admissible class this is the Bregman divergence of
$H_\varphi$, with linear derivative
\begin{equation*}\label{eq-main-entropy-derivative}
 \delta H_\varphi(\dbQ,X)
 =\int_0^T\varphi(\beta_t)\,\dd t
       +\int_0^T\nabla\varphi(\beta_t)\cdot\dd W_t^\beta.
\end{equation*}
The precise derivative statement is Lemma~\ref{lem-linear-derivative}.
The quadratic drift-entropy identity and strong convexity give
\begin{equation}\label{eq-relative-ell-psi}
 H_\ell(\dbQ\mid\dbQ')\ge\lambda_{\ell|\psi}H_\psi(\dbQ\mid\dbQ'),
 \qquad
 H_\varphi(\dbQ\mid\dbQ')\ge c_\varphi H(\dbQ\mid\dbQ'),
\end{equation}
whenever the changes of measure and the displayed quantities are
well defined; in particular this holds for the BMO laws used below.
Here $H(\dbQ\mid\dbQ')=\dbE^{\dbQ}\log(\dd\dbQ/\dd\dbQ')$.

\begin{rem}[Convention for linear derivatives]
\label{rem-linear-derivative-convention}
Linear derivatives on probability laws are defined modulo constants.
Every pairing against $\dbQ-\dbQ'$ is independent of that choice.
We keep the bounded representative of $\delta F$ specified in the
assumptions. In the linear case a harmless additive constant in $F$
is suppressed, so that $F(\dbQ)=\langle f,\dbQ\rangle$.
\end{rem}

The derivative and its invariance under additive constants are justified in Lemma~\ref{lem-linear-derivative}.

\paragraph{The common Bregman proximal construction.}
The implicit construction below is common to both flows. Its time step is a learning-time step.

\begin{prop}[Well-posedness and discrete convergence of the Bregman proximal scheme]
\label{prop-jko}\label{prop-bregman-proximal}
\label{prop-JKO-minimizer}\label{prop-proximal-minimizer}
\label{prop-JKO-step-wellposedness}\label{prop-proximal-step-wellposedness}
\label{prop-JKO-exponential-convergence}\label{prop-proximal-exponential-convergence}
Suppose that Assumptions~\ref{ass-rc},
\ref{ass-pe} and~\ref{ass-ini} hold.
Then there exists a unique pair \((\dbQ^*,\beta^*)\in\cA_\ell\cap\cA_\psi\) such that $\dbQ^*$ minimises the functional \(J(\dbQ)=H_\ell(\dbQ)+F(\dbQ).\) Moreover, $\beta^*\in\BMO$ and
$\dd\dbQ^*/\dd\dbP=D(\beta^*)$.

Let $h_0 :=\frac{1}{2L_\ell}$. For every $0<h<h_0$, there exists a unique sequence \(\bigl(\dbQ^{h,n},\beta^{h,n}\bigr)_{n\ge0} \subset \cA_\ell\cap\cA_\psi\) such that \((\dbQ^{h,0},\beta^{h,0}) = (\dbQ^0,\beta^0),\) and, for every $n\ge0$,
\[
(\dbQ^{h,n+1},\beta^{h,n+1})
\in
\operatorname*{argmin}_{(\dbQ,\beta)\in\overline{\cA}}
\left\{
J(\dbQ)
+
\frac1h
H_\psi(\dbQ\,|\,\dbQ^{h,n})
\right\}.
\]
We refer to
$\bigl(\dbQ^{h,n},\beta^{h,n}\bigr)_{n\ge0}$
as the Bregman proximal sequence with time-step size $h$
starting from $(\dbQ^0,\beta^0)$. Moreover,
\[
H_\psi(\dbQ^*\,|\,\dbQ^{h,n})
\le
\frac{
H_\psi(\dbQ^*\,|\,\dbQ^0)
}{
\bigl(1+\lambda_{\ell|\psi}h\bigr)^n
} ,\qquad n\ge0.
\]
In addition, for every $n\ge1$,
\[
0
\le
J(\dbQ^{h,n})-J(\dbQ^*)
\le
\frac{
\lambda_{\ell|\psi}
H_\psi(\dbQ^*\,|\,\dbQ^0)
}{
\bigl(1+\lambda_{\ell|\psi}h\bigr)^n-1
},\qquad
H_\psi(\dbQ^{h,n}\mid \dbQ^*)
\le
\frac{
H_\psi(\dbQ^*\,|\,\dbQ^0)
}{
\bigl(1+\lambda_{\ell|\psi}h\bigr)^n-1
}.
\]
\end{prop}
\proofguide{The proof can be found in Section~\ref{sec-common-proximal-proof}, in the minimiser construction, proximal-step admissibility argument and discrete decay proof. The penalty controls the change of policy without linearising the objective.}

\subsection{Continuous mirror flows: structure and convergence}
\label{sec-main-flows}\label{sec-main-newton}\label{sec-main-natural}
\readingguide{The next two theorems construct the limiting flow by different
mechanisms. The Newton-flow theorem retains the exact entropy-coordinate
representation and its stronger global estimates. The general theorem
uses nonlinear BMO/BSDE stability. Their common Bregman and objective
bounds are stated only in the corollary that follows.}
For both results, $(\dbQ^{h,n},\beta^{h,n})$ denotes the sequence of
Proposition~\ref{prop-bregman-proximal}, $h_0 :=(2L_\ell)^{-1}$, and
\[
 \dbQ^{h,s}:=\dbQ^{h,\lfloor s/h\rfloor},\qquad
 \beta^{h,s}:=\beta^{h,\lfloor s/h\rfloor},\qquad s\ge0.
\]
The minimiser of $J$ is denoted by $(\dbQ^*,\beta^*)$.

\paragraph{The Newton flow.}
\begin{thm}[Exact representation and strong estimates for the Newton flow]
\label{thm-newton-structure}\label{prop-static-newton}\label{thm-linear}
\label{prop-linear-bounded}\label{prop-existence}\label{prop-linear-uniqueness}
Suppose that Assumptions~\ref{ass-rc}, \ref{ass-pe}, \ref{ass-ini}
and~\ref{ass-pe-linear} hold, and that $\psi=\ell$. Write
$F(\dbQ)=\langle f,\dbQ\rangle$.

\par\noindent\textnormal{\textbf{(i) Static construction and stability.}}
For $0\le a<1$, defining $d_\ell$ through \eqref{eq:introBreg}, set
\[
 g_a(t,X,b) :=(1-a)\ell(t,X,b) + a d_\ell(t,X,b,\beta^0_t).
\]
There is a unique optimal law $\widehat{\dbQ}^a$ for running cost $g_a$
and terminal cost $(1-a)f$, with drift $\widehat\beta^a$. Set
$(\widehat{\dbQ}^{1},\widehat\beta^{1})=(\dbQ^0,\beta^0)$.
These pairs belong to $\cA_\ell$, and $\widehat{\dbQ}^{0}=\dbQ^*$.
There is a constant $C$, depending only on the standing data, the
initial $\ell$-BMO bound and $T$, such that, for $a,a'\in[0,1]$,
\begin{equation}\label{eq-static-stability}
 \begin{split}
 &\|\widehat\beta^{a}-\widehat\beta^{a'}\|_{\BMO}
 +\|\nabla\ell(\widehat\beta^{a})-\nabla\ell(\widehat\beta^{a'})\|_{\BMO}
 +\|\widehat{\dbQ}^{a}-\widehat{\dbQ}^{a'}\|_{\TV}
 \le C|a-a'|.
 \end{split}
\end{equation}

\par\noindent\textnormal{\textbf{(ii) Identification and affine dual dynamics.}}
Set $(\dbQ^s,\beta^s):=(\widehat{\dbQ}^{e^{-s}},\widehat\beta^{e^{-s}})$.
The discrete and continuous curves have the exact representations
\begin{equation}\label{eq-static-identification}
 \dbQ^{h,n}=\widehat{\dbQ}^{(1+h)^{-n}},\qquad
 \dbQ^s=\widehat{\dbQ}^{e^{-s}}.
\end{equation}
With $\overline\rho=\delta H_\ell(\dbQ)-\dbE[\delta H_\ell(\dbQ)]$
and $\overline f=f-\dbE[f]$, the corresponding centred coordinates satisfy
\begin{equation}\label{eq-main-dual-affine}
 \overline\rho^{h,n}=-\overline f+(1+h)^{-n}(\overline\rho^0+\overline f),
 \qquad
 \overline\rho^s=-\overline f+e^{-s}(\overline\rho^0+\overline f).
\end{equation}
The identities hold in every finite $L^q(\dbP)$. In particular,
$\overline\rho\in C^1([0,\infty);L^q(\dbP))$ is the unique solution of
$\partial_s\overline\rho^s=-(\overline\rho^s+\overline f)$ with its prescribed
initial value, and $\overline\rho^{h,\cdot}\to\overline\rho$ uniformly in
learning time in each such space.

\par\noindent\textnormal{\textbf{(iii) Uniform-in-learning-time estimates.}}
For $0<h<h_0$,
\begin{equation}\label{eq-newton-strong-rate}
 \sup_{s\ge0}\bigl(\|\beta^{h,s}-\beta^s\|_{\BMO}
                  +\|\dbQ^{h,s}-\dbQ^s\|_{\TV}\bigr)\le Ch.
\end{equation}
For every $s\ge0$,
\begin{equation}\label{eq-newton-strong-equilibrium}
 \|\beta^s-\beta^*\|_{\BMO}
 +\|\nabla\ell(\beta^s)-\nabla\ell(\beta^*)\|_{\BMO}
 +\|\dbQ^s-\dbQ^*\|_{\TV}\le Ce^{-s},
 \qquad
 0\le J(\dbQ^s)-J(\dbQ^*)\le Ce^{-2s}.
\end{equation}
No upper Hessian bound on $\ell$ is assumed.
\end{thm}
\proofguide{The proof can be found in Section~\ref{sec-static-proof}. The BMO bounds,
density convergence and
learning-time continuity are consequences
of this representation, recorded in the proof rather than repeated here.
The strong estimates use initial conditional integrability;
Corollary~\ref{cor-common-flow-decay} gives initial-divergence bounds.}

\paragraph{The general natural-gradient construction.} Without the joint conditions $\psi=\ell$ and linear $F$,
the affine entropy equation is no longer available in general.
The mirror coordinate determines the law and the terminal condition of
policy evaluation; the evaluation integrand then determines its velocity.
The next theorem controls this nonlinear dependence on bounded learning
intervals.

\begin{thm}[Well-posedness and approximation of the general natural-gradient flow]
\label{thm-natural-structure}\label{thm-MF}
\label{prop-MF-existence}

Suppose that Assumptions~\ref{ass-rc},
\ref{ass-pe}, \ref{ass-ini},
and~\ref{ass-rc-quadratic} hold.

\noindent
\textnormal{\textbf{(i) Finite-time BMO estimates.}}
For every $S>0$, there exist $h_S>0$ and $C_S>0$ such that,
for $0<h\le h_S$ and $nh\le S$,
\[
 \|\beta^{h,n}\|_{\BMO(\dbQ^{h,n})}\le C_S,
 \qquad
 \|\beta^{h,n+1}-\beta^{h,n}\|_{\BMO}\le C_Sh.
\]
The constant $C_S$ depends only on $S$ and the constants in the stated
assumptions.

\noindent
\textnormal{\textbf{(ii) Continuous-time limit and convergence rate.}}
There exists a unique curve
\[
\bigl(\dbQ^s,\beta^s\bigr)_{s\ge0}
\subset
\cA_\ell\cap\cA_\psi
\]
such that, after decreasing $h_S$ and enlarging $C_S$ if necessary,
\[
 \sup_{s\in[0,S]}
 \bigl(\|\dbQ^{h,s}-\dbQ^s\|_{\TV}
       +\|\beta^{h,s}-\beta^s\|_{\BMO}\bigr)
 \le C_Sh,\qquad 0<h\le h_S.
\]
The drift curve is locally Lipschitz in $\BMO$, and the law curve is
locally Lipschitz in total variation.

\noindent
\textnormal{\textbf{(iii) Characterization as a BMO mirror gradient flow.}}
For every $s\ge0$, let $(Y^s,Z^s)$ denote the solution of
\[
\begin{cases}
\displaystyle
\dd Y_t^s
=
-\left(\beta_t^s\cdot Z_t^s
+
\ell(\beta_t^s)
\right)\dd t
+
Z_t^s\,\dd X_t,
\\[1ex]
\displaystyle
Y_T^s
=
\delta F(\dbQ^s,X),
\end{cases}
\]
with $\dd\dbQ^s/\dd\dbP=D(\beta^s)$. Set
$q^s:=\nabla\psi(\beta^s)$, so that
$\beta^s=(\nabla\psi)^{-1}(q^s)$, and write $Z(q^s)=Z^s$.
The limiting dynamics is the unique BMO-continuous solution of
\[
 q^S-q^0
 =-\int_0^S
 \left[\nabla\ell\bigl((\nabla\psi)^{-1}(q^s)\bigr)+Z(q^s)\right]\dd s,
 \qquad S\ge0.
\]
This is the natural-gradient equation in the mirror coordinate $q$,
written in integral form.
\end{thm}
\proofguide{The proof can be found in Section~\ref{sec-natural-proofs}: it combines the BMO bootstrap,
increment estimate and
construction and consistency argument.
Corollary~\ref{cor-common-flow-decay} gives the convergence bounds for both flows.}

\paragraph{Common convergence consequences.}
Whenever both sets of hypotheses hold, the constructions
coincide, since they are limits of the same proximal sequence. The following
bounds use their common variational structure; the faster Newton bounds
in \eqref{eq-newton-strong-equilibrium} remain specific to the Newton flow.

\begin{cor}[Initial-divergence convergence bounds]
\label{cor-common-flow-decay}\label{thm-linear-continuous-exponential}
\label{thm-MF-exponential}
Under the assumptions of either Theorem~\ref{thm-newton-structure} or
Theorem~\ref{thm-natural-structure}, let $(\dbQ^s,\beta^s)$ be the
corresponding flow.
Under the Newton-flow assumptions, take $\psi=\ell$ and $\lambda_{\ell|\psi}=1$.
For every $s\ge0$,
\begin{equation}\label{eq-common-flow-bregman-decay}
 H_\psi(\dbQ^*\mid\dbQ^s)
 \le e^{-\lambda_{\ell|\psi}s} H_\psi(\dbQ^*\mid\dbQ^0).
\end{equation}
The map $s\mapsto J(\dbQ^s)$ is non-increasing and, for every $s>0$,
\begin{equation}\label{eq-common-flow-value-decay}
 0\le J(\dbQ^s)-J(\dbQ^*)
 \le\frac{\lambda_{\ell|\psi}H_\psi(\dbQ^*\mid\dbQ^0)}
              {e^{\lambda_{\ell|\psi}s}-1},\qquad
 H_\psi(\dbQ^s\mid \dbQ^*)
 \le\frac{H_\psi(\dbQ^*\mid\dbQ^0)}
              {e^{\lambda_{\ell|\psi}s}-1}
\end{equation}
\end{cor}
\proofguide{The proof can be found in Section~\ref{sec-common-flow-decay-proof}. The proof first shows,
separately for the Newton flow and the general flow, that the entropy,
Bregman divergence and objective appearing in the discrete estimates
converge to their continuous-flow counterparts. The
Newton limit passage uses the affine
entropy coordinate and no upper Hessian bound, whereas the
general-flow limit passage uses
continuity on a BMO ball. The final step of that proof then passes the
discrete contraction and objective estimates of
Proposition~\ref{prop-bregman-proximal} directly to continuous learning time.
Lemma~\ref{lem-forward-coercivity} gives the forward estimate from the limiting objective gap.}

\begin{rem*}
For $s\ge\log2/\lambda_{\ell|\psi}$,
\eqref{eq-common-flow-value-decay} gives
$J(\dbQ^s)-J(\dbQ^*)\le
2\lambda_{\ell|\psi}e^{-\lambda_{\ell|\psi}s}H_\psi(\dbQ^*\mid\dbQ^0)$.
Pinsker's inequality and \eqref{eq-relative-ell-psi} also imply
\[
 \|\dbQ^s-\dbQ^*\|_{\TV}
 \le\sqrt{2H_\psi(\dbQ^*\mid\dbQ^0)/c_\psi}\,
                 e^{-\lambda_{\ell|\psi}s/2}.
\]
Unlike the stronger Newton-flow estimates in
\eqref{eq-newton-strong-equilibrium}, these bounds depend on the initial
policy only through $H_\psi(\dbQ^*\mid\dbQ^0)$.
\end{rem*}

\subsection{Particle approximation and defective convergence}\label{sec-particle-long-time}

We now study the finite-particle approximation of the mean-field
gradient flow. Throughout this section, we assume that
Assumptions~\ref{ass-rc}, \ref{ass-pe},
\ref{ass-rc-quadratic} hold.

Fix $N\geq1$. Let
\[
    \Omega^N
    :=
    \underbrace{\Omega\times\cdots\times\Omega}_{N\text{ times}},
    \qquad
    \mathbf X
    :=
    (X^1,\ldots,X^N),
\]
and let $\dbP^N := \dbP^{\otimes N}$.
We equip $\Omega^N$ with the canonical filtration
\[
    \mathcal F_t^N
    :=
    \sigma
    \left(
        X_s^i:
        0\leq s\leq t,\;
        1\leq i\leq N
    \right),
    \qquad
    t\in[0,T],
\]
augmented by the $\dbP^N$-null sets, and denote the resulting
filtration by $\mathbb F^N
    :=
    (\mathcal F_t^N)_{0\leq t\leq T}$.
We denote by $\mathcal T^N$ the collection of all
$\mathbb F^N$-stopping times with values in $[0,T]$.
For \(\mathbf b=(b^1,\ldots,b^N)\in(\dbR^d)^N,\) define the tensorized running cost and Bregman potential by
\begin{align}
    \ell^N(t,\boldsymbol\omega,\mathbf b)
    &:=
    \sum_{i=1}^N
    \ell(t,\omega^i,b^i),
    \label{eq-particle-ell-N}
    \\
    \psi^N(t,\boldsymbol\omega,\mathbf b)
    &:=
    \sum_{i=1}^N
    \psi(t,\omega^i,b^i),
    \label{eq-particle-psi-N}
\end{align}
where $\boldsymbol\omega
    =
    (\omega^1,\ldots,\omega^N)
    \in\Omega^N$. By Assumptions~\ref{ass-rc} and~\ref{ass-rc-quadratic},
the tensorized functions satisfy, for every
$(t,\boldsymbol\omega,\mathbf b)$,
\begin{align}
    c_\ell I_{Nd}
    &\leq
    \nabla_{\mathbf b}^2
    \ell^N(t,\boldsymbol\omega,\mathbf b)
    \leq
    C_\ell I_{Nd},
    \label{eq-particle-ell-Hessian}
    \\
c_\psi I_{Nd}
    &\leq
    \nabla_{\mathbf b}^2
    \psi^N(t,\boldsymbol\omega,\mathbf b)
    \leq
    C_\psi I_{Nd},
    \label{eq-particle-psi-Hessian}
\end{align}
and
\begin{equation}
    \nabla_{\mathbf b}^2
    \ell^N(t,\boldsymbol\omega,\mathbf b)
    \geq
    \lambda_{\ell|\psi}
    \nabla_{\mathbf b}^2
    \psi^N(t,\boldsymbol\omega,\mathbf b).
    \label{eq-particle-relative-convexity}
\end{equation}
Thus the constants
$c_\ell,c_\psi,C_\ell,C_\psi$ and
$\lambda_{\ell|\psi}$ are preserved under tensorization.
An $N$-particle control is an
$\mathbb F^N$-progressively measurable process
\[
    \boldsymbol\beta
    =
    (\beta^1,\ldots,\beta^N):
    [0,T]\times\Omega^N
    \longrightarrow
    (\dbR^d)^N.
\]
Each component $\beta^i$ is allowed to depend on the whole particle
configuration $\mathbf X$.
We define the $N$-particle admissible set
$\overline{\mathcal A}^N$ as the collection of pairs
$(\dbQ^N,\boldsymbol\beta)$ such that \(\dbQ^N\in\mathcal P(\Omega^N) \), \( \dbQ^N\sim\dbP^N,\) $\boldsymbol\beta$ is $\mathbb F^N$-progressively measurable,
\[
    \mathbf W_t^{\boldsymbol\beta}
    :=
    \mathbf X_t-\mathbf X_0
    -
    \int_0^t
    \boldsymbol\beta_s\,\dd s
\]
is an $\dbR^{Nd}$-valued
$(\dbQ^N,\mathbb F^N)$-Brownian motion, and
\begin{equation*}
    H_\ell^N(\dbQ^N)
    :=
    \dbE^{\dbQ^N}
    \left[
        \int_0^T
        \ell^N
        (t,\mathbf X,\boldsymbol\beta_t)
        \,\dd t
    \right]
    <+\infty.
    \label{eq-particle-generalized-entropy}
\end{equation*}
Equivalently, the density of $\dbQ^N$ with respect to $\dbP^N$
is given by
\begin{equation*}
    \frac{\dd\dbQ^N}{\dd\dbP^N}
    =
    D^N(\boldsymbol\beta)
    :=
    \exp
    \left(
        \sum_{i=1}^N
        \int_0^T
        \beta_t^i\cdot\dd X_t^i
        -
        \frac12
        \int_0^T
        \sum_{i=1}^N
        |\beta_t^i|^2
        \,\dd t
    \right).
    \label{eq-particle-density}
\end{equation*}
For an $\mathbb F^N$-progressively measurable process
$\boldsymbol\alpha$ and a probability measure $\dbQ^N$ on $\Omega^N$,
define
\begin{equation*}
    \|\boldsymbol\alpha\|_{\BMO(\dbQ^N)}
    :=
    \sup_{\tau\in\mathcal T^N}
    \left\|
        \dbE^{\dbQ^N}
        \left[
            \int_\tau^T
            |\boldsymbol\alpha_t|^2\,\dd t
            \,\Big|\,
            \mathcal F_\tau^N
        \right]^{1/2}
    \right\|_\infty.
    \label{eq-particle-BMO}
\end{equation*}
When $\dbQ^N=\dbP^N$, we simply write
$\|\cdot\|_{\BMO}$.
For $\varphi\in\{\ell,\psi\}$,
we write $\varphi^N
    :=
    \sum_{i=1}^N\varphi^i$
as in
\eqref{eq-particle-ell-N}--\eqref{eq-particle-psi-N},
and define
\[
    |\boldsymbol\alpha|_{\varphi^N-\BMO(\dbQ^N)}
    :=
    \|\boldsymbol\alpha\|_{\BMO(\dbQ^N)}
    +
    \left\|
        \nabla_{\mathbf b}
        \varphi^N(\boldsymbol\alpha)
    \right\|_{\BMO(\dbQ^N)}
   +
    \left\|
        \sqrt{
            \bigl(
                \varphi^N(\boldsymbol\alpha)
            \bigr)^+
        }
    \right\|_{\BMO(\dbQ^N)}.
    \label{eq-particle-varphi-BMO}
\]
We then define
\begin{multline*}
\mathcal A^N
    :=
    \{
        (\dbQ^N,\boldsymbol\beta)
        \in\overline{\mathcal A}^N:
        \|\boldsymbol\beta\|_{\BMO}<+\infty
    \}, \quad
    \mathcal A_\ell^N
    :=
    \{
        (\dbQ^N,\boldsymbol\beta)
        \in\mathcal A^N:
|\boldsymbol\beta|_{\ell^N-\BMO}<+\infty
    \},
    \\
    \mathcal A_\psi^N
    := \{       (\dbQ^N,\boldsymbol\beta)
        \in\mathcal A^N:       |\boldsymbol\beta|_{\psi^N-\BMO}<+\infty
    \}.
    \end{multline*}
Let $(\dbP^N,\boldsymbol\alpha),
    (\dbQ^N,\boldsymbol\beta)
    \in\overline{\mathcal A}^N$.
For $\varphi\in\{\ell,\psi\}$, define the relative generalized
$\varphi$-entropy by
\[ H_\varphi^N(\dbP^N\mid\dbQ^N)
    :=
    \dbE^{\dbP^N}
    \Bigg[
        \int_0^T
        \Big(
            \varphi^N(t,\mathbf X,\boldsymbol\alpha_t)
            -
            \varphi^N(t,\mathbf X,\boldsymbol\beta_t)
            -
            \nabla_{\mathbf b}
            \varphi^N(t,\mathbf X,\boldsymbol\beta_t)
            \cdot
            (\boldsymbol\alpha_t-\boldsymbol\beta_t)
        \Big)
        \,\dd t
    \Bigg].
    \label{eq-particle-relative-varphi-entropy}
\]
Equivalently,
\[
    H_\varphi^N(\dbP^N\mid\dbQ^N)
    =
    \dbE^{\dbP^N}
    \Bigg[
        \int_0^T
        \sum_{i=1}^N
        \Big(
            \varphi(t,X^i,\alpha_t^i)
            -
            \varphi(t,X^i,\beta_t^i)
            -
\nabla\varphi(t,X^i,\beta_t^i)
            \cdot
            (\alpha_t^i-\beta_t^i)
        \Big)
        \,\dd t
    \Bigg].
\]
By \eqref{eq-particle-relative-convexity},
\begin{equation*}
    H_\ell^N(\dbP^N\mid\dbQ^N)
    \geq
    \lambda_{\ell|\psi}
    H_\psi^N(\dbP^N\mid\dbQ^N).
    \label{eq-particle-relative-ell-psi}
\end{equation*}
For $\mathbf x=(x^1,\ldots,x^N)\in\Omega^N$,
let
\begin{equation*}
    \mu^N(\mathbf x)
    :=
    \frac1N
    \sum_{i=1}^N
    \delta_{x^i}
    \in\mathcal P(\Omega)
    \label{eq-particle-empirical-measure}
\end{equation*}
be the empirical measure of the particle configuration, and set
\begin{equation*}
    V_N(\mathbf x)
    :=
    N F\bigl(\mu^N(\mathbf x)\bigr).
    \label{eq-particle-potential-function}
\end{equation*}
The corresponding functional on $\mathcal P(\Omega^N)$ is
\begin{equation*}
    \mathcal F_N(\dbQ^N)
    :=
    \langle V_N,\dbQ^N\rangle
    =
    N
    \dbE^{\dbQ^N}
    \left[
        F(\mu^N)
    \right].
    \label{eq-particle-potential}
\end{equation*}
For every fixed $N$, the functional $\mathcal F_N$ is linear in
$\dbQ^N$, and its linear derivative is
\[
    \delta\mathcal F_N(\dbQ^N,\mathbf x)
    =
    V_N(\mathbf x).
\]
We define the $N$-particle objective functional by
\begin{equation*}
    J_N(\dbQ^N)
    :=
    H_\ell^N(\dbQ^N)
    +
    \mathcal F_N(\dbQ^N).
    \label{eq-particle-energy}
\end{equation*}
For $0<h<h_0$, define recursively
\begin{equation}
    (\dbQ_{h,n+1}^N,\boldsymbol\beta_{h,n+1}^N)
    :=
    \operatorname*{argmin}_{
        (\dbQ^N,\boldsymbol\beta)\in
        \overline{\mathcal A}^N
    }
    \left\{
        J_N(\dbQ^N)
        +
        \frac1h
        H_\psi^N
        (\dbQ^N\mid\dbQ_{h,n}^N)
    \right\},
    \qquad
    n\geq0,
    \label{eq-particle-JKO}
\end{equation}
with $(\dbQ_{h,0}^N,\boldsymbol\beta_{h,0}^N)
    = (\dbQ_0^N,\boldsymbol\beta_0^N).$
For $s\geq0$, define the piecewise-constant interpolation
\begin{equation*}
    \dbQ_{h,s}^N
    :=
    \dbQ_{h,\lfloor s/h\rfloor}^N,
    \qquad
    \boldsymbol\beta_{h,s}^N
    :=
    \boldsymbol\beta_{h,\lfloor s/h\rfloor}^N.
    \label{eq-particle-interpolation}
\end{equation*}

\begin{thm}[Finite-particle Bregman gradient flow]
\label{thm-particle-flow}

Suppose that Assumptions~\ref{ass-rc},
\ref{ass-pe}, \ref{ass-rc-quadratic},
and~\ref{ass-particle-initial} hold.
Then, for every fixed $N\geq1$, the functional $J_N$ admits a unique
minimizer $(\dbQ_N^*,\boldsymbol\beta_N^*)
    \in
    \overline{\mathcal A}^N$.
Moreover, for every $0<h<h_0$, the Bregman proximal scheme
\eqref{eq-particle-JKO} admits a unique solution at every step and
defines a sequence
\[
    \bigl(
        \dbQ_{h,n}^N,
        \boldsymbol\beta_{h,n}^N
    \bigr)_{n\geq0}
    \subset
    \mathcal A_\ell^N
    \cap
    \mathcal A_\psi^N.
\]
For every $S>0$, there exist constants $h_{S,N}>0$, $C_{S,N}>0$
such that, for every $0<h\leq h_{S,N}$ and every $n$ satisfying
$nh\leq S$,
\begin{equation*}
    \|
        \boldsymbol\beta_{h,n}^N
    \|_{\BMO(\dbQ_{h,n}^N)}
    \leq
    C_{S,N}, \quad
    \|
        \boldsymbol\beta_{h,n+1}^N
        -
        \boldsymbol\beta_{h,n}^N
    \|_{\BMO}
    \leq
    C_{S,N}h.
    \label{eq-particle-increment-BMO}
\end{equation*}
There exists a unique curve
$(
        \dbQ_s^N,
        \boldsymbol\beta_s^N
    )_{s\geq0}
    \subset
    \mathcal A_\ell^N
    \cap
    \mathcal A_\psi^N$
such that, for every $S>0$, possibly enlarging $C_{S,N}$,
\begin{equation*}
    \sup_{s\in[0,S]}
    \left(
        \|
            \dbQ_{h,s}^N-\dbQ_s^N
        \|_{\TV}
        +
        \|
            \boldsymbol\beta_{h,s}^N
            -
            \boldsymbol\beta_s^N
        \|_{\BMO}
    \right)
    \leq
    C_{S,N}h,
    \qquad
    0<h\leq h_{S,N}.
    \label{eq-particle-flow-convergence}
\end{equation*}
In particular,
\[
    \lim_{h\to0^+}
    \sup_{s\in[0,S]}
    \left(
        \|
            \dbQ_{h,s}^N-\dbQ_s^N
        \|_{\TV}
        +
        \|
            \boldsymbol\beta_{h,s}^N
            -
            \boldsymbol\beta_s^N
        \|_{\BMO}
    \right)
    =
    0.
\]
Furthermore, the finite-particle flow converges exponentially to its
own minimizer:
\begin{equation*}
    H_\psi^N
    \left(
        \dbQ_N^*
        \mid
        \dbQ_s^N
    \right)
    \leq
    e^{-\lambda_{\ell|\psi}s}
    H_\psi^N
    \left(
        \dbQ_N^*
        \mid
        \dbQ_0^N
    \right),
    \qquad
    s\geq0.
    \label{eq-particle-exact-convergence}
\end{equation*}
Moreover
\begin{equation*}
    \qquad
    H^N_\psi(\dbQ^N_s\mid \dbQ^*_N)
 \le\frac{H^N_\psi(\dbQ^*_N\mid\dbQ^N_0)}
              {e^{\lambda_{\ell|\psi}s}-1},
              \qquad
              s>0.
\end{equation*}
\end{thm}

The proof can be found in the finite-particle construction in Section~\ref{sec-finite-particle-proofs}. This theorem concerns convergence to the particle problem's own minimiser; the next estimate compares it with the mean-field optimum.

For each fixed $N$, Theorem~\ref{thm-particle-flow} follows from
Theorem~\ref{thm-natural-structure} and Corollary~\ref{cor-common-flow-decay}
applied on the $Nd$-dimensional canonical
space; see
\eqref{eq-particle-ell-Hessian}--\eqref{eq-particle-relative-convexity}.
The constants $h_{S,N}$ and $C_{S,N}$ are allowed to depend on $N$.
For the verification of this dimensional reduction, see the proof of Theorem~\ref{thm-particle-flow}. No uniform-in-$N$ flow estimate is asserted here.
Let $(\dbQ^*,\beta^*)$
denote the unique minimizer of the mean-field functional \(J(\dbQ) = H_\ell(\dbQ) + F(\dbQ).\) Define the finite-particle defect by
\begin{equation*}
    \varepsilon_N
    :=
    N
    \left(
        \dbE^{(\dbQ^*)^{\otimes N}}
        \left[
            F(\mu^N)
        \right]
        -
        F(\dbQ^*)
    \right).
    \label{eq-particle-epsilon}
\end{equation*}
Under Assumption~\ref{ass-pe}, we have $0
    \leq
    \varepsilon_N
    \leq
    C_{\mathrm{emp}}$
for some constant $C_{\mathrm{emp}}>0$ independent of $N$.

\begin{thm}[Defective long-time convergence]
\label{thm-particle-long-time}

Suppose that Assumptions~\ref{ass-rc},
\ref{ass-pe}, \ref{ass-rc-quadratic},
and~\ref{ass-particle-initial} hold.
Then, for every fixed $N\geq1$ and every $s\geq0$,
\begin{equation}
    H_\psi^N
    \left(
        (\dbQ^*)^{\otimes N}
        \mid
        \dbQ_s^N
    \right)
    \leq
    e^{-\lambda_{\ell|\psi}s}
    H_\psi^N
    \left(
        (\dbQ^*)^{\otimes N}
        \mid
        \dbQ_0^N
    \right)
    +
    \frac{1-e^{-\lambda_{\ell|\psi}s}}
         {\lambda_{\ell|\psi}}
    \varepsilon_N.
    \label{eq-particle-continuous-decay}
\end{equation}
Moreover, for every $s>0$,
\begin{equation}
    H_\psi^N
    \left(
        \dbQ_s^N
        \mid
        (\dbQ^*)^{\otimes N}
    \right)
    \leq
    \frac{
        H_\psi^N
        \left(
            (\dbQ^*)^{\otimes N}
            \mid
            \dbQ_0^N
        \right)
    }{
        e^{\lambda_{\ell|\psi}s}-1
    }
    +
    \frac{\varepsilon_N}{\lambda_{\ell|\psi}}.
    \label{eq-particle-reverse-continuous-decay}
\end{equation}
Consequently,
\begin{align}
    \limsup_{s\to\infty}
    \frac1N
    H_\psi^N
    \left(
        (\dbQ^*)^{\otimes N}
        \mid
        \dbQ_s^N
    \right)
    &\leq
    \frac{\varepsilon_N}{\lambda_{\ell|\psi}N}
    \leq
    \frac{C_{\mathrm{emp}}}{\lambda_{\ell|\psi}N},
    \label{eq-particle-long-time-defect}
    \\
    \limsup_{s\to\infty}
    \frac1N
    H_\psi^N
    \left(
        \dbQ_s^N
        \mid
        (\dbQ^*)^{\otimes N}
    \right)
    &\leq
    \frac{\varepsilon_N}{\lambda_{\ell|\psi}N}
    \leq
    \frac{C_{\mathrm{emp}}}{\lambda_{\ell|\psi}N}.
    \label{eq-particle-reverse-long-time-defect}
\end{align}
\end{thm}
\proofguide{The proof can be found in Section~\ref{sec-finite-particle-proofs}: the empirical-defect bound and discrete contraction yield the continuous-time conclusion. The defect is bounded independently of $N$, but it is not zero.}

\begin{cor}[Product initial condition]
\label{cor-particle-product-initial}

Under the assumptions of
Theorem~\ref{thm-particle-long-time}, suppose in addition that
$\dbQ_0^N=\dbQ_0^{\otimes N}$.
Then
\begin{equation}
    H_\psi^N
    \left(
        (\dbQ^*)^{\otimes N}
        \mid
        \dbQ_0^{\otimes N}
    \right)
    =
    N H_\psi
    \left(
        \dbQ^*
        \mid
        \dbQ_0
    \right).
    \label{eq-particle-product-initial-entropy}
\end{equation}
Consequently, for every $s\geq0$,
\begin{equation}
    \frac1N
    H_\psi^N
    \left(
        (\dbQ^*)^{\otimes N}
        \mid
        \dbQ_s^N
    \right)
    \leq
    e^{-\lambda_{\ell|\psi}s}
    H_\psi
    \left(
        \dbQ^*
        \mid
        \dbQ_0
    \right)
    +
    \frac{1-e^{-\lambda_{\ell|\psi}s}}
         {\lambda_{\ell|\psi}}
    \frac{\varepsilon_N}{N}
    \leq
    e^{-\lambda_{\ell|\psi}s}
    H_\psi
    \left(
        \dbQ^*
        \mid
        \dbQ_0
    \right)
    +
    \frac{C_{\mathrm{emp}}}{\lambda_{\ell|\psi}N}.
\label{eq-particle-final-estimate}
\end{equation}
Moreover, for every $s>0$ and every $i\in\{1,2,\cdots,N\}$,
\begin{equation}
    H_\psi(\dbQ^{N,i}_s\mid\dbQ^*)
    \le
    \frac1N
    H_\psi^N
    \left(
        \dbQ_s^N
        \mid
        (\dbQ^*)^{\otimes N}
    \right)
    \leq
    \frac{
        H_\psi
        \left(
            \dbQ^*
            \mid
            \dbQ_0
        \right)
    }{
        e^{\lambda_{\ell|\psi}s}-1
    }
    +
    \frac{\varepsilon_N}{\lambda_{\ell|\psi}N}
    \leq
    \frac{
        H_\psi
        \left(
            \dbQ^*
            \mid
            \dbQ_0
        \right)
    }{
        e^{\lambda_{\ell|\psi}s}-1
    }
    +
    \frac{C_{\mathrm{emp}}}{\lambda_{\ell|\psi}N}.
\label{eq-particle-reverse-final-estimate}
\end{equation}
\end{cor}
\proofguide{The proof can be found in the product-initialization argument in Section~\ref{sec-particle-product-proof}. The displayed normalisation is the reason the persistent defect is of order $N^{-1}$.}

\subsection{Propagation of regularity}\label{sec-main-regularity}
We now use Markov structure to strengthen convergence of laws into convergence of feedbacks. Physical-time smoothing and learning-time decay remain separate estimates.

We work on $\R^d$ under Assumption~\ref{ass-markov}. The relevant
coupling estimate concerns the diffusion used to evaluate a policy,
not a contraction in the learning-time variable. For $B\ge0$ and
$r>0$ put
\begin{equation*}\label{eq-reflection-kernels}
 \eta_B(r):=(2\pi r)^{-1/2}+B,
 \qquad m_B(r):=1+2B\sqrt{2r/\pi}+2B^2r.
\end{equation*}

\begin{lem}[Reflection estimate for policy evaluation]
\label{prop-reflection-evaluation}
Let $P^b_{t,r}$ be the transition operator of
$\dd X_r=b(r,X_r)\dd r+\dd W_r$ on $\R^d$, where $b$ is a
bounded Borel feedback with $\|b\|_\infty\le B$.
For bounded measurable $f$, and Lipschitz $k$ of at most linear growth,
\begin{equation}\label{eq-reflection-semigroup}
 \operatorname{Lip}(P^b_{t,r}f)\le\eta_B(r-t)\osc(f),
 \qquad
 \operatorname{Lip}(P^b_{t,r}k)\le m_B(r-t)\operatorname{Lip}(k).
\end{equation}
In particular the evaluation
\[
 v(t,x)=P^b_{t,T}g(x)+\int_t^T P^b_{t,r}f_r(x)\,\dd r
\]
satisfies
\begin{equation}\label{eq-reflection-evaluation}
 \operatorname{Lip}(v(t,\cdot))
 \le m_B(T-t)\operatorname{Lip}(g)
       +\int_t^T\eta_B(r-t)\osc(f_r)\,\dd r,
\end{equation}
whenever the displayed integral is finite. No derivative of $b$
appears in these constants.
\end{lem}
\proofguide{The proof can be found in the reflection-coupling argument in Section~\ref{sec-reflection-proof}. Bounded sources and Lipschitz terminal data use different estimates, which is useful near the terminal time.}

\begin{thm}[Whole-space Newton regularity and policy improvement]
\label{thm-markov-regularity}
Under items (i)--(iii) of Assumption~\ref{ass-markov}, the Newton
law in Theorem~\ref{thm-newton-structure} has a bounded Markov
feedback $\beta_t^s=\beta^s(t,X_t)$. Let
\begin{equation}\label{eq-markov-evaluation}
 v^s(t,x)=\dbE_{t,x}^{\beta^s}
 \left[g(X_T)+\int_t^T\ell(r,X_r,\beta^s(r,X_r))\,\dd r\right].
\end{equation}
Its policy-evaluation PDE and improvement equation are
\begin{align}
\notag &\partial_t v^s+\tfrac12\Delta v^s
       +\beta^s\cdot\nabla_xv^s+\ell(t,x,\beta^s)=0,
       \qquad v^s(T,x)=g(x),\label{eq-markov-direct-eval}\\
 &\partial_s\nabla_b\ell(t,x,\beta^s(t,x))
       =-\nabla_b\ell(t,x,\beta^s(t,x))-\nabla_xv^s(t,x).
       \label{eq-markov-direct-flow}
\end{align}
The second equation may equivalently be read in integral form in
$s$. The feedback is the unique locally bounded solution of this
policy-evaluation/improvement system with the prescribed initial
policy. Spatial derivatives below are classical for $t<T$; bounded
terminal derivatives are understood almost everywhere when necessary.
There is a finite constant $C$, depending only on the data and $T$,
such that
\begin{align}
\notag &\sup_{s\ge0}\bigl(\|\beta^s\|_\infty+\|v^s\|_\infty
                            +\|\nabla_xv^s\|_\infty\bigr)\le C,
                 \label{eq-markov-uniform-low}\\
 &\sup_{s\ge0}\sup_{t<T}\sqrt{T-t}
       \bigl(\|D_x\beta^s(t)\|_\infty+
                          \|D_x^2v^s(t)\|_\infty\bigr)\le C.
                 \label{eq-markov-smoothing-low}
\end{align}
Let $v^*$ be the optimal value and
$\beta^*(t,x)=-D_pH(t,x,\nabla_xv^*(t,x))$ its feedback. Then
\begin{align}
 &\|\beta^s-\beta^*\|_\infty\le Ce^{-s},
                      \label{eq-markov-feedback-rate}\\
 &0\le v^s(t,x)-v^*(t,x)\le Ce^{-2s},\qquad
   \|\nabla_xv^s-\nabla_xv^*\|_\infty\le Ce^{-2s}.
                      \label{eq-markov-value-rate}
\end{align}
The implicit Newton iterates also have Markov feedbacks and satisfy
\begin{equation}\label{eq-markov-feedback-discrete}
 \sup_{s\ge0}\|\beta^{h,\lfloor s/h\rfloor}-\beta^s\|_\infty\le Ch
\end{equation}
for sufficiently small $h>0$.
\end{thm}
\proofguide{The proof can be found in the six-step regularity argument in Section~\ref{sec-reflection-proof}. Spatial bounds are uniform in learning time, but the constants may depend on the physical horizon.}

\begin{rem}[Higher spatial norms with smooth data]\label{cor-markov-smooth}\label{rem-markov-smooth}
Under the additional derivative bounds in item (iv) of Assumption~\ref{ass-markov}, the same spatial bootstrap extends the estimates to higher norms; see also the parabolic framework of \cite{HenryGeometricParabolic}. Fix an integer $m\ge0$ and assume that item at order $m$.
With spatial $C_b^m$ norms taken uniformly over $t\in[0,T]$,
there is a finite $C_m$ such that
\begin{align*}
 &\sup_{s\ge0}\bigl(\|\beta^s\|_{C_b^m}+\|v^s\|_{C_b^{m+1}}\bigr)
    \le C_m,\label{eq-markov-smooth-uniform}\\
 &\|\beta^s-\beta^*\|_{C_b^m}\le C_me^{-s},\qquad
   \|v^s-v^*\|_{C_b^{m+1}}\le C_me^{-2s},
                      \label{eq-markov-smooth-rates}\\
 &\sup_{s\ge0}\|\beta^{h,\lfloor s/h\rfloor}-\beta^s\|_{C_b^m}
       \le C_mh.\label{eq-markov-smooth-discrete}
\end{align*}
\end{rem}

\begin{rem}[Static formulation from the zero policy]\label{rem-markov-zero-static}
Suppose in addition that $\beta^0=0$, so $\dbQ^0=\dbP$, and write $a=e^{-s}$. Theorem~\ref{thm-newton-structure} gives
\begin{equation*}\label{eq-markov-zero-variational}
 \dbQ^s=\argmin_{\dbQ}\bigl\{(1-a)J(\dbQ)+aH_\ell(\dbQ\mid\dbP)\bigr\}.
\end{equation*}
For $a=1$ the minimiser is $\dbP$, and for $a=0$ it is $\dbQ^*$. To state the corresponding static PDE without imposing a normalisation on $\ell$, let $q^0(t,x)=\nabla_b\ell(t,x,0)$. The value $U_a$ for running cost $\ell(t,x,b)-a\ell(t,x,0)-a q^0(t,x)\cdot b$ and terminal cost $(1-a)g$ satisfies
\begin{equation*}\label{eq-markov-zero-static-pde}
 \partial_tU_a+\tfrac12\Delta U_a
 -H(t,x,\nabla_xU_a-aq^0)-a\ell(t,x,0)=0,
 \qquad U_a(T,x)=(1-a)g(x),
\end{equation*}
and the feedback is $\beta^s(t,x)=-D_pH(t,x,\nabla_xU_a(t,x)-a q^0(t,x))$.
In particular, when $\ell(t,x,b)=\ell_0(b)+V(t,x)$ with $\ell_0(0)=0$ and $\nabla\ell_0(0)=0$, set $\theta=1-a$ and $u^\theta=U_a$. Then
\begin{equation*}\label{eq-markov-zero-separable}
 \partial_tu^\theta+\tfrac12\Delta u^\theta
 +\inf_b\{b\cdot\nabla_xu^\theta+\ell_0(b)\}+\theta V=0,
 \qquad u^\theta(T,x)=\theta g(x),
\end{equation*}
with $\beta^s=(\nabla\ell_0)^{-1}(-\nabla_xu^{1-e^{-s}})$. The static value $u^{1-e^{-s}}$ is not the current-policy value $v^s$ in \eqref{eq-markov-evaluation}.
These identities are the zero-policy specialization of the static construction in Theorem~\ref{thm-newton-structure}. They give an alternative representation of the same Newton curve.
\end{rem}

\begin{exm}[A nonquadratic, state-dependent Hamiltonian]
Take bounded smooth $U,V$ with bounded spatial derivatives,
$0\le V\le1$, and $0<\epsilon<1/2$. Then
\[
 H(t,x,p)=\tfrac12|p|^2+
       \epsilon V(t,x)\bigl(\sqrt{1+|p|^2}-1\bigr)+U(t,x)
\]
satisfies the Hamiltonian bounds, with
$I_d\le H_{pp}\le(1+\epsilon)I_d$.
Its Legendre running cost is nonquadratic and, for nonconstant $V$,
not a sum of a control cost and a state cost.
For instance $V(t,x)=e^{-|x|^2}$ and $U(t,x)=t e^{-|x|^2}$ are
admissible whole-space coefficients. The theorem applies to any
bounded smooth terminal datum and initial mirror coordinate satisfying
the stated bounds.
\end{exm}
For verification, the momentum Hessian of $\sqrt{1+|p|^2}-1$ is positive semidefinite and bounded above by $I_d$. Bounded derivatives of $U,V$ then give the remaining assumptions; the Legendre-dual bounds are used in Step~1 of the Markov proof.

\section{Path entropy, BSDEs and the proximal scheme}\label{sec-foundations}

The proof has three layers. We first control changes of measure by BMO estimates, then construct and compare the relevant BSDEs, and finally derive the path-entropy calculus and the proximal inequalities.

\Needspace{8\baselineskip}
\subsection{\texorpdfstring{BMO estimates and changes of measure}{BMO estimates and changes of measure}}\label{sec-bmo-tools}

We first recall several standard integrability properties of BMO
martingales. Define the function $\Phi:(1,\infty)\to(0,\infty)$ by
\[
\Phi(p)
:=
\left(
1+\frac{1}{p^2}
\log\frac{2p-1}{2(p-1)}
\right)^{\frac12}-1,
\qquad p\in(1,\infty).
\]
The function $\Phi$ is strictly decreasing and maps $(1,\infty)$ onto
$(0,\infty)$.

We shall use the following standard results. The reverse H\"older
inequality and the invariance of the BMO property under Girsanov
transformations may be found in
\cite[Theorem~2.2, Theorem~3.1 and Theorem~3.6]{kazamaki1994continuous}.

\begin{lem}[Exponential integrability of BMO processes]
\label{lem-BMO-exp-bound}
Let $(\dbQ,\beta)\in\cA$, and let $\gamma$ be a
$\BMO(\dbQ)$ process. Then, for every $\epsilon$ satisfying
\[
\|\gamma\|_{\BMO(\dbQ)}<\epsilon^{-\frac{1}{2}},
\]
we have
\[
\sup_{\tau\in\cT}
\left\|
\dbE^{\dbQ}\left[
    \exp\left(
        \epsilon\int_\tau^T|\gamma_t|^2\,\dd t
    \right)
    \Bigm|\cF_\tau
\right]
\right\|_\infty
\le
(
1-\epsilon\|\gamma\|_{\BMO(\dbQ)}^2)^{-1}
.
\]
Consequently,
\[
\sup_{\tau\in\cT}
\left\|
\dbE^{\dbQ}\left[
    \left(
        \int_\tau^T|\gamma_t|^2\,\dd t
    \right)^p
    \Bigm|\cF_\tau
\right]
\right\|_\infty
<\infty,
\qquad
\forall\,p\in[1,\infty).
\]
\end{lem}

\begin{lem}[Reverse H\"older inequality]\label{lem-BMO-reverse-holder}
Let $\dbQ\sim\dbQ'$ be two probability measures such that, under
$\dbQ$, the density process of $\dbQ'$ with respect to $\dbQ$ is
$D(\delta\beta)$
for some
$\delta\beta\in\BMO(\dbQ)$. Let
\[
D_t
:=
\dbE^{\dbQ}\left[
    \frac{\dd\dbQ'}{\dd\dbQ}
    \Bigm|\cF_t
\right],
\qquad t\in[0,T],
\]
be the density process of $\dbQ'$ with respect to $\dbQ$. Suppose that $\delta\beta\in\BMO(\dbQ)$ and denote its $\BMO(\dbQ)$-norm by $b$. For $b>0$ and $1<p<\Phi^{-1}(b)$, there is a constant $C_{p,b}$ defined as follows:
\[
C_{p,b}=\left(
1-\frac{2(p-1)}{2p-1}
\exp\left(
p^2\left(
b^2
+2b
\right)
\right)
\right)^{-1}.
\]
Then the following assertions hold:
\begin{enumerate}
\item
The density process $D_t$ satisfies the reverse
H\"older inequality of order $p$:
\[
\sup_{\tau\in\cT}
\left\|
\dbE^{\dbQ}\left[
    \left(\frac{D_T}{D_\tau}\right)^p
    \Bigm|\cF_\tau
\right]
\right\|_\infty
<C_{p,b}.
\]
In particular,
\(
D_T=\frac{\dd\dbQ'}{\dd\dbQ}\in L^p(\dbQ).
\)

\item
We have
\(
\delta\beta\in\BMO(\dbQ'),
\)
and
\[
\|\delta\beta\|^2_{\BMO(\dbQ')}
\le
2^{\frac{2p-1}{p}}\tfrac{p}{p-1}\,C_{p,b}^{\frac{1}{p}}\,b^2
=:
\kappa^p\left(b\right).
\]
For each fixed $p$, the function $\kappa^p$ is increasing
on $[0,\Phi(p))$.

\item
For every $\dbF$-progressively measurable process $\gamma$, one has
\[
\gamma\in\BMO(\dbQ)
\quad\Longleftrightarrow\quad
\gamma\in\BMO(\dbQ').
\]
and
\[
\|\gamma\|^2_{\BMO(\dbQ')}
\le
{(1+\tfrac{\sqrt{2}}{2}\|\delta\beta\|_{\BMO(\dbQ')})^2} \,\|\gamma\|_{\BMO(\dbQ)}^2.
\]
When $b=0$, $D\equiv1$ and the assertions hold directly, with
$\kappa^p(0)=0$ and any $p>1$. We use the convention
\[
 p_b=\begin{cases}(1+\Phi^{-1}(b))/2,&b>0,\\2,&b=0.\end{cases}
\]
\end{enumerate}
\end{lem}

\begin{lem}[Lipschitz continuity of the stochastic exponential on BMO balls]
\label{lem-BMO-lip}
Let $K>0$, and let $\beta,\beta'$ be two predictable processes such that
\[
\beta,\beta'\in\BMO,
\qquad
\|\beta\|_{\BMO}\vee\|\beta'\|_{\BMO}\le K.
\]
Recall that
\[
D(\beta)
:=
\mathcal E\left(
\int_0^\cdot \beta_t\,\dd X_t
\right)_T,
\qquad
D(\beta')
:=
\mathcal E\left(
\int_0^\cdot \beta'_t\,\dd X_t
\right)_T.
\]
Then $D(\beta)$ and $D(\beta')$ are uniformly integrable martingales
with expectation one and therefore define probability measures
$\dbQ^\beta$ and $\dbQ^{\beta'}$ equivalent to $\dbP$ by
\[
\frac{\dd\dbQ^\beta}{\dd\dbP}
=
D(\beta),
\qquad
\frac{\dd\dbQ^{\beta'}}{\dd\dbP}
=
D(\beta').
\]
Moreover, there exists a constant $L_K>0$, depending only on $K$, such that
\begin{equation*}
\label{eq-BMO-stochastic-exponential-Lipschitz}
\left\|
D(\beta)-D(\beta')
\right\|_{L^1(\dbP)}
\le
L_K
\|\beta-\beta'\|_{\BMO}.
\end{equation*}
\end{lem}

A change in drift must also control the change of law. This estimate is the link between the BMO policy norm and the total-variation topology used for convergence.
\begin{proof}
\phantomsection\label{proof-bmo-lipschitz}

\revstep{1}{Entropy representation}
Set \(\delta\beta:=\beta'-\beta.\) By Pinsker's inequality,
\[
\left\|
D(\beta)-D(\beta')
\right\|_{L^1(\dbP)}
\le
\sqrt{
2H(\dbQ^{\beta'}\,|\,\dbQ^\beta)
}.
\]
By Girsanov's theorem,
\[
H(\dbQ^{\beta'}\,|\,\dbQ^\beta)
=
\frac12
\dbE^{\dbQ^{\beta'}}
\left[
\int_0^T
|\delta\beta_t|^2\,\dd t
\right].
\]
Hence
\begin{equation}
\label{eq-BMO-lip-entropy-bound}
\left\|
D(\beta)-D(\beta')
\right\|_{L^1(\dbP)}
\le
\|\delta\beta\|_{\BMO(\dbQ^{\beta'})}.
\end{equation}

\revstep{2}{Return to the reference BMO norm}
Since
\[
\|\beta'\|_{\BMO}\le K,
\]
Lemma~\ref{lem-BMO-reverse-holder}.2 shows that the constant $C_{p,b}$ is increasing in $b$, and Lemma~\ref{lem-BMO-reverse-holder}.3 yields a constant
$L_K>0$, depending only on $K$, such that, for every
$\gamma\in\BMO$,
\[
\|\gamma\|_{\BMO(\dbQ^{\beta'})}
\le
L_K\|\gamma\|_{\BMO}.
\]
Applying this estimate to $
\gamma=\delta\beta$
and using \eqref{eq-BMO-lip-entropy-bound}, we obtain
\[
\left\|
D(\beta)-D(\beta')
\right\|_{L^1(\dbP)}
\le
L_K
\|\beta-\beta'\|_{\BMO}.
\]
\end{proof}
\subsection{BSDE estimates and stability}\label{sec-bsde-estimates}
The next three results supply the analytic engine for each variational step: identify a drift minimiser, solve the associated BSDE, and close the dependence on the law.

In this subsection, we study a class of BSDEs associated with the
running cost function $g:[0,T]\times\Omega\times\dbR^d\rightarrow\dbR$, which is jointly
$\operatorname{Prog}(\dbF)\otimes\mathcal B(\dbR^d)$-measurable.
Define the corresponding Hamiltonian by
\[
\cH_g(t,\omega,z)
:=
\inf_{b\in\dbR^d}
\left\{
b\cdot z+g(t,\omega,b)
\right\},
\qquad
(t,\omega,z)\in[0,T]\times\Omega\times\dbR^d.
\]
We say that $g$ satisfies the $(p,\gamma)$-quadratic structure condition if there exist a constant $\gamma>0$, a number $p>1$, and a nonnegative $\dbF$-progressively measurable process $\alpha(t,X)$ such that
\[ 
-\alpha(t,X) 
\le -\cH_g(t,X,z) 
\le \frac{\gamma}{2}\bigl(1+|z|^2\bigr), \qquad z\in\dbR^d, 
\] 
and 
\[ 
\|\alpha\|_{p,T} 
:= \sup_{\tau\in\cT} \left\| \dbE\left[ \left( \int_\tau^T \alpha(s,X)\,\dd s \right)^p \Bigm|\cF_\tau \right] \right\|_\infty^{1/p} <\infty. 
\] 
When this condition holds, we say that $g$ is $(p,\gamma)$-quadratically admissible.

In this subsection, $g$ is jointly
$\operatorname{Prog}(\dbF)\otimes\mathcal B(\dbR^d)$-measurable and $(p,\gamma)$-quadratically admissible. Moreover, for every $(t,\omega)\in[0,T]\times\Omega$, \(
b\mapsto g(t,\omega,b)
\)
belongs to $C^2(\dbR^d)$ and there is a $c_g>0$ such that
\[
\nabla^2 g(t,\omega,b)\ge c_g I_d.
\]
Given an $\cF_T$-measurable random variable $\xi$, consider BSDE
\begin{equation}\label{BSDE-g-xi}
\begin{cases}
\dd Y_t
=
-\cH_g(t,X,Z_t)\,\dd t
+Z_t\,\dd X_t,
\\[2mm]
Y_T=\xi.
\end{cases}
\end{equation}
Equivalently,
\[
Y_t
=
\xi
+\int_t^T\cH_g(s,X,Z_s)\,\dd s
-\int_t^T Z_s\,\dd X_s,
\qquad
0\le t\le T,
\quad \dbP\text{-a.s.}
\]
We refer to \eqref{BSDE-g-xi} as
$\operatorname{BSDE}(\cH_g,\xi)$.

\begin{lem}[Estimate for the zero-feedback minimizer]
\label{lem-BSDE-b0}
The minimizer of $g$:
\[
b_g^0(t,\omega)
:=
\argmin_{b\in\dbR^d} g(t,\omega,b)
\]
is uniquely defined and
\[
b_g^0(\cdot,X)\in\BMO.
\]
More precisely,
\begin{equation*}
\|b_g^0(\cdot,X)\|_{\BMO}
\le
\left(
\gamma^2T
+
2\gamma\|\alpha\|_{p,T}
\right)^{1/2}.
\end{equation*}
\end{lem}

The pointwise minimiser of a Hamiltonian is useful only if it is an admissible drift. Strong convexity turns its first-order condition into the needed conditional energy bound.
\begin{proof}
\phantomsection\label{proof-bsde-minimising-drift}

\revstep{1}{Pointwise control of the minimizer}
The uniform strong convexity of
$b\mapsto g(t,\omega,b)$ implies that $b_g^0(t,\omega)$ exists and is
unique.

Fix $(t,\omega)\in[0,T]\times\Omega$ and set
\[
\varphi(z):=-\cH_g(t,\omega,z)
=
\sup_{b\in\dbR^d}
\left\{
-b\cdot z-g(t,\omega,b)
\right\}.
\]
Since $b_g^0(t,\omega)$ minimizes $g(t,\omega,\cdot)$,
\[
\varphi(0)
=
-g(t,\omega,b_g^0(t,\omega)).
\]
Moreover, for every $z\in\dbR^d$,
\[
\varphi(z)
\ge
-b_g^0(t,\omega)\cdot z
-g(t,\omega,b_g^0(t,\omega))
=
\varphi(0)-b_g^0(t,\omega)\cdot z.
\]
On the other hand, the $(p,\gamma)$-quadratic structure condition gives
\[
\varphi(z)
\le
\frac{\gamma}{2}(1+|z|^2),
\qquad
\varphi(0)\ge-\alpha(t,\omega).
\]
Consequently,
\[
-b_g^0(t,\omega)\cdot z
\le
\frac{\gamma}{2}(1+|z|^2)
+\alpha(t,\omega),
\qquad
z\in\dbR^d.
\]
Taking
\[
z=-\frac{b_g^0(t,\omega)}{\gamma},
\]
we obtain
\[
\frac{|b_g^0(t,\omega)|^2}{\gamma}
\le
\frac{\gamma}{2}
+
\frac{|b_g^0(t,\omega)|^2}{2\gamma}
+
\alpha(t,\omega),
\]
and therefore
\begin{equation*}
\label{eq-BSDE-b0-pointwise}
|b_g^0(t,\omega)|^2
\le
\gamma^2+2\gamma\alpha(t,\omega).
\end{equation*}

\revstep{2}{Conditional energy estimate}
Let $\tau\in\cT$. It follows that
\begin{equation*}
\dbE\left[ \int_\tau^T |b_g^0(t,X)|^2\,\dd t \Bigm|\cF_\tau \right] \le \gamma^2 T + 2\gamma \dbE\left[ \int_\tau^T\alpha(t,X)\,\dd t \Bigm|\cF_\tau \right] \le \gamma^2 T + 2\gamma\|\alpha\|_{p,T},
\end{equation*}
where the last inequality follows from conditional H\"older's
inequality and the definition of $\|\alpha\|_{p,T}$.
Taking the essential supremum over $\tau\in\cT$ gives
\[
\|b_g^0(\cdot,X)\|_{\BMO}
\le
\left(
\gamma^2T
+
2\gamma\|\alpha\|_{p,T}
\right)^{1/2}.
\]
\end{proof}

\begin{lem}[Well-posedness and comparison principle]
\label{lem-BSDE-solution}
Let $\xi$ be an $\cF_T$-measurable bounded random variable. Then
$\operatorname{BSDE}(\cH_g,\xi)$ admits a unique solution
$(Y,Z)$ such that
\[
Y\in\mathcal S^\infty,
\qquad
Z\in\BMO.
\]
Moreover,
\[
\sup_{0\le t\le T}\|Y_t\|_\infty
+
\|Z\|_{\BMO}
\le
C\bigl(
p,\gamma,T,\|\xi\|_\infty,\|\alpha\|_{p,T}
\bigr).
\]
Furthermore, let $(\dbQ,\beta)\in\overline{\cA}$ and suppose that
\[
\dbE^{\dbQ}\left[
\int_0^T
|g(t,X,\beta_t)|\,\dd t
\right]
<\infty.
\]
Then, for every stopping time $\tau\in\cT$,
\begin{equation*}
Y_\tau
\le
\dbE^{\dbQ}\left[
\xi
+
\int_\tau^Tg(t,X,\beta_t)\,\dd t
\Bigm|
\cF_\tau
\right],
\qquad
\dbQ\text{-a.s.}
\end{equation*}
Since $\dbQ\sim\dbP$, the inequality also holds $\dbP$-a.s.
\end{lem}

We first freeze the terminal condition. Comparison and localisation keep the quadratic BSDE within the admissible class before the law is allowed to vary.
\begin{proof}
\phantomsection\label{proof-bsde-solution}

\revkeepstep
\noindent
\textit{Step 1: Existence.}
By \cite[Theorem~2.1]{delbaen2011uniqueness}, the BSDE admits a
solution $(Y,Z)$ satisfying
\begin{align*}
-\frac1\gamma\log\dbE\left[\exp\left(\gamma\xi^-+\frac{\gamma^2}{2}(T-t)\right)\Bigm|\cF_t\right]\le Y_t\le C_p e^{C_pT}\left(\dbE\left[|\xi^+|^p+\left|\int_t^T\alpha(s,X)\,\dd s\right|^p\Bigm|\cF_t\right]\right)^{1/p}.
\end{align*}
where $C_p>0$ depends only on $p$ and the structural constants of the
generator. Since $\xi$ is bounded and $\|\alpha\|_{p,T}<\infty$, it
follows that
\[
Y\in\mathcal S^\infty.
\]
Moreover, adapting the conditional a priori estimate in
\cite[Corollary~4]{briand2008quadratic}, we obtain
\[
\|Z\|_{\BMO}
\le
C\bigl(
p,\gamma,T,\|\xi\|_\infty,\|\alpha\|_{p,T}
\bigr).
\]

\revkeepstep
\noindent
\textit{Step 2: Uniqueness.}

\revsubstep{2a}{The minimizing feedback}
For $(t,\omega,z)\in[0,T]\times\Omega\times\dbR^d$, define
\[
h^g(t,\omega,z)
:=
\argmin_{b\in\dbR^d}
\left\{
b\cdot z+g(t,\omega,b)
\right\}.
\]
Since
\(
(t,\omega,z,b)
\mapsto
b\cdot z+g(t,\omega,b)
\)
is
\(
\operatorname{Prog}(\mathbb F)
\otimes
\mathcal B(\mathbb R^d)
\otimes
\mathcal B(\mathbb R^d)
\)
-measurable and continuous in $b$, and since the uniform strong
convexity of $g(t,\omega,\cdot)$ implies coercivity and uniqueness of
the minimizer, then  \cite[Theorem~14.37]{rockafellar1998variational}  implies that
\[
(t,\omega,z)
\mapsto
h^g(t,\omega,z)
\]
is
\(
\operatorname{Prog}(\mathbb F)
\otimes
\mathcal B(\mathbb R^d)
\)
-measurable. In particular, if $Z$ is progressively measurable, then
$h^g(\cdot,\cdot,Z)$ is progressively measurable.
Since $b\mapsto g(t,\omega,b)$ is $c_g$-strongly convex,
$h^g$ exists and is unique. It is characterized by
\begin{equation*}
\label{eq-BSDE-1}
\nabla g
\bigl(
t,\omega,h^g(t,\omega,z)
\bigr)
=-z.
\end{equation*}
In particular,
\begin{equation}
\label{eq-BSDE-1.1}
b_g^0(t,\omega)=h^g(t,\omega,0).
\end{equation}
By the strong monotonicity of $\nabla g$, for every
$z,z'\in\dbR^d$,
\begin{equation}
\label{eq-BSDE-2}
\left|
h^g(t,\omega,z)
-
h^g(t,\omega,z')
\right|
\le
\frac1{c_g}|z-z'|.
\end{equation}
Moreover, $\cH_g(t,\omega,\cdot)$ is continuously differentiable and
\begin{equation}
\label{eq-BSDE-3}
\nabla_z\cH_g(t,\omega,z)
=
h^g(t,\omega,z).
\end{equation}

\revsubstep{2b}{Linearization and change of measure}
Let
$(Y,Z)$ and $(\widetilde Y,\widetilde Z)$
be two solutions of
$\operatorname{BSDE}(\cH_g,\xi)$ in
$\mathcal S^\infty\times\BMO$.
Set \(\Delta Y:=Y-\widetilde Y\), \( \Delta Z:=Z-\widetilde Z\). For $\theta\in[0,1]$, define \(Z_t^\theta := \widetilde Z_t+\theta\Delta Z_t\) and
\[
\lambda_t
:=
\int_0^1
h^g(t,X,Z_t^\theta)\,\dd\theta.
\]
By \eqref{eq-BSDE-3} and the fundamental theorem of
calculus,
\begin{equation}
\label{eq-BSDE-4}
\cH_g(t,X,Z_t)
-
\cH_g(t,X,\widetilde Z_t)
=
\lambda_t\cdot\Delta Z_t.
\end{equation}

We next verify that $\lambda\in\BMO$.
By \eqref{eq-BSDE-2},
\[
|h^g(t,X,z)|
\le
|b_g^0(t,X)|
+
\frac1{c_g}|z|.
\]
Hence
\begin{equation*}
|\lambda_t| \le |b_g^0(t,X)| + \frac1{c_g} \int_0^1|Z_t^\theta|\,\dd\theta \le |b_g^0(t,X)| + \frac1{2c_g} \left( |Z_t|+|\widetilde Z_t| \right).
\end{equation*}
Using Lemma~\ref{lem-BSDE-b0}, we obtain
\begin{equation*}
\|\lambda\|_{\BMO}
\le
\|b_g^0(\cdot,X)\|_{\BMO}
+
\frac1{2c_g}
\left(
\|Z\|_{\BMO}
+
\|\widetilde Z\|_{\BMO}
\right)
<\infty.
\end{equation*}
Since $\lambda\in\BMO$, the stochastic exponential
$D(\lambda)$
is a uniformly integrable martingale. Let $\dbQ^\lambda$ be the probability measure equivalent to $\dbP$ defined by
\[
\frac{\dd\dbQ^\lambda}{\dd\dbP}
=
D(\lambda).
\]

\revsubstep{2c}{Identification of the two solutions}
Subtracting the two BSDEs and using
\eqref{eq-BSDE-4}, we obtain
\[
\dd\Delta Y_t
=
-\lambda_t\cdot\Delta Z_t\,\dd t
+
\Delta Z_t\,\dd X_t
=
\Delta Z_t\,\dd W_t^\lambda,
\qquad
\Delta Y_T=0.
\]
Thus $\Delta Y$ is a continuous local martingale under
$\dbQ^\lambda$. Since $Y,\widetilde Y\in\mathcal S^\infty$,
$\Delta Y$ is bounded and therefore a uniformly integrable martingale.
Consequently,
\[
\Delta Y_t
=
\dbE^{\dbQ^\lambda}
\left[
\Delta Y_T
\mid\cF_t
\right]
=0,
\qquad
0\le t\le T,
\]
hence $Y=\widetilde Y$.
Returning to the BSDE dynamics gives
\[
\int_0^\cdot
\Delta Z_t\,\dd W_t^\lambda
=0.
\]
Therefore its quadratic variation vanishes:
\[
\int_0^T|\Delta Z_t|^2\,\dd t=0,
\qquad
\dbQ^\lambda\text{-a.e.}
\]
Since $\dbQ^\lambda\sim\dbP$,
\[
Z=\widetilde Z,
\qquad
\dd t\otimes\dd\dbP\text{-a.e.}
\]
This proves uniqueness.

\revkeepstep
\noindent
\textit{Step 3: Comparison principle.}

\revsubstep{3a}{Policy evaluation and comparison dynamics}
In this step, we denote by $(Y,Z)$ the unique solution to BSDE~$(\cH_g,\xi)$.
Let $(\dbQ,\beta)\in\overline{\cA}$ satisfy
\[
\dbE^{\dbQ}\left[
\int_0^T|g(t,X,\beta_t)|\,\dd t
\right]
<\infty,
\]
and define
\[
Y_t^\beta
:=
\dbE^{\dbQ}\left[
\xi+\int_t^Tg(r,X,\beta_r)\,\dd r
\Bigm|\cF_t
\right].
\]
Since $(\dbQ,\beta)\in\overline{\cA}$, the process $W^\beta$ is a
$(\dbQ,\dbF)$-Brownian motion.
By the martingale representation theorem
\cite[Theorem~3.4.15]{karatzas1991brownian}, there exists a predictable,
locally square-integrable process $Z^\beta$ such that
\begin{equation}
\label{eq-BSDE-5}
\dd Y_t^\beta
=
-g(t,X,\beta_t)\,\dd t
+
Z_t^\beta\,\dd W_t^\beta.
\end{equation}
Equivalently,
\[
\dd Y_t^\beta
=
-
\left(
\beta_t\cdot Z_t^\beta
+
g(t,X,\beta_t)
\right)\dd t
+
Z_t^\beta\,\dd X_t,
\qquad
Y_T^\beta=\xi.
\]
Under $\dbQ$, we have
\begin{equation}
\label{eq-BSDE-6}
\dd Y_t
=
\left(
\beta_t\cdot Z_t
-
\cH_g(t,X,Z_t)
\right)\dd t
+
Z_t\,\dd W_t^\beta.
\end{equation}
Define \(\Delta Y:=Y-Y^\beta, \qquad \Delta Z:=Z-Z^\beta.\) By the definition of $\cH_g$,
\[
\beta_t\cdot Z_t
+
g(t,X,\beta_t)
-
\cH_g(t,X,Z_t)
\ge0.
\]
Combining
\eqref{eq-BSDE-5}
and
\eqref{eq-BSDE-6},
we obtain
\begin{equation}
\label{eq-BSDE-7}
\dd\Delta Y_t
=
\left(
\beta_t\cdot Z_t
+
g(t,X,\beta_t)
-
\cH_g(t,X,Z_t)
\right)\dd t
+
\Delta Z_t\,\dd W_t^\beta,
\qquad
\Delta Y_T=0.
\end{equation}

\revsubstep{3b}{Localization and uniform integrability}
Fix $\tau\in\cT$ and define
\[
\tau_k^*
:=
\inf\left\{
t\ge\tau:
\int_\tau^t|\Delta Z_s|^2\,\dd s\ge k
\right\}\wedge T.
\]
Since $Z$ and $Z^\beta$ are locally square-integrable,
\[
\tau_k^*\uparrow T,
\qquad
\dbQ\text{-a.s.}
\]
Integrating
\eqref{eq-BSDE-7}
over $[\tau,\tau_k^*]$ and taking the conditional expectation under
$\dbQ$ yields
\begin{equation}
\label{eq-BSDE-8}
\Delta Y_\tau
\le
\dbE^{\dbQ}\left[
\Delta Y_{\tau_k^*}
\Bigm|
\cF_\tau
\right].
\end{equation}
For every stopping time $\tau\in\cT$,
\[
|Y_\tau^\beta|
\le
\dbE^{\dbQ}\left[
|\xi|
+
\int_0^T|g(t,X,\beta_t)|\,\dd t
\Bigm|
\cF_\tau
\right].
\]
Hence the family
\[
\left\{
Y_\tau^\beta:\tau\in\cT
\right\}
\]
is uniformly integrable under $\dbQ$. Moreover, by continuity,
\[
\Delta Y_{\tau_k^*}
\longrightarrow
\Delta Y_T=0,
\qquad
\dbQ\text{-a.s.}
\]
Thus
\[
\Delta Y_{\tau_k^*}
\longrightarrow0
\qquad
\text{in }L^1(\dbQ).
\]
Together with
\eqref{eq-BSDE-8},
this yields
\[
\Delta Y_\tau\le0,
\qquad
\dbQ\text{-a.s.}
\]
Therefore,
\[
Y_\tau
\le
\dbE^{\dbQ}\left[
\xi
+
\int_\tau^Tg(t,X,\beta_t)\,\dd t
\Bigm|
\cF_\tau
\right].
\]
Since $\dbQ\sim\dbP$, the same inequality holds $\dbP$-a.s.
\end{proof}
\begin{lem}[Fixed point for the BSDE-induced map]
\label{lem-BSDE-fp}
Let
\[
k:\cP(\Omega)\times\Omega\longrightarrow\dbR
\]
be bounded and satisfy the following conditions:
\begin{enumerate}
\item
For every $\dbQ\in\cP(\Omega)$, the map \(\omega\mapsto k(\dbQ,\omega)\) is $\cF_T$-measurable.

\item
The map $\dbQ\mapsto k(\dbQ,\cdot)$ is continuous with respect to the
weak topology on $\cP(\Omega)$, uniformly in the path variable. Namely,
whenever $\dbQ^n\Rightarrow\dbQ$ weakly in $\cP(\Omega)$,
\[
\sup_{\omega\in\Omega}
\left|
k(\dbQ^n,\omega)-k(\dbQ,\omega)
\right|
\longrightarrow0.
\]
\end{enumerate}

Then there exist a probability measure
$\overline{\dbQ}\in\cP(\Omega)$ such that \(\operatorname{BSDE} \left( \cH_g, {k(\overline{\dbQ},\cdot)} \right).\) have the unique solution 
$(\overline Y,\overline Z)$. 
Moreover, if we set
\[
\overline\beta_t
:=
\argmin_{b\in\dbR^d}
\left\{
b\cdot\overline Z_t
+
g(t,X,b)
\right\},
\]
then
\[
\overline\beta\in\BMO,
\qquad
\frac{\dd\overline{\dbQ}}{\dd\dbP}
=
D(\overline\beta).
\]
In particular, $\overline{\dbQ}$ is a fixed point of the map induced
by the BSDE solution and the associated optimal feedback control.
\end{lem}

The terminal law derivative creates a fixed-point problem. The estimates below separate continuity of the optimal response from compactness of its range.
\begin{proof}
\phantomsection\label{proof-bsde-fixed-point}

\revstep{1}{A compact convex set of laws}
For $K>0$, define
\[
\mathcal Q_K^0
:=
\left\{
\dbQ\in\cP(\Omega):
\begin{array}{l}
\text{there exists }
\beta\in\BMO
\text{ such that }
\|\beta\|_{\BMO}\le K,
\\[1mm]
\displaystyle
\frac{\dd\dbQ}{\dd\dbP}
=
D(\beta)
\end{array}
\right\}.
\]
By Lemma~\ref{lem-BMO-reverse-holder}, the family
$\mathcal Q_K^0$ is tight with respect to the weak topology.

To apply the Schauder--Tychonoff fixed point theorem
\cite[Theorem~2.3]{cobzas2006fixed}, define \(\mathcal Q_K := \overline{\operatorname{co}(\mathcal Q_K^0)}^{w},\) where the closure is taken with respect to the topology of weak convergence in
$\mathcal P(\Omega)$.
Since $\mathcal Q_K^0$ is tight, for every $\varepsilon>0$ there exists a compact set
$C_\varepsilon\subset\Omega$ such that
\[
\inf_{\dbQ\in\mathcal Q_K^0}
\dbQ(C_\varepsilon)
\ge 1-\varepsilon.
\]
For any $\dbQ\in\operatorname{co}(\mathcal Q_K^0)$, there exist
$\dbQ_1,\ldots,\dbQ_N\in\mathcal Q_K^0$ and
$\alpha_1,\ldots,\alpha_N\ge0$ with $\sum_{i=1}^N\alpha_i=1$ such that
\[
\dbQ=\sum_{i=1}^N\alpha_i\dbQ_i.
\]
Hence
\[
\dbQ(C_\varepsilon)
=
\sum_{i=1}^N\alpha_i\dbQ_i(C_\varepsilon)
\ge 1-\varepsilon,
\]
which shows that $\operatorname{co}(\mathcal Q_K^0)$ is tight.
By Prokhorov's theorem
\cite[Theorems~5.1--5.2]{billingsley1999convergence},
$\operatorname{co}(\mathcal Q_K^0)$ is relatively compact for the topology of weak convergence. Therefore,
\[
\mathcal Q_K
=
\overline{\operatorname{co}(\mathcal Q_K^0)}^{\,w}
\]
is weakly compact. Moreover, $\mathcal Q_K$ is convex, since the weak closure of a convex set is convex.

\revstep{2}{The BSDE map and its invariant set}
Let $\dbQ^K\in\mathcal Q_K$. Since $k$ is bounded,
$k(\dbQ^K,\cdot)$ is a measurable terminal condition {$\cF_T$-measurable and uniformly bounded}.
Therefore, Lemma~\ref{lem-BSDE-solution} implies that \(\operatorname{BSDE} \left( \cH_g, k(\dbQ^K,\cdot) \right)\) admits a unique solution, which we denote by $(Y^K,Z^K)$.

Define \(h^K_t := \argmin_{b\in\dbR^d} \left\{ b\cdot Z_t^K+g(t,X,b) \right\}.\) For each $(t,\omega)$, the map \(b\mapsto b\cdot Z_t^K(\omega)+g(t,\omega,b)\) is strongly convex. Hence $h^K$ exists and is unique and is
characterized by
\begin{equation}
\label{eq-BSDE-9}
\nabla g(t,X,h^K_t)
=-Z_t^K.
\end{equation}
Lemma~\ref{lem-BSDE-solution} also shows that there exists a constant
$C_B>0$, independent of $\dbQ^K\in\mathcal Q_K$, such that
\[
\|Y^K\|_{\mathcal S^\infty}
+
\|Z^K\|_{\BMO}
\le C_B.
\]
\eqref{eq-BSDE-1.1} and \eqref{eq-BSDE-9} imply
\[
|h^K_t-b_g^0(t,X)|
\le
{ c_g^{-1} }|Z_t^K|.
\]
Therefore,
\[
\|h^K\|_{\BMO}
\le
\|b_g^0(\cdot,X)\|_{\BMO}
+
{ c_g^{-1} }\|Z^K\|_{\BMO}.
\]
Together with Lemma~\ref{lem-BSDE-b0}, we obtain
\[
\|h^K\|_{\BMO}
\le
\left(
\gamma^2T
+
2\gamma\|\alpha\|_{p,T}
\right)^{1/2}
+
{ c_g^{-1} }C_B.
\]
{Thus, with the constant independent of $K$ defined by
\[
 K_0:=\bigl(\gamma^2T+2\gamma\|\alpha\|_{p,T}\bigr)^{1/2}
       +c_g^{-1}C_B,
\]
we have}
\begin{equation}
\label{eq-BSDE-10}
\|h^K\|_{\BMO}
\le K_0.
\end{equation}
Choose $K\ge K_0$ and define $
\Phi:\mathcal Q_K\rightarrow\cP(\Omega)$
by
\[
\frac{\dd\Phi(\dbQ^K)}{\dd\dbP}
=
D(h^K).
\]
The preceding estimate yields
\[
\Phi(\mathcal Q_K)
\subset
\mathcal Q_K^0.
\]
Hence $\Phi$ is well-defined from $\mathcal Q_K$ into itself.

\revstep{3}{Continuity of the BSDE solution}
It remains to prove that $\Phi$ is continuous on $\mathcal Q_K$
endowed with the weak topology.
Let
$(\dbQ^n)_{n\ge1}\subset\mathcal Q_K$
and $\dbQ^\infty\in\mathcal Q_K$ be such that
\[
\dbQ^n\Rightarrow\dbQ^\infty
\qquad
\text{weakly in }\cP(\Omega).
\]
Set
\[
k^n:=k(\dbQ^n,\cdot),
\qquad
k^\infty:=k(\dbQ^\infty,\cdot).
\]
By the {continuity assumption on $k$},
\[
\|k^n-k^\infty\|_{L^\infty(\dbP)}
\le
\sup_{\omega\in\Omega}
\left|
k(\dbQ^n,\omega)
-
k(\dbQ^\infty,\omega)
\right|
\longrightarrow0.
\]

Let $(Y^n,Z^n)$ and $(Y,Z)$ be the unique solutions of
\[
{\operatorname{BSDE}(\cH_g,k^n)},
\qquad
{\operatorname{BSDE}(\cH_g,k^\infty)},
\]
respectively. Following the stability approach for quadratic BSDEs in
\cite[Proposition~7]{briand2008quadratic}, set
\[
\Delta Y^n:=Y^n-Y,
\qquad
\Delta Z^n:=Z^n-Z,
\qquad
\Delta k^n:=k^n-k^\infty.
\]
 Since
$\|k^n-k^\infty\|_\infty\to0$, Lemma~\ref{lem-BSDE-solution} shows that sequence $(Z^n)_{n\ge1}$ is uniformly
bounded in $\BMO$. Define
\[
\lambda_t^n
:=
\int_0^1
h^g
\bigl(
t,X,Z_t+\theta\Delta Z_t^n
\bigr)
\,\dd\theta.
\]
By \eqref{eq-BSDE-3},
\[
\cH_g(t,X,Z_t^n)
-
\cH_g(t,X,Z_t)
=
\lambda_t^n\cdot\Delta Z_t^n.
\]
Moreover, by \eqref{eq-BSDE-2},
Lemma~\ref{lem-BSDE-b0}, and the uniform BMO bound on $Z^n$,
\[
\sup_{n\ge1}\|\lambda^n\|_{\BMO}<\infty.
\]
In particular,
\[
M_t^n:=\int_0^t\lambda_s^n\cdot\dd X_s
\]
is a BMO martingale. {Hence the stochastic exponential $(D_t(\lambda^n))_{t\le T}$ is a uniformly integrable martingale}
\cite[Section~3.3]{kazamaki1994continuous}. We may therefore define an
equivalent probability measure $\widetilde{\dbQ^n}$ by
\[
\frac{\dd\widetilde{\dbQ^n}}{\dd\dbP}
=
D(\lambda^n).
\]
By Girsanov's theorem,
\[
W_t^n
:=
X_t-\int_0^t\lambda_s^n\,\dd s,
\qquad t\in[0,T],
\]
is a Brownian motion under $\widetilde{\dbQ^n}$. Consequently, the difference
equation reduces to
\[
\Delta Y_t^n
=
{\Delta k^n}
-
\int_t^T
\Delta Z_s^n\,\dd W_s^n.
\]
Hence
\[
\Delta Y_t^n
=
\dbE^{\widetilde{\dbQ^n}}
\left[
{\Delta k^n}\mid\cF_t
\right],
\]
and therefore
\begin{equation}\label{eq-BSDE-11}
\|\Delta Y^n\|_{\mathcal S^\infty}
\le
\|{\Delta k^n}\|_\infty.
\end{equation}
Furthermore, for every stopping time $\tau\in\cT$, It\^o's formula
gives
\[
\dbE^{\widetilde{\dbQ^n}}
\left[
\int_\tau^T
|\Delta Z_s^n|^2\,\dd s
\Bigm|\cF_\tau
\right]
\le
\|{\Delta k^n}\|_\infty^2,
\]
thus
\begin{equation}\label{eq-BSDE-12}
\|\Delta Z^n\|_{\BMO(\widetilde{\dbQ^n})}
\le
\|{\Delta k^n}\|_\infty.  
\end{equation}
Since $(\lambda^n)_{n\ge1}$ is uniformly bounded in
$\BMO$, by Lemma~\ref{lem-BMO-reverse-holder}.2 and Lemma~\ref{lem-BMO-reverse-holder}.3, there exist a constant $L_Z>0$ independent of $n$, such that
\begin{equation}\label{eq-BSDE-13}
\|\Delta Z^n\|_{\BMO}
\le
L_Z
\|\Delta Z^n\|_{\BMO(\widetilde{\dbQ^n})}. 
\end{equation}
Together with \eqref{eq-BSDE-11}, \eqref{eq-BSDE-12} and \eqref{eq-BSDE-13}, we obtain that
\begin{equation*}
\lim_{n\rightarrow+\infty}(\|\Delta Y^n\|_{\mathcal S^\infty}
+
\|\Delta Z^n\|_{\BMO})=0
\end{equation*}

\revstep{4}{Continuity of the induced law}
We set
\[
h^n_t
:=
\argmin_{b\in\dbR^d}
\left\{
b\cdot Z_t^n+g(t,X,b)
\right\}.
\]
By the first-order optimality condition,
\[
\nabla g(t,X,h^n_t)
=-Z_t^n.
\]
Since $\nabla g(t,X,\cdot)$ is uniformly strongly monotone, its
inverse is Lipschitz continuous with Lipschitz constant ${ c_g^{-1} }$.
Hence
\[
|h^n_t-h^g_t|
\le
{ c_g^{-1} }|Z_t^n-Z_t|.
\]
Therefore,
\[
\|h^n-h^g\|_{\BMO}
\le
{ c_g^{-1} }
\|Z^n-Z\|_{\BMO}
\longrightarrow0.
\]
Moreover,
\eqref{eq-BSDE-10}
shows that
$h^n$ and $h^g$ are uniformly bounded in
$\BMO$. Hence Lemma~\ref{lem-BMO-lip} yields
\[
\|\Phi({\dbQ^n})-{\Phi(\dbQ^\infty)}\|_{\mathrm{TV}}
=
\left\|
D(h^n)
-
D(h^g)
\right\|_{L^1}
\longrightarrow0.
\]
In particular,
\[
\Phi({\dbQ^n})
\Rightarrow
{\Phi(\dbQ^\infty)}
\qquad
\text{weakly in }\cP(\Omega).
\]
Therefore $\Phi$ is continuous from $\mathcal Q_K$ into itself.

\revstep{5}{The fixed point}
Thus $\mathcal Q_K$ is a nonempty compact convex subset of
$\cP(\Omega)$ endowed with the weak topology, and
$\Phi:\mathcal Q_K\to\mathcal Q_K$ is continuous.
By the Schauder--Tychonoff fixed-point theorem, there exists
$\overline{\dbQ}\in\mathcal Q_K$ such that
\[
\Phi(\overline{\dbQ})
=
\overline{\dbQ}.
\]
Since \(\Phi(\mathcal Q_K) \subset \mathcal Q_K^0,\) the fixed point actually belongs to $\mathcal Q_K^0$.
Let $(\overline Y,\overline Z)$ be the unique solution of \(\operatorname{BSDE} \left( \cH_g, {k(\overline{\dbQ},\cdot)} \right),\) and define
\[
\overline\beta_t
:=
h^g(t,X,\overline Z_t)
=
\argmin_{b\in\dbR^d}
\left\{
b\cdot\overline Z_t
+
g(t,X,b)
\right\}.
\]
By the definition of the map $\Phi$,
\[
\frac{\dd\Phi(\overline{\dbQ})}{\dd\dbP}
=
D(\overline\beta).
\]
Since $\overline{\dbQ}$ is a fixed point of $\Phi$, this implies
\[
\frac{\dd\overline{\dbQ}}{\dd\dbP}
=
D(\overline\beta).
\]
Finally, \eqref{eq-BSDE-10} gives that $
\|\overline\beta\|_{\BMO}
\le K_0$.
\end{proof}

\Needspace{8\baselineskip}
\subsection{Path-entropy calculus and duality}\label{sec-entropy-proximal-proofs}
With admissible controls available, the argument turns from stochastic estimates to convexity. The order is derivative, strict convexity, minimiser, then proximal iteration.

Throughout the proximal proofs we write
$J^{h,n}(\dbQ)=J(\dbQ)+h^{-1}H_\psi(\dbQ\mid\dbQ^{h,n})$,
$\mathcal H^{h,n}=\cH_{g^{h,n}}$, and regard
$\psi_{h,n}$ and $\psi^{h,n}$ as the same pointwise Bregman penalty.

\begin{lem}[Linear derivative of $H_\varphi$]
\label{lem-linear-derivative}
Let Assumptions~\ref{ass-rc} hold.
Suppose that
\[
\max_{\varphi\in\{\ell,\psi\}}
\left\{
\|\nabla\varphi(0)\|_\infty
+
\|\varphi(0)\|_\infty 
\right\}
\le C_0
\]
and fix
\(
\varphi\in\{\ell,\psi\}.
\)
For $(\dbP^0,\alpha^0)\in\overline{\cA}$, we define the linear
derivative, whenever it exists, by
\[
\delta H_\varphi(\dbP^0,\cdot)
:=
\int_0^T \varphi(\alpha_t^0)\,\dd t
+
\int_0^T
\nabla\varphi(\alpha_t^0)
\cdot
\bigl(\dd X_t-\alpha_t^0\,\dd t\bigr).
\]
If $(\dbP^0,\alpha^0)\in\cA_\varphi$, then for every
$(\dbP^1,\alpha^1)\in\cA$,
$\delta H_\varphi(\dbP^0,\cdot)$ is integrable under both
$\dbP^0$ and $\dbP^1$, and satisfies
\[
\left\langle
\delta H_\varphi(\dbP^0,\cdot),
\dbP^1-\dbP^0
\right\rangle
=
\lim_{s\to0+}
\frac{\dd}{\dd s}
H_\varphi\bigl((1-s)\dbP^0+s\dbP^1\bigr).
\]
Moreover,
\[
H_\varphi(\dbP^1\,|\,\dbP^0)
=
H_\varphi(\dbP^1)
-
H_\varphi(\dbP^0)
-
\left\langle
\delta H_\varphi(\dbP^0,\cdot),
\dbP^1-\dbP^0
\right\rangle.
\]
\end{lem}

The drift of a mixture is not the mixture of its drifts. Keeping track of the density process is what makes the path-space derivative and the Bregman identity agree.
\begin{proof}
\phantomsection\label{proof-entropy-derivative}

\revstep{1}{Mixture densities and the formal derivative}
For $\varphi\in\{\ell,\psi\}$, we use \(c_\varphi\) to denote the lower bound for its Hessian. For $i\in\{0,1\}$, let
\[
D_t^i
:=
\dbE\left[
\frac{\dd\dbP^i}{\dd\dbP}
\Bigm|\cF_t
\right],
\qquad
t\in[0,T],
\]
be the density process of $\dbP^i$ with respect to $\dbP$.
For $s\in[0,1]$, set \(\dbP^s:=(1-s)\dbP^0+s\dbP^1\), \(D_t^s:=(1-s)D_t^0+sD_t^1.\) Since \(\dd D_t^i = D_t^i\alpha_t^i\cdot\dd X_t\), \( i\in\{0,1\}\), the drift associated with $\dbP^s$ is
\[
\alpha_t^s
=
\frac{
(1-s)D_t^0\alpha_t^0+sD_t^1\alpha_t^1
}{
D_t^s
}.
\]
Equivalently, setting $
\lambda_t^s
:=
\tfrac{sD_t^1}{D_t^s}$,
we have
\begin{equation}
\label{eq-derivative-1}
\alpha_t^s
=
(1-\lambda_t^s)\alpha_t^0
+
\lambda_t^s\alpha_t^1,
\qquad
0\le\lambda_t^s\le1.
\end{equation}
By the definition of $H_\varphi$ and the tower property,
\[
H_\varphi(\dbP^s)
=
\dbE\left[
\int_0^T
D_t^s\varphi(\alpha_t^s)\,\dd t
\right].
\]
A direct calculation then gives
\begin{equation}\label{eq-derivative-2}
\frac{\dd}{\dd s} \left( D_t^s\varphi(\alpha_t^s) \right) = (D_t^1-D_t^0)\varphi(\alpha_t^s) + \frac{D_t^0D_t^1}{D_t^s} \nabla\varphi(\alpha_t^s) \cdot (\alpha_t^1-\alpha_t^0).
\end{equation}

\revstep{2}{Integrability under both laws}
We next justify the differentiation under the integral sign.
For $0\le s\le1/2$,
\begin{equation*}
D_t^0\lambda_t^s
=
\frac{sD_t^0D_t^1}{D_t^s}
\le
D_t^1,
\qquad
\frac{D_t^0D_t^1}{D_t^s}
\le
2D_t^1.
\end{equation*}
By Assumption~\ref{ass-ini} and Assumption~\ref{ass-rc},
\begin{equation*}
\varphi(b) \ge \varphi(0) + \nabla \varphi(0)\cdot b + \frac{c_\varphi}{2}|b|^2 \ge \varphi(0) - \frac{|\nabla\varphi(0)|^2}{2c_\varphi}.
\end{equation*}
Therefore,
\begin{equation}
\label{eq-derivative-3}
\|\varphi^-(b)\|_\infty\le C_0+\frac{C^2_0}{2c_\varphi}:=C_{\varphi,-}
\end{equation}
Since
$(\dbP^0,\alpha^0)\in\cA_\varphi$
and
$(\dbP^1,\alpha^1)\in\cA$,
the definition of the admissible classes and
Lemma~\ref{lem-BMO-reverse-holder} yield
\[
\alpha^0,\alpha^1,
\nabla\varphi(\alpha^0)
\in
\BMO(\dbP^1),
\]
and
\[
\sqrt{\varphi^+(\alpha^0)}
\in
\BMO(\dbP^0)\cap\BMO(\dbP^1).
\]
In particular,
\begin{equation}
\label{eq-derivative-4}
\dbE^{\dbP^1}
\left[
\int_0^T
|\varphi(\alpha_t^0)|\,\dd t
\right]
<\infty.
\end{equation}
Moreover, by the integrability requirement in the definition of
$\cA$, together with Lemma~\ref{lem-BMO-reverse-holder},
\[
\dbE^{\dbP^1}
\left[
\int_0^T
|\ell(\alpha_t^1)|\,\dd t
\right]
<\infty.
\]
If $\varphi=\ell$, this already gives
\[
\dbE^{\dbP^1}
\left[
\int_0^T
\varphi^+(\alpha_t^1)\,\dd t
\right]
<\infty.
\]

Suppose now that $\varphi=\psi$.
By Assumption~\ref{ass-rc},
\[
\lambda_{\ell|\psi}
\Bigl(
\psi(b)-\psi(a)-\nabla\psi(a)\cdot(b-a)
\Bigr)
\le
\ell(b)-\ell(a)-\nabla\ell(a)\cdot(b-a),
\]
for every $a,b\in\dbR^d$.
Applying this with \(a=\alpha_t^0, \qquad b=\alpha_t^1,\) we obtain
\begin{equation*}
\psi(\alpha_t^1) \le \psi(\alpha_t^0) + \nabla\psi(\alpha_t^0) \cdot (\alpha_t^1-\alpha_t^0) + \frac1{\lambda_{\ell|\psi}} \Bigl( \ell(\alpha_t^1) -\ell(\alpha_t^0) -\nabla\ell(\alpha_t^0) \cdot (\alpha_t^1-\alpha_t^0) \Bigr).
\end{equation*}
By Assumption~\ref{ass-rc},
the definition of $\cA_\psi$, and
Lemma~\ref{lem-BMO-reverse-holder},
\[
\nabla\ell(\alpha^0),
\,
\nabla\psi(\alpha^0),
\,
\alpha^1-\alpha^0
\in
\BMO(\dbP^1),
\]
while the remaining terms are integrable under $\dbP^1$ by the
definition of $\cA$ and
\eqref{eq-derivative-4}.
Together with Assumption~\ref{ass-ini}, this yields
\[
\dbE^{\dbP^1}
\left[
\int_0^T
|\psi(\alpha_t^1)|\,\dd t
\right]
<\infty.
\]
Thus, in either case,
\begin{equation}
\label{eq-derivative-5}
\dbE^{\dbP^1}
\left[
\int_0^T
\varphi^+(\alpha_t^1)\,\dd t
\right]
<\infty.
\end{equation}

\revstep{3}{Domination of the cost term}
We now construct an integrable majorant for the first term in
\eqref{eq-derivative-2}.
By convexity of $\varphi$ and
\eqref{eq-derivative-1},
\[
\varphi(\alpha_t^s)
\le
(1-\lambda_t^s)\varphi(\alpha_t^0)
+
\lambda_t^s\varphi(\alpha_t^1).
\]
Combining this with
\eqref{eq-derivative-3} gives
\[
|\varphi(\alpha_t^s)|
\le
C_{\varphi,-}
+
(1-\lambda_t^s)\varphi^+(\alpha_t^0)
+
\lambda_t^s\varphi^+(\alpha_t^1).
\]
Therefore, for $0\le s\le1/2$,
\begin{equation}
\label{eq-derivative-6}
|D_t^1-D_t^0|
\,|\varphi(\alpha_t^s)|
\le
(D_t^0+D_t^1)
\left(
C_{\varphi,-}
+
\varphi^+(\alpha_t^0)
\right)
+
2D_t^1\varphi^+(\alpha_t^1).
\end{equation}
The right-hand side is integrable with respect to
$\dd t\otimes\dd\dbP$ by the definition of $\cA_\varphi$,
Lemma~\ref{lem-BMO-reverse-holder},
and
\eqref{eq-derivative-5}.

\revstep{4}{Domination of the gradient term}
For the second term in
\eqref{eq-derivative-2}, set \(q_t^s := \nabla\varphi(\alpha_t^s) \cdot (\alpha_t^1-\alpha_t^0).\) Since $\nabla\varphi$ is monotone,
\[
q_t^s
\ge
q_t^0
:=
\nabla\varphi(\alpha_t^0)
\cdot
(\alpha_t^1-\alpha_t^0).
\]
Moreover, convexity gives
\[
\varphi(\alpha_t^1)
\ge
\varphi(\alpha_t^s)
+
(1-\lambda_t^s)q_t^s,
\]
and therefore
\[
(q_t^s)^+
\le
\frac{
\varphi^+(\alpha_t^1)
+
C_0+\frac{C^2_0}{c_\varphi}
}{
1-\lambda_t^s
}.
\]
Since
\[
1-\lambda_t^s
=
\frac{(1-s)D_t^0}{D_t^s},
\]
we obtain, for $0\le s\le1/2$,
\begin{align}
\frac{D_t^0D_t^1}{D_t^s}|q_t^s|
&\le
2D_t^1
\left(
|q_t^0|
+
\varphi^+(\alpha_t^1)
+
C_0+\frac{C^2_0}{c_\varphi}
\right).
\label{eq-derivative-7}
\end{align}
The right-hand side is integrable. Indeed,
by conditional Cauchy--Schwarz,
\[
\dbE^{\dbP^1}
\left[
\int_0^T
|q_t^0|\,\dd t
\right]
\le
\|\nabla\varphi(\alpha^0)\|_{\BMO(\dbP^1)}
\,
\|\alpha^1-\alpha^0\|_{\BMO(\dbP^1)}
<\infty,
\]
while the remaining terms are intergrable.

\revstep{5}{Differentiation under the integral}
Since $D_t^0>0$,
\[
\lambda_t^s\longrightarrow0,
\qquad
\alpha_t^s\longrightarrow\alpha_t^0,
\qquad
\frac{D_t^0D_t^1}{D_t^s}
\longrightarrow D_t^1,
\qquad
s\downarrow0.
\]
Hence, by the continuity in the control variable from
Assumption~\ref{ass-rc}, together with
\eqref{eq-derivative-6} and
\eqref{eq-derivative-7}, the dominated convergence
theorem yields
\begin{align*}
\left.\frac{\dd}{\dd s}H_\varphi\bigl((1-s)\dbP^0+s\dbP^1\bigr)\right|_{s=0+}=\dbE\left[\int_0^T\Bigl((D_t^1-D_t^0)\varphi(\alpha_t^0)+D_t^1\nabla\varphi(\alpha_t^0)\cdot(\alpha_t^1-\alpha_t^0)\Bigr)\,\dd t\right].
\end{align*}

\revstep{6}{Linear derivative and Bregman identity}
It remains to identify the right-hand side with the announced
linear derivative. Set
\[
\Gamma
:=
\delta H_\varphi(\dbP^0,\cdot)
=
\int_0^T\varphi(\alpha_t^0)\,\dd t
+
\int_0^T
\nabla\varphi(\alpha_t^0)
\cdot
\bigl(
\dd X_t-\alpha_t^0\,\dd t
\bigr).
\]
The preceding estimates show that
\[
\Gamma
\in
L^1(\dbP^0)\cap L^1(\dbP^1).
\]
Moreover,
$\nabla\varphi(\alpha^0)$ belongs to
$\BMO(\dbP^0)\cap\BMO(\dbP^1)$
by the definition of $\cA_\varphi$ and
Lemma~\ref{lem-BMO-reverse-holder}.
Therefore the corresponding stochastic integrals are true martingales,
and
\[
\dbE^{\dbP^0}[\Gamma]
=
\dbE^{\dbP^0}
\left[
\int_0^T
\varphi(\alpha_t^0)\,\dd t
\right],
\]
whereas
\[
\dbE^{\dbP^1}[\Gamma]
=
\dbE^{\dbP^1}
\left[
\int_0^T
\left(
\varphi(\alpha_t^0)
+
\nabla\varphi(\alpha_t^0)
\cdot
(\alpha_t^1-\alpha_t^0)
\right)\,\dd t
\right].
\]
Consequently,
\begin{equation*}
\left. \frac{\dd}{\dd s} H_\varphi\bigl((1-s)\dbP^0+s\dbP^1\bigr) \right|_{s=0+} = \dbE\left[ (D_T^1-D_T^0)\Gamma \right] = \left\langle \delta H_\varphi(\dbP^0,\cdot), \dbP^1-\dbP^0 \right\rangle.
\end{equation*}
Finally, by the definition of the relative $\varphi$-entropy,
\begin{align*}
H_\varphi(\dbP^1\,|\,\dbP^0)
&=
\dbE^{\dbP^1}
\left[
\int_0^T
\Bigl(
\varphi(\alpha_t^1)
-
\varphi(\alpha_t^0)
-
\nabla\varphi(\alpha_t^0)
\cdot
(\alpha_t^1-\alpha_t^0)
\Bigr)\,\dd t
\right]
\\
&=
H_\varphi(\dbP^1)
-
H_\varphi(\dbP^0)
-
\left\langle
\delta H_\varphi(\dbP^0,\cdot),
\dbP^1-\dbP^0
\right\rangle.
\end{align*}
This completes the proof.
\end{proof}

\begin{lem}[Forward coercivity of the objective]
\label{lem-forward-coercivity}
Assume Assumptions~\ref{ass-rc} and~\ref{ass-pe}.
Suppose that $(\dbQ',\beta')\in\cA_\ell\cap\cA_\psi$ and that,
for some deterministic constant $c'\in\dbR$,
\begin{equation}
\label{eq-JKO-13}
\delta H_\ell(\dbQ',\cdot)
+
\delta F(\dbQ',\cdot)
=
c',
\qquad
\dbP\text{-a.s.}
\end{equation}
Then, for every $(\dbQ,\beta)\in\cA_\ell\cap\cA_\psi$,
\begin{equation*}
J(\dbQ)-J(\dbQ')
\ge
H_\ell(\dbQ\,|\,\dbQ')
\ge
\lambda_{\ell|\psi}
H_\psi(\dbQ\,|\,\dbQ').
\end{equation*}
\end{lem}

\begin{proof}
By Lemma~\ref{lem-linear-derivative},
\[
H_\ell(\dbQ)-H_\ell(\dbQ')
=
H_\ell(\dbQ\,|\,\dbQ')
+
\left\langle
\delta H_\ell(\dbQ',\cdot),
\dbQ-\dbQ'
\right\rangle.
\]
Hence,
\begin{align*}
J(\dbQ)-J(\dbQ')
&=
H_\ell(\dbQ\,|\,\dbQ')
+
F(\dbQ)-F(\dbQ')
-
\left\langle
\delta F(\dbQ',\cdot),
\dbQ-\dbQ'
\right\rangle
\\
&\quad+
\left\langle
\delta H_\ell(\dbQ',\cdot)
+
\delta F(\dbQ',\cdot),
\dbQ-\dbQ'
\right\rangle.
\end{align*}
The last term vanishes by
\eqref{eq-JKO-13}, since
$\dbQ-\dbQ'$ has total mass zero. The convexity of $F$ therefore gives
\[
J(\dbQ)-J(\dbQ')
\ge
H_\ell(\dbQ\,|\,\dbQ').
\]
The second inequality follows from
\eqref{eq-relative-ell-psi}.
\end{proof}

\begin{lem}\label{lem-H-convex}
Let Assumption~\ref{ass-rc} and Assumption~\ref{ass-ini} hold and fix \(\varphi\in\{\ell,\psi\}.\) {Then $H_\varphi$ is strictly convex on its finite-cost domain
of equivalent Wiener laws with finite drift energy. More precisely,
for $0<s<1$ and two distinct laws $\dbQ,\dbQ'\sim\dbP$ with
finite $\varphi$-cost and finite drift energy,}
\[
H_\varphi((1-s)\dbQ+s\dbQ')
<
(1-s)H_\varphi(\dbQ)
+
sH_\varphi(\dbQ').
\]
\end{lem}

The density-weighted drift formula reveals convexity. Equality forces the two drifts, and hence their stochastic-exponential laws, to coincide.
\begin{proof}
\phantomsection\label{proof-entropy-convexity}

\revstep{1}{Convexity along mixtures}
We first establish a quantitative convexity estimate for $H_\varphi$. {Let $(\dbQ,\beta)$ and $(\dbQ',\beta')$ have the finite costs and energies specified in the statement, and define} \(\dbQ^{(s)} := (1-s)\dbQ+s\dbQ' \), \( s\in[0,1].\) Let $D_t$ and $D_t'$ denote the density processes of $\dbQ$ and
$\dbQ'$, respectively, with respect to $\dbP$. Then the density
process of $\dbQ^{(s)}$ is \(D_t^{(s)} = (1-s)D_t+sD_t',\) and the corresponding drift is given by
\[
\beta_t^{(s)}
=
\frac{
(1-s)D_t\beta_t+sD_t'\beta_t'
}{
D_t^{(s)}
}.
\]
Equivalently,
\[
\beta_t^{(s)}
=
\bigl(1-\lambda_t^{(s)}\bigr)\beta_t
+
\lambda_t^{(s)}\beta_t',
\qquad
\lambda_t^{(s)}
:=
\frac{sD_t'}{D_t^{(s)}}.
\]
Since $D_t,D_t'>0$, we have \(\lambda_t^{(s)}\in[0,1].\) By the uniform strong convexity of $\varphi$, for every
$b,b'\in\dbR^d$ and $\lambda\in[0,1]$,
\[
\varphi\bigl((1-\lambda)b+\lambda b'\bigr)
\le
(1-\lambda)\varphi(b)
+
\lambda\varphi(b')
-
\frac{c_\varphi}{2}
\lambda(1-\lambda)|b-b'|^2.
\]
Applying this inequality with \(b=\beta_t \), \( b'=\beta_t' \), \( \lambda=\lambda_t^{(s)}\), and multiplying by $D_t^{(s)}$, we obtain
\begin{equation*}
D_t^{(s)}\varphi(\beta_t^{(s)}) \le (1-s)D_t\varphi(\beta_t) + sD_t'\varphi(\beta_t') - \frac{c_\varphi}{2}s(1-s) \frac{D_tD_t'}{D_t^{(s)}} |\beta_t-\beta_t'|^2.
\end{equation*}
{The same density-weighted inequality for $|b|^2$ shows that
$\dbQ^{(s)}$ has finite drift energy. Equivalence follows from
$D_T^{(s)}>0$. Add the deterministic lower bound of $\varphi$ to
its values in the preceding inequality. All resulting cost terms
are nonnegative, so Tonelli's theorem applies and proves finite
$\varphi$-cost of the mixture. No $\varphi$-BMO condition on its
mixing drift is used. Subtracting the added constant, integration
gives}
\begin{equation}\label{eq-JKO-strong-convexity-H}
H_\varphi(\dbQ^{(s)}) \le (1-s)H_\varphi(\dbQ) + sH_\varphi(\dbQ') - \frac{c_\varphi}{2}s(1-s) \dbE\left[ \int_0^T \frac{D_tD_t'}{D_t^{(s)}} |\beta_t-\beta_t'|^2\,\dd t \right].
\end{equation}

\revstep{2}{Strictness of the inequality}
If the correction term in \eqref{eq-JKO-strong-convexity-H} vanishes, then, since \(D_t,D_t'>0\), we have

$$ \beta=\beta' \qquad dt\otimes dP\text{-a.e.} $$
Hence the corresponding stochastic exponential density processes solve the same stochastic differential equation with the same initial condition, and therefore \(D=D'\), so that \(Q=Q'\). This contradicts the assumption that the two pairs are distinct. Hence the correction term is strictly positive and \(H_\varphi\) is strictly convex.
\end{proof}

\begin{lem}[Dual representation and weak lower semicontinuity of
$H_\ell$]
\label{lem-H-weak-lsc}

Let Assumption~\ref{ass-rc} and Assumption~\ref{ass-ini} hold. 
Let $p>1$ and $q:=p/(p-1)$. Let
$(\dbQ^k,\beta^k)_{k\ge1}\subset\cA_\ell$ and
$(\dbQ,\beta)\in\cA$ satisfy
\[
D(\beta^k)
\rightharpoonup
D(\beta)
\qquad
\text{weakly in }L^p(\dbP).
\]
Then
\begin{equation}
\label{eq-uni-1}
H_\ell(\dbQ)
\le
\liminf_{k\to\infty}H_\ell(\dbQ^k).
\end{equation}
\end{lem}

\readingguide{The dual formula and weak lower semicontinuity are common path-entropy tools. In the Newton limit they complement the strong convergence of the entropy coordinate, without a separate compactness construction.}
\begin{proof}
\phantomsection\label{proof-entropy-lsc}

\revstep{1}{Dual upper bound}
We first show that, for every $(\dbQ,\beta)\in\overline{\cA}$ such that \(D(\beta)\in L^p(\dbP),\) one has
\[
H_\ell(\dbQ)
=
\sup_{Z\in\mathcal Z_q}
\dbE\left[D(\beta)\Xi^Z\right].
\]
For $(t,\omega)\in[0,T]\times\Omega$, let
\[
\ell^*(t,\omega,z)
:=
\sup_{b\in\mathbb R^d}
\left\{
z\cdot b-\ell(t,\omega,b)
\right\}.
\]
Let $\mathcal Z_q$ denote the set of bounded
$\mathbb F$-progressively measurable processes $Z$ such that
\[
\Xi^Z
:=
\int_0^T Z_t\cdot \dd X_t
-
\int_0^T\ell^*(Z_t)\,\dd t
\in L^q(\dbP).
\]
Since $Z$ is bounded,
\[
\int_0^\cdot Z_t\cdot\dd W_t^\dbQ
\]
is a square-integrable $\dbQ$-martingale. In particular, its terminal
value has zero $\dbQ$-expectation. Since \(\dd X_t = \beta_t\,\dd t+\dd W_t^\dbQ,\) we obtain
\[
\dbE[D(\beta)\Xi^Z]
=
\dbE^\dbQ
\left[
\int_0^T
\left(
Z_t\cdot\beta_t
-
\ell^*(Z_t)
\right)\dd t
\right].
\]
The maximality of $\ell^*$ shows that for every $Z\in\mathcal Z_q$,
\[
Z_t\cdot\beta_t-\ell^*(Z_t)
\le
\ell(\beta_t).
\]
and hence
\begin{equation}
\dbE[D(\beta)\Xi^Z]
\le
\dbE^\dbQ
\left[
\int_0^T
\ell(\beta_t)\,\dd t
\right]
=
H_\ell(\dbQ).
\label{eq-uni-2}
\end{equation}

\revstep{2}{Truncate the maximizing field}
For the converse inequality, set
\[
z_t:=\nabla\ell(\beta_t),
\qquad
Z_t^m:=\lambda_t^m z_t,
\qquad
\lambda_t^m:=1\wedge\frac{m}{|z_t|}.
\]
Then $Z^m$ is bounded and progressively measurable. By Assumption~\ref{ass-rc},
\begin{equation}
\ell(t,\omega,b)
\ge
\ell(t,\omega,0)
+
\nabla\ell(t,\omega,0)\cdot b
+
\frac{c_\ell}{2}|b|^2.
\label{eq-ell-convex-b}    
\end{equation}
Consequently,
\[
-\ell(t,\omega,0)
\le
\ell^*(t,\omega,z)
\le
-\ell(t,\omega,0)
+
\frac{1}{2c_\ell}
\left|
z-\nabla\ell(t,\omega,0)
\right|^2.
\]
Since $\ell(\cdot,\cdot,0)$ and
$\nabla\ell(\cdot,\cdot,0)$ are bounded, $\ell^*$ is uniformly
bounded on bounded subsets of the $z$-variable. Therefore
$Z^m\in\mathcal Z_q$ for every $m$. 

For every $(t,\omega)\in[0,T]\times\Omega$, 
for each fixed $b\in\dbR^d$, the map \(z\mapsto b\cdot z-\ell(t,\omega,b)\) is affine. Therefore, $\ell^*(t,\omega,\cdot)$ is convex as the
pointwise supremum of a family of affine functions. Then we obtain
\[
0
\le
\ell(\beta_t)
-
\left(
Z_t^m\cdot\beta_t-\ell^*(Z_t^m)
\right)
\le
(1-\lambda_t^m)
\left(
\ell(\beta_t)+\ell^*(0)
\right).
\]
By \eqref{eq-ell-convex-b},
\[
\inf_{b\in\dbR^d}
\ell(t,\omega,b)
\ge
\ell(t,\omega,0)
-
\frac{1}{2c_\ell}
|\nabla\ell(t,\omega,0)|^2.
\]
We obtain
\[
\ell^*(t,\omega,0)
\le
-\ell(t,\omega,0)
+
\frac{1}{2c_\ell}
|\nabla\ell(t,\omega,0)|^2.
\]
Therefore, by the boundedness of
$\ell(\cdot,\cdot,0)$ and
$\nabla\ell(\cdot,\cdot,0)$,
\[
\sup_{(t,\omega)}
\ell^*(t,\omega,0)
\le
C_0+\frac{C_0^2}{2c_\ell}
<\infty.
\]
Since $(\dbQ,\beta)\in\cA$, the function \((t,\omega)\mapsto \ell(t,\omega,\beta_t)+\ell^*(t,\omega,0)\) is integrable with respect to $\dd t\otimes\dd\dbQ$.
Thus, by dominated convergence,
\[
\dbE^\dbQ
\int_0^T
\left[
\ell(\beta_t)
-
\left(
Z_t^m\cdot\beta_t-\ell^*(Z_t^m)
\right)
\right]\dd t
\longrightarrow0^+.
\]
Hence
\[
\sup_{Z\in\mathcal Z_q}\dbE[D(\beta)\Xi^Z]
\ge
H_\ell(\dbQ)
\]
Together with \eqref{eq-uni-2}, we obtain
\begin{equation}
\sup_{Z\in\mathcal Z_q}\dbE[D(\beta)\Xi^Z]
=
H_\ell(\dbQ)
\label{eq-uni-3}
\end{equation}

\revstep{3}{Weak lower semicontinuity}
We now prove Lemma~\ref{lem-H-weak-lsc}. Fix any
$Z\in\mathcal Z_q$. Since $\Xi^Z\in L^q(\dbP)$ and
\[
D(\beta^k)\rightharpoonup D(\beta)
\qquad
\text{weakly in }L^p(\dbP),
\]
we have
\[
\dbE\left[
D(\beta^k)\,\Xi^Z
\right]
\longrightarrow
\dbE\left[
D(\beta)\,\Xi^Z
\right].
\]
By \eqref{eq-uni-2}
\[
\dbE\left[
D(\beta^k)\,\Xi^Z
\right]
\le
H_\ell(\dbQ^k),
\qquad
k\ge1.
\]
Taking the lower limit in $k$ yields
\[
\dbE\left[
D(\beta)\,\Xi^Z
\right]
\le
\liminf_{k\to\infty}H_\ell(\dbQ^k).
\]
Finally, taking the supremum over $Z\in\mathcal Z_q$ and using
\eqref{eq-uni-3} again, we obtain
\[
H_\ell(\dbQ)
=
\sup_{Z\in\mathcal Z_q}
\dbE\left[
D(\beta)\,\Xi^Z
\right]
\le
\liminf_{k\to\infty}H_\ell(\dbQ^k).
\]
This proves \eqref{eq-uni-1}.
\end{proof}

\subsection{The common Bregman proximal construction}\label{sec-common-proximal-proof}\label{sec-proximal-proofs}
\subsubsection{The minimizer}

The BSDE produces a candidate law. Convexity then upgrades its pointwise optimality condition to a global variational inequality and uniqueness.
\begin{proof}[Proof of Proposition~\ref{prop-bregman-proximal}: the minimizer]
\phantomsection\label{proof-proximal-minimizer}

We first prove that the functional
\[
J(\dbQ)=H_\ell(\dbQ)+F(\dbQ)
\]
admits a unique minimizer.

\revstep{1}{Construct the candidate}
By Assumptions~\ref{ass-rc} and~\ref{ass-ini}, the running cost
$\ell$ satisfies the quadratic admissibility and uniform convexity
conditions required in Lemma~\ref{lem-BSDE-fp}. 
Moreover, by
Assumption~\ref{ass-pe}, the map
\[
(\dbQ,\omega)
\mapsto
\delta F(\dbQ,\omega)
\]
is uniformly bounded and continuous with respect to $\dbQ$ endowed
with the weak topology, uniformly in $\omega$. Hence
Lemma~\ref{lem-BSDE-fp}, applied with \(g=\ell \), \( {k(\dbQ,\cdot)=\delta F(\dbQ,\cdot)}\), yields a probability measure $\dbQ^*$ and a unique solution
\[
(Y^*,Z^*)
\in
\mathcal S^\infty\times\BMO
\]
of \(\operatorname{BSDE} \bigl( \cH_\ell, \delta F(\dbQ^*,\cdot) \bigr),\) such that, with
\[
\beta_t^*
:=
\argmin_{b\in\dbR^d}
\left\{
b\cdot Z_t^*
+
\ell(t,X,b)
\right\},
\]
we have
\[
\beta^*\in\BMO,
\qquad
\frac{\dd\dbQ^*}{\dd\dbP}
=
D(\beta^*).
\]

\revstep{2}{Control of the gradients}
Since $\beta^*\in\BMO$ and
$\dd\dbQ^*/\dd\dbP=D(\beta^*)$,
Lemma~\ref{lem-BMO-reverse-holder} yields
$Z^*\in\BMO(\dbQ^*)$.
The first order condition shows that
\[
\nabla\ell(\beta_t^*)+Z_t^*=0,
\]
and therefore
\begin{equation}\label{eq-JKO-1}
\nabla\ell(\beta^*)\in\BMO. 
\end{equation}
The gradient comparison in Assumption~\ref{ass-rc} then yields
\begin{equation}\label{eq-JKO-2}
\nabla\psi(\beta^*)\in\BMO. 
\end{equation}

\revstep{3}{Admissibility and conditional integrability}
By the optimality of $\beta^*$,
\[
\cH_\ell(t,X,Z_t^*)
=
\beta_t^*\cdot Z_t^*
+
\ell(t,X,\beta_t^*).
\]
Hence, under $\dbQ^*$,
\[
\dd Y_t^*
=
-\ell(t,X,\beta_t^*)\,\dd t
+
Z_t^*\,\dd W_t^{\beta^*}.
\]
For $m\ge1$, define
\[
\tau_m
:=
\inf\left\{
t\ge0:
\int_0^t
\ell^+(s,X,\beta_s^*)\,\dd s
\ge m
\right\}\wedge T.
\]
Since $Z^*\in\BMO(\dbQ^*)$, the stochastic integral
\[
\int_0^\cdot Z_t^*\,\dd W_t^{\beta^*}
\]
is a uniformly integrable martingale. Therefore,
\[
\dbE^{\dbQ^*}
\left[
\int_0^{\tau_m}
\ell(t,X,\beta_t^*)\,\dd t
\right]
=
Y_0^*
-
\dbE^{\dbQ^*}[Y_{\tau_m}^*].
\]
Since $Y^*\in\mathcal S^\infty$,
\begin{align*}
\dbE^{\dbQ^*}
\left[
\int_0^{\tau_m}
\ell^+(t,X,\beta_t^*)\,\dd t
\right]
&=
Y_0^*
-
\dbE^{\dbQ^*}[Y_{\tau_m}^*]
+
\dbE^{\dbQ^*}
\left[
\int_0^{\tau_m}
\ell^-(t,X,\beta_t^*)\,\dd t
\right]
\\
&\le
2\|Y^*\|_{\mathcal S^\infty}
+
\dbE^{\dbQ^*}
\left[
\int_0^T
\ell^-(t,X,\beta_t^*)\,\dd t
\right].
\end{align*}
The right-hand side is independent of $m$. Hence, by monotone
convergence,
\[
\dbE^{\dbQ^*}
\left[
\int_0^T
\ell^+(t,X,\beta_t^*)\,\dd t
\right]
<\infty.
\]
By \eqref{eq-derivative-3},
\begin{equation}\label{eq-JKO-3}
\dbE^{\dbQ^*}
\left[
\int_0^T
|\ell(t,X,\beta_t^*)|\,\dd t
\right]
<\infty.
\end{equation}
Moreover, for $\varphi\in\{\ell,\psi\}$, convexity gives
\[
\varphi^+(\beta_t^*)
\le
|\varphi(t,X,0)|
+
|\nabla\varphi(\beta_t^*)|\,|\beta_t^*|.
\]
Thus, by conditional Cauchy--Schwarz,
\[
\left\|
\sqrt{\varphi^+(\beta^*)}
\right\|_{\BMO}^2
\le
T\|\varphi(0)\|_\infty
+
\|\nabla\varphi(\beta^*)\|_{\BMO}
\|\beta^*\|_{\BMO}
<\infty.
\]
Together with \eqref{eq-JKO-1}, \eqref{eq-JKO-2} and \eqref{eq-JKO-3}, we obtain that
\begin{equation*}
(\dbQ^*,\beta^*)
\in
\cA_\ell\cap\cA_\psi.
\end{equation*}

\revstep{4}{Optimality and uniqueness}
Let $(\dbQ,\beta)\in\overline{\cA}$.
By the comparison inequality in
Lemma~\ref{lem-BSDE-solution}, applied at $\tau=0$,
\[
\dbE^{\dbQ^*}
\left[
\delta F(\dbQ^*,X)
+
\int_0^T\ell(\beta_t^*)\,\dd t
\right]
\le
\dbE^{\dbQ}
\left[
\delta F(\dbQ^*,X)
+
\int_0^T\ell(\beta_t)\,\dd t
\right].
\]
Therefore,
\[
H_\ell(\dbQ)-H_\ell(\dbQ^*)
\ge
-
\left\langle
\delta F(\dbQ^*,\cdot),
\dbQ-\dbQ^*
\right\rangle.
\]
By the convexity of $F$,
\[
F(\dbQ)-F(\dbQ^*)
\ge
\left\langle
\delta F(\dbQ^*,\cdot),
\dbQ-\dbQ^*
\right\rangle.
\]
Adding the two inequalities yields
\[
J(\dbQ)\ge J(\dbQ^*),
\]
so $(\dbQ^*,\beta^*)$ is a minimizer of $J$.

{Every competing minimiser in $\overline{\cA}$ has finite
$\ell$-cost. The lower bound
$\ell(t,X,b)\ge c_\ell|b|^2/4-C$ then gives finite drift energy.
Thus both minimisers lie in the finite-cost domain of
Lemma~\ref{lem-H-convex}. Their nontrivial mixture would have strictly
smaller $H_\ell$, and convexity of $F$ would give strictly smaller $J$.
This proves uniqueness in $\overline{\cA}$, without a BMO assumption
on the competitor.}

\revstep{5}{Stationarity at the minimizer}
\phantomsection\label{proof-optimizer-stationarity}
The optimal BSDE also yields the constant identity needed in
Lemma~\ref{lem-forward-coercivity}. Under $\dbQ^*$, integration gives
\[
 Y_0^*=\delta F(\dbQ^*,X)+\int_0^T\ell(t,X,\beta_t^*)\,\dd t
             -\int_0^T Z_t^*\cdot\dd W_t^{\beta^*}.
\]
Using $Z_t^*=-\nabla\ell(t,X,\beta_t^*)$ and
Lemma~\ref{lem-linear-derivative}, we obtain
\[
 \delta H_\ell(\dbQ^*,X)+\delta F(\dbQ^*,X)=Y_0^*,
 \qquad \dbQ^*\text{-a.s.}
\]
The stochastic integral is a uniformly integrable martingale since
$Z^*\in\BMO(\dbQ^*)$, and the cost is integrable by Step~3.
The initial sigma-field is trivial, so $Y_0^*$ is deterministic.
Since $\dbQ^*\sim\dbP$, the identity also holds $\dbP$-a.s.
Thus \eqref{eq-JKO-13} holds at $\dbQ'=\dbQ^*$ with $c'=Y_0^*$.

\end{proof}

\subsubsection{Admissibility of the proximal step}

The previous policy enters through a Bregman penalty. The proof checks that this new running cost meets the same admissibility conditions as the original one, allowing iteration.
\begin{proof}[Proof of Proposition~\ref{prop-bregman-proximal}: the proximal sequence]
\phantomsection\label{proof-proximal-step}

We construct the sequence recursively. By Assumption~\ref{ass-ini},
\[
(\dbQ^{h,0},\beta^{h,0})
\in
\cA_\ell\cap\cA_\psi.
\]
Suppose that, for some $n\ge0$,
$(\dbQ^{h,k},\beta^{h,k})$ has been constructed for every $k\le n$.
For $(t,\omega,b)\in[0,T]\times\Omega\times\dbR^d$, define
\[
\psi^{h,n}(t,\omega,b)
:=
\psi(t,\omega,b)
-\psi(\beta_t^{h,n})
-\nabla\psi(\beta_t^{h,n})
\cdot
(b-\beta_t^{h,n}),
\]
and set
\[
g^{h,n}(t,\omega,b)
:=
\ell(t,\omega,b)
+\frac1h\psi_{h,n}(t,\omega,b).
\]
By the convexity of $\psi$,
\[
\psi^{h,n}(t,\omega,b)\ge0,
\qquad
\psi^{h,n}(\beta_t^{h,n})=0.
\]

\revkeepstep
\noindent
\textit{Step 1: Structural properties of the penalized running cost.}
Since $\psi_t^{h,n}\ge0$,
\[
\begin{aligned}
-
\inf_{b\in\dbR^d}
\left\{
b\cdot z+g^{h,n}(t,\omega,b)
\right\}
&\le
-\inf_{b\in\dbR^d}
\left\{
b\cdot z+\ell(t,\omega,b)
\right\}
\le
-\ell(t,\omega,0)+ {\frac{|z+\nabla\ell(t,\omega,0)|^2}{2c_\ell}}\\
&\le C_0+\frac{C^2_0}{c_\ell}+\frac{|z|^2}{c_\ell}.
\end{aligned}
\] 
On the other hand, evaluating the infimum at $b=0$ gives
\[
-
\inf_{b\in\dbR^d}
\left\{
b\cdot z+g^{h,n}(t,\omega,b)
\right\}
\ge
-g^{h,n}(t,\omega,0)
\ge
-|g^{h,n}(t,\omega,0)|.
\]
Moreover,
\[
g^{h,n}(t,\omega,0)
=
\ell(t,\omega,0)
+\frac1h
\left(
\psi(t,\omega,0)
-\psi(\beta_t^{h,n})
+
\nabla\psi(\beta_t^{h,n})
\cdot\beta_t^{h,n}
\right).
\]
By the induction hypothesis and the assumptions on $\psi$,
Lemma~\ref{lem-BMO-exp-bound} yields, for some $p>1$,
\[
\bigl\|
|g^{h,n}(0)|
\bigr\|_{p,T}
<\infty.
\]
Consequently, $g^{h,n}$ satisfies a
$(p,\gamma_\ell)$ quadratic structure condition, where $\gamma_\ell=2\max\{C_0+\frac{C^2_0}{c_\ell},\frac{1}{c_\ell}\}$.
In addition,
\[
\nabla^2g^{h,n}(t,\omega,b)
=
\nabla^2\ell(t,\omega,b)
+
\frac1h\nabla^2\psi(t,\omega,b)
\ge
\left(
c_\ell+\frac{c_\psi}{h}
\right)I_d.
\]
Thus $g^{h,n}$ is uniformly strongly convex in $b$.

\revkeepstep
\noindent
\textit{Step 2: Existence and optimality.}

\revsubstep{2a}{Construct the candidate}
By Assumption~\ref{ass-pe}, the fixed-point result
Lemma~\ref{lem-BSDE-fp} applies to the running cost $g^{h,n}$.
Hence there exist \(\dbQ^{h,n+1}\in\cP(\Omega)\) and a solution
\[
(Y^{h,n+1},Z^{h,n+1})
\in
\mathcal S^\infty\times\BMO
\]
to
\[
\begin{cases}
\dd Y_t^{h,n+1}
=
-\mathcal H^{h,n}
(t,X,Z_t^{h,n+1})\,\dd t
+
Z_t^{h,n+1}\,\dd X_t,
\\[1mm]
Y_T^{h,n+1}
=
\delta F(\dbQ^{h,n+1},X).
\end{cases}
\]
Define
\[
\beta_t^{h,n+1}
:=
\argmin_{b\in\dbR^d}
\left\{
b\cdot Z_t^{h,n+1}
+
g^{h,n}(t,\omega,b)
\right\}.
\]
The minimizer is unique by the strong convexity of $g^{h,n}$ and
admits a progressively measurable version.
Moreover, Lemma~\ref{lem-BSDE-fp} yields
\begin{equation}
\label{eq-JKO-12}
\beta^{h,n+1}\in\BMO,
\qquad
\frac{\dd\dbQ^{h,n+1}}{\dd\dbP}
=
D(\beta^{h,n+1}).
\end{equation}

\revsubstep{2b}{Admissibility and localization}
The strong convexity of $\ell$ and
Assumption~\ref{ass-ini} imply
\[
\ell(t,X,b)
\ge
\ell(t,X,0)
+
\nabla\ell(t,X,0)\cdot b
+
\frac{c_\ell}{2}|b|^2
\ge
-C_{\ell,-},
\]
where
\begin{equation}
\label{eq-MF-C-ell-minus}
C_{\ell,-}
:=
C_0+\frac{C_0^2}{2c_\ell}.
\end{equation}
Since $\psi_t^{h,n}\ge0$,
\[
\bigl(g^{h,n}(t,\omega,b)\bigr)^-
\le
\ell^-(t,X,b)
\le
C_{\ell,-},
\]
and therefore
\begin{equation*}
\dbE^{\dbQ^{h,n+1}}
\left[
\int_0^T
\bigl(
g^{h,n}(t,X,\beta_t^{h,n+1})
\bigr)^-
\,\dd t
\right]
<\infty.
\end{equation*}
By the definition of $\beta^{h,n+1}$,
\[
\mathcal H^{h,n}(t,X,Z_t^{h,n+1})
=
\beta_t^{h,n+1}\cdot Z_t^{h,n+1}
+
g^{h,n}(t,X,\beta_t^{h,n+1}),
\]
so that, under $\dbQ^{h,n+1}$,
\begin{equation}
\label{eq-JKO-4}
\dd Y_t^{h,n+1}
=
-
g^{h,n}(t,X,\beta_t^{h,n+1})\,\dd t
+
Z_t^{h,n+1}\,\dd W_t^{\beta^{h,n+1}}.
\end{equation}

For $m\ge1$, define
\[
\tau_m
:=
\inf\left\{
t\ge0:
\int_0^t
\bigl(
g^{h,n}(s,X,\beta_s^{h,n+1})
\bigr)^+
\,\dd s
\ge m
\right\}\wedge T.
\]
Since
$Z^{h,n+1}\in\BMO(\dbQ^{h,n+1})$,
the stochastic integral in
\eqref{eq-JKO-4}
is a uniformly integrable martingale. Therefore,
\begin{equation*}
\dbE^{\dbQ^{h,n+1}} \left[ \int_0^{\tau_m} \bigl( g^{h,n}(t,X,\beta_t^{h,n+1}) \bigr)^+ \,\dd t \right] \le 2\|Y^{h,n+1}\|_{\mathcal S^\infty} + \dbE^{\dbQ^{h,n+1}} \left[ \int_0^T \bigl( g^{h,n}(t,X,\beta_t^{h,n+1}) \bigr)^- \,\dd t \right].
\end{equation*}
Hence, by monotone convergence,
\begin{equation}
\label{eq-JKO-5}
\dbE^{\dbQ^{h,n+1}}
\left[
\int_0^T
\left|
g^{h,n}(t,X,\beta_t^{h,n+1})
\right|
\,\dd t
\right]
<\infty.
\end{equation}
Since
\[
g^{h,n}(t,\omega,b)
=
\ell(b)+\frac1h\psi_t^{h,n}(b),
\qquad
\psi_t^{h,n}(b)\ge0,
\]
we have
\[
\ell^+(\beta_t^{h,n+1})
\le
\bigl(
g^{h,n}(t,X,\beta_t^{h,n+1})
\bigr)^+,
\qquad
\ell^-(\beta_t^{h,n+1})
\le
C_{\ell,-}.
\]
Thus
\[
\dbE^{\dbQ^{h,n+1}}
\left[
\int_0^T
|\ell(\beta_t^{h,n+1})|\,\dd t
\right]
<\infty.
\]
Together with \eqref{eq-JKO-12} shows that
\begin{equation}
\label{eq-JKO-11}
(\dbQ^{h,n+1},\beta^{h,n+1})\in\cA.
\end{equation}

\revsubstep{2c}{Comparison and optimality}
Moreover,
\eqref{eq-JKO-5} implies that, for every stopping time
$\tau\in\cT$,
\begin{equation}
\label{eq-JKO-6}
Y_\tau^{h,n+1}
=
\dbE^{\dbQ^{h,n+1}}
\left[
\delta F(\dbQ^{h,n+1},X)
+
\int_\tau^T
g^{h,n}(t,X,\beta_t^{h,n+1})\,\dd t
\Bigm|
\cF_\tau
\right].
\end{equation}
Let $(\dbQ,\beta)\in\overline{\cA}$ with
$J^{h,n}(\dbQ)<\infty$.
By the definition of $J^{h,n}$, the integrability condition in
Lemma~\ref{lem-BSDE-solution} is satisfied. Hence, for every
$\tau\in\cT$,
\begin{equation}\label{eq-JKO-7}
Y_\tau^{h,n+1}
\le
\dbE^{\dbQ}
\left[
\delta F(\dbQ^{h,n+1},X)
+
\int_\tau^T
g^{h,n}(t,X,\beta_t)\,\dd t
\Bigm|
\cF_\tau
\right].
\end{equation}
Combining this with
\eqref{eq-JKO-6}
and taking $\tau=0$, we obtain
\begin{align}
&\dbE^{\dbQ^{h,n+1}}
\left[
\delta F(\dbQ^{h,n+1},X)
\right]
+
H_\ell(\dbQ^{h,n+1})
+
\frac1h
H_\psi(\dbQ^{h,n+1}\,|\,\dbQ^{h,n})
\nonumber\\
&\qquad\le
\dbE^{\dbQ}
\left[
\delta F(\dbQ^{h,n+1},X)
\right]
+
H_\ell(\dbQ)
+
\frac1h
H_\psi(\dbQ\,|\,\dbQ^{h,n}).
\label{eq-JKO-8}
\end{align}
By the convexity and linear differentiability of $F$,
\[
F(\dbQ)
\ge
F(\dbQ^{h,n+1})
+
\left\langle
\delta F(\dbQ^{h,n+1},\cdot),
\dbQ-\dbQ^{h,n+1}
\right\rangle.
\]
Combining this with
\eqref{eq-JKO-8} yields
\[
J^{h,n}(\dbQ)
\ge
J^{h,n}(\dbQ^{h,n+1}).
\]

\revkeepstep
\noindent
\textit{Step 3: Regularity}

\revsubstep{3a}{The gradient bounds}
We now prove that
\[
(\dbQ^{h,n+1},\beta^{h,n+1})
\in
\cA_\ell\cap\cA_\psi.
\]
By the first-order condition,
\begin{equation*}
Z_t^{h,n+1}
+
\nabla\ell(\beta_t^{h,n+1})
+
\frac1h
\left(
\nabla\psi(\beta_t^{h,n+1})
-
\nabla\psi(\beta_t^{h,n})
\right)
=0.
\end{equation*}
Equivalently,
\[
\nabla\psi(\beta_t^{h,n+1})
=
\nabla\psi(\beta_t^{h,n})
-hZ_t^{h,n+1}
-h\nabla\ell(\beta_t^{h,n+1}).
\]
Hence, by Assumption~\ref{ass-rc},
\begin{align*}
|\nabla\psi(\beta_t^{h,n+1})|
&\le
|\nabla\psi(\beta_t^{h,n})|
+h|Z_t^{h,n+1}|
+h|\nabla\ell(\beta_t^{h,n+1})|
\\
&\le
|\nabla\psi(\beta_t^{h,n})|
+h|Z_t^{h,n+1}|
+hL_\ell
\left(
1+|\nabla\psi(\beta_t^{h,n+1})|
\right).
\end{align*}
Since $h\le h_0=\frac{1}{2L_\ell}$,
\[
1-hL_\ell\ge\frac12,
\]
and therefore
\begin{equation*}
|\nabla\psi(\beta_t^{h,n+1})|
\le
2|\nabla\psi(\beta_t^{h,n})|
+
2h|Z_t^{h,n+1}|
+
2hL_\ell.
\end{equation*}
Taking conditional square integrals gives
\begin{equation}\label{eq-JKO-9}
\|\nabla\psi(\beta^{h,n+1})\|_{\BMO} \le 2\|\nabla\psi(\beta^{h,n})\|_{\BMO} + 2h\|Z^{h,n+1}\|_{\BMO} + 2hL_\ell\sqrt{T} <+\infty.
\end{equation}
Here we used the induction hypothesis
$(\dbQ^{h,n},\beta^{h,n})\in\cA_\psi$
and
$Z^{h,n+1}\in\BMO$.
Moreover, by Assumption~\ref{ass-rc},
\[
|\nabla\ell(\beta_t^{h,n+1})|
\le
L_\ell
\left(
1+|\nabla\psi(\beta_t^{h,n+1})|
\right),
\]
and hence
\begin{equation}
\label{eq-JKO-10}
\|\nabla\ell(\beta^{h,n+1})\|_{\BMO}
\le
L_\ell
\left(
\sqrt{T}
+
\|\nabla\psi(\beta^{h,n+1})\|_{\BMO}
\right)
<+\infty.
\end{equation}

\revsubstep{3b}{The positive parts of the costs}
It remains to control the positive parts of
$\ell(\beta^{h,n+1})$ and
$\psi(\beta^{h,n+1})$.
Let $\varphi\in\{\ell,\psi\}$.
By convexity,
\[
\varphi(t,\omega,0)
\ge
\varphi(t,\omega,b)
-
\nabla\varphi(t,\omega,b)\cdot b,
\]
and therefore
\begin{equation*}
\varphi^+(t,\omega,b)
\le
|\varphi(t,\omega,0)|
+
|\nabla\varphi(t,\omega,b)|\,|b|.
\end{equation*}
Applying this with
$b=\beta_t^{h,n+1}$ and using conditional
Cauchy--Schwarz yields
\begin{equation*}
\left\|
\sqrt{\varphi^+(\beta^{h,n+1})}
\right\|_{\BMO}^2
\le
T\|\varphi(0)\|_\infty
+
\|\nabla\varphi(\beta^{h,n+1})\|_{\BMO}
\|\beta^{h,n+1}\|_{\BMO}
<\infty.
\end{equation*}
Here
$\beta^{h,n+1}\in\BMO$
follows from Lemma~\ref{lem-BSDE-fp},
while
\eqref{eq-JKO-9} and
\eqref{eq-JKO-10}
give the corresponding gradient estimates.
By \eqref{eq-JKO-11}, \eqref{eq-JKO-9} and \eqref{eq-JKO-10}, we obtain
\[
(\dbQ^{h,n+1},\beta^{h,n+1})
\in
\cA_\ell\cap\cA_\psi.
\]
Together with Step~2, this proves
\[
(\dbQ^{h,n+1},\beta^{h,n+1})
\in
\argmin_{(\dbQ,\beta)\in\overline{\cA}}
J^{h,n}(\dbQ).
\]
This closes the induction.

\revkeepstep
\noindent
\textit{Step 4: Uniqueness in $\cA_\ell\cap\cA_\psi$.}
Let
$(\dbQ,\beta),(\dbQ',\beta')
\in\cA_\ell\cap\cA_\psi$
be two minimizers of $J^{h,n}$, and set \(\dbQ^{(s)} := (1-s)\dbQ+s\dbQ', \qquad s\in[0,1].\) Let $D,D'$ and $D^{(s)}$ denote the corresponding density processes.
By Lemma~\ref{lem-linear-derivative}, applied to $\psi$,
\[
H_\psi(\dbQ\,|\,\dbQ^{h,n})
=
H_\psi(\dbQ)
-
H_\psi(\dbQ^{h,n})
-
\left\langle
\delta H_\psi(\dbQ^{h,n},\cdot),
\dbQ-\dbQ^{h,n}
\right\rangle.
\]
Hence the quantitative convexity estimate
\eqref{eq-JKO-strong-convexity-H} applied to $\psi$ yields
\begin{equation*}
H_\psi(\dbQ^{(s)}\,|\,\dbQ^{h,n}) \le (1-s)H_\psi(\dbQ\,|\,\dbQ^{h,n}) + sH_\psi(\dbQ'\,|\,\dbQ^{h,n}) - \frac{c_\psi}{2}s(1-s) \dbE\left[ \int_0^T \frac{D_tD_t'}{D_t^{(s)}} |\beta_t-\beta_t'|^2\,\dd t \right].
\end{equation*}
Combining this with
\eqref{eq-JKO-strong-convexity-H} applied to $\ell$
and the convexity of $F$, we obtain
\begin{equation*}
J^{h,n}(\dbQ^{(s)}) \le (1-s)J^{h,n}(\dbQ) + sJ^{h,n}(\dbQ') - \frac12 \left( c_\ell+\frac{c_\psi}{h} \right) s(1-s) \dbE\left[ \int_0^T \frac{D_tD_t'}{D_t^{(s)}} |\beta_t-\beta_t'|^2\,\dd t \right].
\end{equation*}
Since $\dbQ$ and $\dbQ'$ are both minimizers, the last term must
vanish. Hence
\[
\beta=\beta',
\qquad
\dd t\otimes\dd\dbP\text{-a.e.}
\]
The corresponding density processes therefore coincide, and thus $
\dbQ=\dbQ'$, proving uniqueness for the minimizer is unique in
$\cA_\ell\cap\cA_\psi$.

\end{proof}

\subsubsection{Discrete divergence and objective decay}
The three-point identity converts the implicit optimality condition into a one-step decrease. Relative convexity controls the distance to the optimum, and a weighted sum yields the objective gap.
\begin{proof}[Proof of Proposition~\ref{prop-bregman-proximal}: the decay estimates]
\phantomsection\label{proof-proximal-decay}

\revstep{1}{First-order inequality}
For brevity, write
\[
\delta H_\ell^{h,n}
:=
\delta H_\ell(\dbQ^{h,n},\cdot),
\qquad
\delta H_\psi^{h,n}
:=
\delta H_\psi(\dbQ^{h,n},\cdot),
\qquad
\delta F^{h,n}
:=
\delta F(\dbQ^{h,n},\cdot).
\]
Since \((\dbQ^{h,n+1},\beta^{h,n+1}) \in \cA_\ell\cap\cA_\psi,\) Lemma~\ref{lem-linear-derivative} applies at
$\dbQ^{h,n+1}$.
The first-order variational inequality for the {Bregman proximal} step reads
\begin{align*}
\Biggl\langle
\delta H_\ell^{h,n+1}
+
\delta F^{h,n+1}
+
\frac1h
\left(
\delta H_\psi^{h,n+1}
-
\delta H_\psi^{h,n}
\right),
\,
\dbQ-\dbQ^{h,n+1}
\Biggr\rangle
\ge0.
\end{align*}
Applying this with $\dbQ=\dbQ^*$ yields
\begin{equation}
\label{eq-JKO-16}
\left\langle
\delta H_\psi^{h,n}
-
\delta H_\psi^{h,n+1},
\,
\dbQ^*-\dbQ^{h,n+1}
\right\rangle
\le
h
\left\langle
\delta H_\ell^{h,n+1}
+
\delta F^{h,n+1},
\,
\dbQ^*-\dbQ^{h,n+1}
\right\rangle.
\end{equation}

\revstep{2}{Three-point identity and contraction}
Set \(\mathcal E_n := H_\psi(\dbQ^*\,|\,\dbQ^{h,n}).\) By the three-point identity,
\begin{equation}\label{eq-JKO-14}
\mathcal E_{n+1}-\mathcal E_n = - H_\psi( \dbQ^{h,n+1}\,|\,\dbQ^{h,n}) + \left\langle \delta H_\psi^{h,n} - \delta H_\psi^{h,n+1}, \, \dbQ^*-\dbQ^{h,n+1} \right\rangle.
\end{equation}
Combining
\eqref{eq-JKO-14}
with
\eqref{eq-JKO-16},
we obtain
\begin{equation}\label{eq-JKO-15}
\mathcal E_{n+1}-\mathcal E_n \le - H_\psi( \dbQ^{h,n+1}\,|\,\dbQ^{h,n}) + h \left\langle \delta H_\ell^{h,n+1} + \delta F^{h,n+1}, \, \dbQ^*-\dbQ^{h,n+1} \right\rangle.
\end{equation}
By Lemma~\ref{lem-linear-derivative},
\[
H_\ell(\dbQ^*\,|\,\dbQ^{h,n+1})
=
H_\ell(\dbQ^*)
-
H_\ell(\dbQ^{h,n+1})
-
\left\langle
\delta H_\ell^{h,n+1},
\dbQ^*-\dbQ^{h,n+1}
\right\rangle.
\]
Moreover, the convexity of $F$ gives
\[
F(\dbQ^*)-F(\dbQ^{h,n+1})
\ge
\left\langle
\delta F^{h,n+1},
\dbQ^*-\dbQ^{h,n+1}
\right\rangle.
\]
Hence
\begin{equation}\label{eq-JKO-17}
\left\langle \delta H_\ell^{h,n+1} + \delta F^{h,n+1}, \dbQ^*-\dbQ^{h,n+1} \right\rangle \le J(\dbQ^*)-J(\dbQ^{h,n+1}) - H_\ell( \dbQ^*\,|\,\dbQ^{h,n+1}).
\end{equation}
By \eqref{eq-relative-ell-psi},
\[
H_\ell(
\dbQ^*\,|\,\dbQ^{h,n+1})
\ge
\lambda_{\ell|\psi}
H_\psi(
\dbQ^*\,|\,\dbQ^{h,n+1})
=
\lambda_{\ell|\psi}\mathcal E_{n+1}.
\]
Substituting
\eqref{eq-JKO-17}
into
\eqref{eq-JKO-15}
yields
\begin{equation}\label{eq-JKO-18}
\mathcal E_n \ge \left( 1+\lambda_{\ell|\psi}h \right) \mathcal E_{n+1} + h\left( J(\dbQ^{h,n+1})-J(\dbQ^*) \right) + H_\psi( \dbQ^{h,n+1}\,|\,\dbQ^{h,n}).
\end{equation}
In particular,
\[
\left(
1+\lambda_{\ell|\psi}h
\right)
\mathcal E_{n+1}
\le
\mathcal E_n,
\]
and hence
\[
\mathcal E_n
\le
\frac{\mathcal E_0}
{
(1+\lambda_{\ell|\psi}h)^n
}.
\]

\revstep{3}{Weighted summation and objective decay}
By \eqref{eq-JKO-18}, we obtain for every $0\le i \le n-1$,
\[
\mathcal E_i
\ge
\left(
1+\lambda_{\ell|\psi}h
\right)
\mathcal E_{i+1}
+
h\left(
J(\dbQ^{h,i+1})-J(\dbQ^*)
\right)
\]
Multiplying by
$(1+\lambda_{\ell|\psi}h)^i$
and summing over
$i=0,\ldots,n-1$ gives
\[
h
\sum_{i=0}^{n-1}
(1+\lambda_{\ell|\psi}h)^i
(J(\dbQ^{h,i+1})-J(\dbQ^*))
\le
\mathcal E_0.
\]
Therefore,
\[
(J(\dbQ^{h,n})-J(\dbQ^*))
\frac{
(1+\lambda_{\ell|\psi}h)^n-1
}{
\lambda_{\ell|\psi}
}
\le
\mathcal E_0.
\]
Since the {objective value along the Bregman proximal sequence} is non-increasing,
\(
J(\dbQ^{h,i})-J(\dbQ^*)
\)
is non-increasing, we obtain
\[
J(\dbQ^{h,n})-J(\dbQ^*)
\le
\frac{
\lambda_{\ell|\psi}\mathcal E_0
}{
(1+\lambda_{\ell|\psi}h)^n-1
}.
\]
By Lemma~\ref{lem-forward-coercivity} and the stationarity identity at $\dbQ^*$, we obtain
\[
H_\psi(\dbQ^{h,n}\mid\dbQ^*)\le
\frac{J(\dbQ^{h,n})-J(\dbQ^*)}{\lambda_{\ell|\psi}}.
\]
This completes the proof.
\end{proof}

\section{Newton and natural gradient flows}\label{sec-flow-proofs}
\readingguide{We first prove the two construction theorems, keeping their
assumptions separate. The Newton proof gives the exact representation and
strong refinements; the natural-gradient proof supplies a nonlinear BMO
flow. After both curves have been constructed, we verify the limit
passages needed to transfer the discrete Bregman and objective estimates
to each continuous flow; the required continuity estimate is proved in
Section~\ref{sec-common-limits} immediately before it is used.}

\subsection{The Newton flow: static construction and its consequences}
\label{sec-static-proof}\label{sec-newton-proofs}\label{sec-linear-uniqueness}
The entropy-coordinate formula alone does not reconstruct a law. The static control problems below construct the inverse along the required segment and provide its stability.
\begin{proof}[Proof of Theorem~\ref{thm-newton-structure}]
\phantomsection\label{proof-static-newton}

We write $q^0=\nabla\ell(\beta^0)$. Constants in this proof may change
from line to line; they depend only on the data indicated in the
statement. We use the BMO change-of-measure estimates in their norm
form: on a fixed BMO ball, the norms before and after a Girsanov
transformation are comparable by a finite constant. The squared-norm
version uses the square of that constant; see
\cite[Theorem~2]{ChikvinidzeMania2014}.

\par\noindent\textit{Step 1: the static control problems.}
Strong convexity and the bounds at zero imply
\begin{equation}\label{eq-static-coercive}
 \ell(t,X,b)\ge\frac{c_\ell}{4}|b|^2-C.
\end{equation}
For $0\le a<1$,
\[
 g_a(b)=\ell(b)-a\ell(\beta^0)-a q^0\cdot(b-\beta^0),
 \qquad \nabla^2g_a=\nabla^2\ell,
 \qquad g_a\ge(1-a)\ell.
\]
Consequently $\cH_{g_a}(z)\ge-C(1-a)-C|z|^2/(1-a)$.
Also $\cH_{g_a}(z)\le g_a(0)$, and
\[
 g_a(0)\le C+|q^0|\,|\beta^0|.
\]
The conditional moments of every finite order of the time integral of
the right-hand side are bounded, by Cauchy--Schwarz and the BMO energy
inequality. Thus $g_a$ satisfies the quadratic structure condition of
Section~\ref{sec-bsde-estimates}, with constants that may depend on
$a<1$. Lemma~\ref{lem-BSDE-solution}, the feedback construction and
verification give a bounded solution $(Y^a,Z^a)$ and an optimal drift
$\widehat\beta^a\in\BMO$, characterised by
\begin{equation}\label{eq-static-feedback}
 \nabla\ell(\widehat\beta^a)=a q^0-Z^a,
 \qquad
 Y_T^a=(1-a)f.
\end{equation}
The finite-cost convexity and uniqueness argument of
Section~\ref{sec-proximal-proofs} applies to $g_a$.
In particular $\widehat{\dbQ}^a=D(\widehat\beta^a)\dbP$ is its unique
optimal law. Formula \eqref{eq-static-feedback}, together with
\[
 \ell^+(b)\le C+|\nabla\ell(b)|\,|b|,
\]
shows that $(\widehat{\dbQ}^a,\widehat\beta^a)\in\cA_\ell$.
At $a=1$, take $Y^1=Z^1=0$ and $\widehat\beta^1=\beta^0$.
Indeed $g_1$ is the nonnegative pointwise Bregman divergence from
$\beta^0$, and $\cH_{g_1}(0)=0$, so these processes solve the
corresponding BSDE directly. We do not apply the preceding quadratic
structure estimate at $a=1$.

\par\noindent\textit{Step 2: uniform conditional estimates.}
\phantomsection\label{proof-newton-bmo}
Under $\widehat{\dbQ}^a$, policy evaluation reads
\[
 \dd Y_t^a=-g_a(\widehat\beta_t^a)\,\dd t
                        +Z_t^a\,\dd W_t^{\widehat\beta^a}.
\]
For a stopping time $\tau$, comparison with the initial policy gives
\[
 Y_\tau^a
 \le(1-a)\dbE^{\dbQ^0}\!\left[
 f+\int_\tau^T\ell(\beta_t^0)\,\dd t\,\middle|\,\cF_\tau\right]
 \le C(1-a).
\]
Here we used $g_a(\beta^0)=(1-a)\ell(\beta^0)$ and the
$\ell$-BMO integrability of $\beta^0$, transferred to $\dbQ^0$.
The lower bound in \eqref{eq-static-coercive} and
$g_a\ge(1-a)\ell$ similarly give $Y_\tau^a\ge-C(1-a)$.
Therefore
\begin{equation}\label{eq-static-Y-bound}
 \|Y^a\|_\infty\le C(1-a).
\end{equation}
The evaluation identity and $g_a^-\le C(1-a)$ now imply
\begin{equation}\label{eq-static-running-bound}
 \sup_{\tau\in\cT}\left\|
 \dbE^{\widehat{\dbQ}^a}\!\left[
 \int_\tau^T|g_a(\widehat\beta_t^a)|\,\dd t
       \,\middle|\,\cF_\tau\right]\right\|_\infty
 \le C(1-a).
\end{equation}
For $a<1$, dropping the nonnegative Bregman term in $g_a$ and dividing
by $1-a$ also gives a uniform upper bound on the conditional integral
of $\ell(\widehat\beta^a)$. Hence
\[
 \sup_{0\le a<1}\|\widehat\beta^a\|_{\BMO(\widehat{\dbQ}^a)}\le C.
\]
Applying It\^o's formula to $(Y^a)^2$, localising and then using
\eqref{eq-static-Y-bound}--\eqref{eq-static-running-bound}, yields
\[
 \|Z^a\|_{\BMO(\widehat{\dbQ}^a)}^2
 \le (1-a)^2\|f\|_\infty^2
       +2\|Y^a\|_\infty
           \sup_\tau\left\|\dbE^{\widehat{\dbQ}^a}
              \left[\int_\tau^T|g_a(\widehat\beta_t^a)|\,\dd t
                  \,\middle|\,\cF_\tau\right]\right\|_\infty
 \le C(1-a)^2.
\]
The inverse change of measure from $\widehat{\dbQ}^a$ to $\dbP$
has stochastic logarithm $-\widehat\beta^a$. Its BMO norm is uniformly
bounded by the preceding estimate. BMO invariance therefore gives
\begin{equation}\label{eq-static-uniform-BMO}
 \sup_{a\in[0,1]}\|\widehat\beta^a\|_{\BMO}\le C,
 \qquad \|Z^a\|_{\BMO}\le C(1-a),
 \qquad
 \sup_{a\in[0,1]}|\widehat\beta^a|_{\ell-\BMO}\le C.
\end{equation}
The last assertion follows from \eqref{eq-static-feedback} and the
convexity bound on $\ell^+$ above. The endpoint $a=1$ is included
using the initial assumption.

\par\noindent\textit{Step 3: stability in the interpolation parameter.}
The minimiser in the Hamiltonian is
\[
 b^a(z)=(\nabla\ell)^{-1}(a q^0-z).
\]
Strong monotonicity gives
\begin{equation}\label{eq-static-inverse-bound}
 |b^a(z)|\le c_\ell^{-1}
     \bigl(|q^0|+|z|+|\nabla\ell(0)|\bigr),\qquad
 |b^a(z)-b^{a'}(z')|
 \le c_\ell^{-1}\bigl(|a-a'|\,|q^0|+|z-z'|\bigr).
\end{equation}
The envelope identities are
\[
 \nabla_z\cH_{g_a}(z)=b^a(z),\qquad
 \partial_a\cH_{g_a}(z)
     =-\ell(\beta^0)-q^0\cdot(b^a(z)-\beta^0).
\]
Fix $a,a'\in[0,1]$ and put $\Delta Y=Y^a-Y^{a'}$,
$\Delta Z=Z^a-Z^{a'}$. Integrating the two preceding derivatives
along line segments gives progressive processes $\lambda$ and $\kappa$
such that
\[
 \cH_{g_a}(Z^a)-\cH_{g_{a'}}(Z^{a'})
        =\lambda\cdot\Delta Z+(a-a')\kappa.
\]
By \eqref{eq-static-uniform-BMO}--\eqref{eq-static-inverse-bound},
$\|\lambda\|_{\BMO}\le C$. Moreover, under
$\dbR=D(\lambda)\dbP$,
\begin{equation}\label{eq-static-kappa}
 \sup_{\tau\in\cT}\left\|
 \dbE^{\dbR}\!\left[\int_\tau^T|\kappa_t|\,\dd t
                        \,\middle|\,\cF_\tau\right]\right\|_\infty
 \le C.
\end{equation}
For completeness, $|\kappa|$ is bounded by $|\ell(\beta^0)|$ plus
$|q^0|$ times $|\beta^0|$ and averages of $|b^r(Z^{a'})|$.
The first term has bounded conditional integral by the
$\ell$-BMO condition; each product is controlled by conditional
Cauchy--Schwarz. All the relevant BMO norms remain bounded under
$\dbR$, since $\lambda$ lies in a fixed BMO ball. This proves
\eqref{eq-static-kappa} uniformly in both parameters.

Under $\dbR$, the difference BSDE has driver $(a-a')\kappa$ and
terminal value $-(a-a')f$. Conditional expectation gives
$\|\Delta Y\|_\infty\le C|a-a'|$. The square identity then gives
\[
 \|\Delta Z\|_{\BMO(\dbR)}^2
 \le |a-a'|^2\|f\|_\infty^2
      +2|a-a'|\|\Delta Y\|_\infty
         \sup_\tau\left\|\dbE^{\dbR}
             \left[\int_\tau^T|\kappa_t|\,\dd t
                  \,\middle|\,\cF_\tau\right]\right\|_\infty
 \le C|a-a'|^2.
\]
Transferring back to $\dbP$ and using
\eqref{eq-static-feedback}--\eqref{eq-static-inverse-bound} proves
the first two terms of \eqref{eq-static-stability}.
Lemma~\ref{lem-BMO-lip} proves the total variation estimate.

\par\noindent\textit{Step 4: dual coordinates and the implicit scheme.}
\phantomsection\label{proof-newton-uniqueness}
Integrating the evaluation BSDE and substituting
\eqref{eq-static-feedback} gives the pathwise identity
\begin{equation}\label{eq-static-dual-unscaled}
 \delta H_\ell(\widehat{\dbQ}^a)+(1-a)f
          -a\delta H_\ell(\dbQ^0)=Y_0^a.
\end{equation}
Indeed, the finite-variation part of the left-hand side is
$\int g_a(\widehat\beta^a)$, and its martingale part under
$\widehat{\dbQ}^a$ is $-\int Z^a\,\dd W^{\widehat\beta^a}$.
All terms belong to every finite $L^q(\dbP)$ by the BMO energy and
martingale moment inequalities. The initial sigma-field is trivial,
so $Y_0^a$ is deterministic. Centering \eqref{eq-static-dual-unscaled}
therefore gives the affine dual interpolation.

Write $\rho^{h,n}=\delta H_\ell(\dbQ^{h,n})$, and use bars for centring under $\dbP$.
Fix $0<h<h_0$ and $n\ge0$. Write
$q^{h,n}=\nabla\ell(\beta^{h,n})$ and let $(Y^{h,n+1},Z^{h,n+1})$
be the optimal penalised BSDE. Its running cost is
\[
 g^{h,n}(b)=\ell(b)+h^{-1}
 [\ell(b)-\ell(\beta^{h,n})-q^{h,n}\cdot(b-\beta^{h,n})],
\]
and its terminal value is $f$. Optimality in the drift gives
\[
 (1+h^{-1})q^{h,n+1}-h^{-1}q^{h,n}=-Z^{h,n+1}.
\]
Under $\dbQ^{h,n+1}$, integration of the BSDE therefore yields
\[
 Y_0^{h,n+1}=f+\int_0^T g^{h,n}(\beta_t^{h,n+1})\,\dd t
               -\int_0^T Z_t^{h,n+1}\,\dd W_t^{\beta^{h,n+1}}.
\]
In the entropy-derivative formula, rewrite the integral for the
previous policy using
$\dd W^{\beta^{h,n}}=\dd W^{\beta^{h,n+1}}
 +(\beta^{h,n+1}-\beta^{h,n})\,\dd t$.
The finite-variation terms become $\int g^{h,n}(\beta^{h,n+1})$
and the stochastic terms become $-\int Z^{h,n+1}\,\dd W^{\beta^{h,n+1}}$.
Consequently
\[
 \rho^{h,n+1}+f+h^{-1}(\rho^{h,n+1}-\rho^{h,n})
 =Y_0^{h,n+1}=:c_{h,n+1},\qquad\dbP\text{-a.s.}
\]
All these terms have finite $L^q(\dbP)$ moments by the gradient-BMO
estimates. The initial sigma-field is trivial, so $c_{h,n+1}$ is
deterministic.
After centering, this is equivalent to
\begin{equation*}
\label{eq-uni-7}
\overline\rho^{h,n+1}
+\overline{f}
+
\frac{1}{h}
\bigl(
\overline\rho^{h,n+1}
-
\overline\rho^{h,n}
\bigr)
=0,
\qquad
\dbP\text{-a.s.}
\end{equation*}

Thus \eqref{eq-main-dual-affine} holds at the discrete level.
If two regular admissible laws have the same centred derivative,
then the sum of their two $\ell$-Bregman divergences is zero, by
Lemma~\ref{lem-linear-derivative}. Strong convexity identifies their
drifts and hence their laws. Comparing the discrete dual coordinate
with \eqref{eq-static-dual-unscaled} proves the first identity in
\eqref{eq-static-identification}.
The conditional bound in Step 2 therefore also gives
\[
 \sup_{0<h<h_0}\sup_{n\ge0}
 \|\beta^{h,n}\|_{\BMO(\dbQ^{h,n})}\le C(c_\ell,C_0,C_F).
\]

Define $\dbQ^s=\widehat{\dbQ}^{e^{-s}}$. For $h$ in the prescribed
bounded range,
\[
 \sup_{s\ge0}|(1+h)^{-\lfloor s/h\rfloor}-e^{-s}|\le Ch.
\]
To see this, put $r_h=\log(1+h)/h$ and $t_h=h\lfloor s/h\rfloor$.
The difference between $e^{-r_ht_h}$ and $e^{-t_h}$ is at most
$(1-r_h)t_he^{-r_ht_h}\le(1-r_h)/(er_h)$, and the remaining
difference is at most $h$. Since $1-r_h\le Ch$, the assertion follows.
Together with \eqref{eq-static-stability}, this proves
\eqref{eq-newton-strong-rate}, identifies the limit of the whole
sequence, and proves the continuous dual identity.
\phantomsection\label{proof-newton-density}
There is also strong $L^p$ convergence for some $p>1$.
All these drifts lie in a fixed BMO ball, so reverse
H\"older gives a common $p_0>1$ with
\[
 \sup_{h,s}\|D(\beta^{h,s})\|_{p_0}
 +\sup_s\|D(\beta^s)\|_{p_0}<\infty.
\]
Choose $1<p<p_0$ and $\vartheta\in(0,1)$ such that
$1/p=\vartheta+(1-\vartheta)/p_0$. Interpolation yields
\[
 \sup_s\|D(\beta^{h,s})-D(\beta^s)\|_p
 \le (Ch)^\vartheta
 \left(\sup_s\|D(\beta^{h,s})-D(\beta^s)\|_{p_0}\right)^{1-\vartheta}
 \longrightarrow0.
\]
For every finite $q>1$, the same scalar estimate gives
\begin{equation}\label{eq-uni-8}
 \sup_{s\ge0}\|\overline\rho^{h,s}-\overline\rho^s\|_{L^q(\dbP)}
 \le Ch\|\overline\rho^0+\overline f\|_{L^q(\dbP)}.
\end{equation}
Moreover $\overline\rho\in C^1([0,\infty);L^q(\dbP))$ is the unique solution of
\begin{equation*}\label{eq-uni-6}
 \partial_s\overline\rho^s=-(\overline\rho^s+\overline f),
 \qquad \overline\rho^{s=0}=\overline\rho^0.
\end{equation*}
The stability estimate also gives, for $s_1,s_2\ge0$,
\[
\|\beta^{s_2}-\beta^{s_1}\|_{\BMO}
 +\|\dbQ^{s_2}-\dbQ^{s_1}\|_{\TV}
 \le C|e^{-s_2}-e^{-s_1}|\le C|s_2-s_1|.
\]
In particular the local $1/2$-H\"older conclusion follows. Together with
Steps 2--4, this records the BMO and continuity consequences of the static
construction and identifies the limit without a second candidate curve.

\par\noindent\textit{Step 5: equilibrium and the objective gap.}
At $a=0$ the static problem is the original control problem, so
$\widehat{\dbQ}^0=\dbQ^*$. Taking $a'=0$ in
\eqref{eq-static-stability} gives the control and law estimates in
\eqref{eq-newton-strong-equilibrium}. Integrating the optimal-policy BSDE under
$\widehat{\dbQ}^a$, and using
$\nabla\ell(\beta^*)+Z^*=0$, gives directly
\[
 J(\widehat{\dbQ}^a)-J(\dbQ^*)
 =\dbE^{\widehat{\dbQ}^a}\!\int_0^T
 \bigl[\ell(\widehat\beta_t^a)-\ell(\beta_t^*)
       -\nabla\ell(\beta_t^*)\cdot
            (\widehat\beta_t^a-\beta_t^*)\bigr]\,\dd t.
\]
The stochastic integral has zero expectation by the uniform BMO
bounds and change of measure.
For any differentiable convex $\ell$,
\[
 0\le\ell(b)-\ell(b')-\nabla\ell(b')\cdot(b-b')
 \le(\nabla\ell(b)-\nabla\ell(b'))\cdot(b-b').
\]
Apply this with $b=\widehat\beta^a$ and $b'=\beta^*$, take expectation
under $\widehat{\dbQ}^a$, and use conditional Cauchy--Schwarz and the
uniform BMO change-of-measure bounds. The first two estimates in
\eqref{eq-static-stability} give
$J(\widehat{\dbQ}^a)-J(\dbQ^*)\le Ca^2$. Substituting $a=e^{-s}$
finishes the proof.
\end{proof}

\Needspace{8\baselineskip}
\subsection{\texorpdfstring{The general natural gradient flow}{The general natural gradient flow}}\label{sec-natural-proofs}
Here a nonlinear law potential or a mismatch between the mirror and
running costs prevents the affine formula of the Newton flow in general. We first
obtain a priori estimates and then use them to extend the locally
constructed mirror-coordinate flow.

\Needspace{6\baselineskip}
\subsubsection{\texorpdfstring{Uniform BMO bounds}{Uniform BMO bounds}}

The following estimate is inspired by the weighted Csiszár--Kullback--Pinsker inequality of Bolley and Villani \cite{BolleyVillani2005}.

\begin{lem}[One-sided Bolley--Villani inequality]
\label{lem-MF-OsBV}
Let $\dbQ^1\in\cP(\Omega)$ and let
$V:\Omega\to\mathbb{R}$ be measurable. Assume that, for some $a>0$,
\[
M_a:=\langle e^{aV},\dbQ^1\rangle<\infty,
\qquad
N_2:=\langle (V_-)^2,\dbQ^1\rangle<\infty,
\]
where \(V_-:=\max\{-V,0\}.\) Set
\[
R_a:=\frac{M_a}{ae}+\sqrt{N_2}
\]
and
\[
\sigma_a^2
:=
e^{aR_a/2}
\left[
2N_2
+
\frac{32e^{-2}}{a^2}M_a
+
2R_a^2(1+M_a)
\right].
\]
Then, for every probability measure $\dbQ^2$ satisfying $
H(\dbQ^2\mid\dbQ^1)<\infty$,
one has
\[
\langle V,\dbQ^2-\dbQ^1\rangle
\le
\sqrt{2\sigma_a^2H(\dbQ^2\mid\dbQ^1)}
+
\frac{2}{a}H(\dbQ^2\mid\dbQ^1).
\]
Here, the left-hand side is understood in the extended sense whenever
$\langle V_-,\dbQ^2\rangle=\infty$.
\end{lem}

The next inequality transfers a conditional cost bound between nearby laws. It is used for integrability along the learning scheme, not for particle concentration.
\begin{proof}
\phantomsection\label{proof-one-sided-entropy}

\revstep{1}{Integrability and centering}
For every $x\ge0$,
\[
x\le \frac{1}{ae}e^{ax}.
\]
Consequently,
\[
\langle V_+,\dbQ^1\rangle
\le
\frac{1}{ae}\langle e^{aV},\dbQ^1\rangle
=
\frac{M_a}{ae}.
\]
Moreover, by the Cauchy--Schwarz inequality,
\[
\langle V_-,\dbQ^1\rangle
\le
\left\langle (V_-)^2,\dbQ^1\right\rangle^{1/2}
=
\sqrt{N_2}.
\]
Thus, $V\in L^1(\dbQ^1)$. Setting
$m:=\langle V,\dbQ^1\rangle$,
we obtain
\[
|m|
\le
\langle V_++V_-,\dbQ^1\rangle
\le
\frac{M_a}{ae}+\sqrt{N_2}
=
R_a.
\]
Let \(X:=V-m\) and define, for $\lambda\in[0,a/2]$, \(\Lambda(\lambda) := \log\langle e^{\lambda X},\dbQ^1\rangle.\) Notice that
\[
\Lambda(0)=0,
\qquad
\Lambda'(0)=\langle X,\dbQ^1\rangle=0.
\]

\revstep{2}{Control the log-Laplace transform}
We first establish a uniform bound on $\Lambda''$. For
$s\in[0,a/2]$, define the tilted probability measure $\dbQ_s^1$ by
\[
\frac{\d\dbQ_s^1}{\d\dbQ^1}
=
\frac{e^{sX}}
{\langle e^{sX},\dbQ^1\rangle}.
\]
The estimates below imply that
\[
\langle X^2e^{sX},\dbQ^1\rangle<\infty,
\qquad
0\le s\le\frac a2.
\]
Therefore, $\Lambda$ is twice differentiable on $(0,a/2)$, with
\[
\Lambda''(s)
=
\operatorname{Var}_{\dbQ_s^1}(X)
\le
\langle X^2,\dbQ_s^1\rangle.
\]
Since $\langle X,\dbQ^1\rangle=0$, Jensen's inequality gives
\[
\langle e^{sX},\dbQ^1\rangle
\ge
\exp\left(s\langle X,\dbQ^1\rangle\right)
=
1.
\]
Therefore,
\[
\Lambda''(s)
\le
\langle X^2e^{sX},\dbQ^1\rangle.
\]
Using $X=V-m$, $|m|\le R_a$, and $s\le a/2$, we obtain
\[
e^{sX}
=
e^{s(V-m)}
\le
e^{s|m|}e^{sV}
\le
e^{aR_a/2}e^{sV}.
\]
Furthermore,
\[
(V-m)^2\le 2V^2+2m^2.
\]
It follows that
\[
\Lambda''(s)
\le
2e^{aR_a/2}
\left(
\langle V^2e^{sV},\dbQ^1\rangle
+
m^2\langle e^{sV},\dbQ^1\rangle
\right).
\]
We estimate the two terms separately. On the set $\{V<0\}$,
\[
V^2e^{sV}\le (V_-)^2.
\]
On the set $\{V\ge0\}$, since $s\le a/2$, we have \(V^2e^{sV} = V^2e^{-(a-s)V}e^{aV}.\) For every $\gamma>0$,
\[
\sup_{x\ge0}x^2e^{-\gamma x}
=
\frac{4e^{-2}}{\gamma^2}.
\]
Taking $\gamma=a-s\ge a/2$, we obtain
\[
\sup_{x\ge0}x^2e^{-(a-s)x}
\le
\frac{16e^{-2}}{a^2}.
\]
Hence,
\[
\langle V^2e^{sV},\dbQ^1\rangle
\le
N_2+\frac{16e^{-2}}{a^2}M_a.
\]
On the other hand,
\begin{equation*}
\langle e^{sV},\dbQ^1\rangle = \langle e^{sV}\mathbf{1}_{\{V<0\}},\dbQ^1\rangle + \langle e^{sV}\mathbf{1}_{\{V\ge0\}},\dbQ^1\rangle \le \langle\mathbf{1}_{\{V<0\}},\dbQ^1\rangle + \langle e^{aV}\mathbf{1}_{\{V\ge0\}},\dbQ^1\rangle \le 1+M_a.
\end{equation*}
Since $m^2\le R_a^2$, it follows that
\begin{align*}
\Lambda''(s)\le 2e^{aR_a/2}\left[N_2+\frac{16e^{-2}}{a^2}M_a+R_a^2(1+M_a)\right]=e^{aR_a/2}\left[2N_2+\frac{32e^{-2}}{a^2}M_a+2R_a^2(1+M_a)\right]=\sigma_a^2.
\end{align*}
Consequently, for every $\lambda\in[0,a/2]$,
\begin{equation*}
\Lambda(\lambda) = \Lambda(0)+\lambda\Lambda'(0) + \int_0^\lambda(\lambda-s)\Lambda''(s)\,\dd s \le \frac{\sigma_a^2\lambda^2}{2}.
\end{equation*}

\revstep{3}{Entropy inequality by truncation}
We now justify the use of the entropy variational inequality for the
possibly unbounded random variable $X$. Fix
$0<\lambda\le a/2$. First observe that
\[
e^{\lambda X_+}
\le
1+e^{\lambda X},
\]
and
\[
\langle e^{\lambda X},\dbQ^1\rangle
=
e^{-\lambda m}\langle e^{\lambda V},\dbQ^1\rangle
<\infty.
\]
Thus,
\[
\langle e^{\lambda X_+},\dbQ^1\rangle<\infty.
\]

For $N\ge1$, apply the entropy variational inequality to the bounded
random variable $\lambda(X_+\wedge N)$:
\[
\lambda\langle X_+\wedge N,\dbQ^2\rangle
\le
H(\dbQ^2\mid\dbQ^1)
+
\log\langle e^{\lambda(X_+\wedge N)},\dbQ^1\rangle.
\]
Since
\[
e^{\lambda(X_+\wedge N)}
\le
e^{\lambda X_+},
\]
the right-hand side is bounded independently of $N$. Letting
$N\to\infty$ and using the monotone convergence theorem, we obtain
\[
\langle X_+,\dbQ^2\rangle<\infty.
\]
If $\langle X_-,\dbQ^2\rangle=\infty$,
then
$\langle X,\dbQ^2\rangle=-\infty$,
and the desired upper bound is immediate. We may therefore assume
that $
X\in L^1(\dbQ^2)$.
For $N\ge1$, define the two-sided truncation \(X_N:=(-N)\vee(X\wedge N).\) The entropy variational inequality gives
\[
\lambda\langle X_N,\dbQ^2\rangle
\le
H(\dbQ^2\mid\dbQ^1)
+
\log\langle e^{\lambda X_N},\dbQ^1\rangle.
\]
Since $
|X_N|\le |X|$,
we have
\[
\langle X_N,\dbQ^2\rangle
\longrightarrow
\langle X,\dbQ^2\rangle
\]
by the dominated convergence theorem. Moreover,
\[
e^{\lambda X_N}
\le
e^{\lambda X_+},
\]
and hence, again by the dominated convergence theorem,
\[
\langle e^{\lambda X_N},\dbQ^1\rangle
\longrightarrow
\langle e^{\lambda X},\dbQ^1\rangle.
\]
Letting $N\to\infty$, we obtain
\[
\lambda\langle X,\dbQ^2\rangle
\le
H(\dbQ^2\mid\dbQ^1)
+
\log\langle e^{\lambda X},\dbQ^1\rangle.
\]
Since
\[
\langle X,\dbQ^2\rangle
=
\langle V-m,\dbQ^2\rangle
=
\langle V,\dbQ^2-\dbQ^1\rangle,
\]
the preceding log-Laplace estimate yields
\[
\lambda\langle V,\dbQ^2-\dbQ^1\rangle
\le
H(\dbQ^2\mid\dbQ^1)
+
\frac{\sigma_a^2\lambda^2}{2}.
\]
Dividing by $\lambda$, we find
\[
\langle V,\dbQ^2-\dbQ^1\rangle
\le
\frac{H(\dbQ^2\mid\dbQ^1)}{\lambda}
+
\frac{\sigma_a^2\lambda}{2}.
\]

\revstep{4}{Optimize the parameter}
Set \(H:=H(\dbQ^2\mid\dbQ^1).\) If $H=0$, then $\dbQ^2=\dbQ^1$, and the conclusion is immediate.
Suppose henceforth that $H>0$, and define
\[
\lambda_*:=\sqrt{\frac{2H}{\sigma_a^2}}.
\]
If $\lambda_*\le a/2$, choosing $\lambda=\lambda_*$ gives
\[
\langle V,\dbQ^2-\dbQ^1\rangle
\le
\sqrt{2\sigma_a^2H}.
\]
If $\lambda_*>a/2$, choosing $\lambda=a/2$ gives
\[
\langle V,\dbQ^2-\dbQ^1\rangle
\le
\frac{2H}{a}
+
\frac{a\sigma_a^2}{4}.
\]
Moreover, $\lambda_*>a/2$ implies
\[
H>\frac{a^2\sigma_a^2}{8},
\]
and therefore
\[
\frac{a\sigma_a^2}{4}
\le
\sqrt{2\sigma_a^2H}.
\]
Thus, in both cases,
\[
\langle V,\dbQ^2-\dbQ^1\rangle
\le
\sqrt{2\sigma_a^2H}
+
\frac{2H}{a}.
\]
Recalling that
\[
H=H(\dbQ^2\mid\dbQ^1),
\]
we obtain
\[
\langle V,\dbQ^2-\dbQ^1\rangle
\le
\sqrt{2\sigma_a^2H(\dbQ^2\mid\dbQ^1)}
+
\frac{2}{a}H(\dbQ^2\mid\dbQ^1).
\]
This completes the proof.
\end{proof}

\begin{prop}[Uniform BMO estimate]
\label{prop-MF-bounded-psi}
We work under the same assumptions and notation as in
Theorem~\ref{thm-natural-structure}. Then, for every $S>0$, there exist an explicit
constant $h_S>0$ and a constant $C_S>0$ such that for every $0<h\le h_S$ and every $n\ge0$
satisfying $nh\le S$,
\[
\left\|\beta^{h,n}\right\|_{\BMO(\dbQ^{h,n})}
\le
C_S,
\]
where $C_S$ depends only on
$(S,T,c_\ell,C_\ell,c_\psi,C_\psi,C_0,C_F,L_F)$.
\end{prop}
For different mirror and running costs, the uniform Newton bound is unavailable. We instead control the conditional energies on each finite learning interval.
\begin{proof}
\phantomsection\label{proof-natural-bmo}

For every $0<h<h_0$ and $n\ge0$, set \(B^{h,n} := \left\|\beta^{h,n}\right\|_{\BMO(\dbQ^{h,n})}^2.\) \revkeepstep
\noindent
\textit{Step 1: The Bolley--Villani estimate.}
Since $(\dbQ^{h,n},\beta^{h,n})\in\cA_\ell$, we have
$\beta^{h,n}\in\BMO$. By the invariance of the BMO
property under Girsanov transformations,
$\beta^{h,n}\in\BMO(\dbQ^{h,n})$. Hence $B^{h,n}<\infty$.
No uniform bound on $B^{h,n}$ is used at this stage.

\revsubstep{1a}{Variational increment bound}
By the minimality of $\dbQ^{h,n+1}$,
\begin{align}
H_\psi\bigl(
\dbQ^{h,n+1}\mid\dbQ^{h,n}
\bigr)
\nonumber
&\le
h\Bigl(
-
H_\ell\bigl(
\dbQ^{h,n+1}\mid\dbQ^{h,n}
\bigr)
+
\left\langle
-\delta H_\ell(\dbQ^{h,n},\cdot),
\dbQ^{h,n+1}-\dbQ^{h,n}
\right\rangle
\nonumber\\
&\hspace{37mm}
+
F(\dbQ^{h,n})
-
F(\dbQ^{h,n+1})
\Bigr).
\label{eq-MF-linearized-increment}
\end{align}
Moreover, the lower Hessian bound on $\psi$ gives
\begin{equation}
\label{eq-relative-entropy-comparison}
H\bigl(
\dbQ^{h,n+1}\mid\dbQ^{h,n}
\bigr)
\le
\frac{1}{c_\psi}
H_\psi\bigl(
\dbQ^{h,n+1}\mid\dbQ^{h,n}
\bigr).
\end{equation}

We aim to apply Lemma~\ref{lem-MF-OsBV} with
\[
V=-\delta H_\ell(\dbQ^{h,n},\cdot),
\qquad
\dbQ^1=\dbQ^{h,n},
\qquad
\dbQ^2=\dbQ^{h,n+1}.
\]
\revsubstep{1b}{Positive exponential moment}
First verify the integrability assumptions of
Lemma~\ref{lem-MF-OsBV}. Set
\[
m(t,\omega)
:=
\inf_{b\in\dbR^d}\ell(t,\omega,b).
\]
Since $\ell$ is convex and $C_\ell$-smooth, we have
\[
\left|\nabla\ell(t,X,b)\right|^2
\le
2C_\ell
\left(
\ell(t,X,b)-m(t,X)
\right).
\]
Moreover, by the $c_\ell$-strong convexity of $\ell$,
\[
m(t,\omega)
\ge
\ell(t,\omega,0)
-
\frac{1}{2c_\ell}
\left|\nabla\ell(t,\omega,0)\right|^2.
\]
By Assumption~\ref{ass-ini}, the constant
\[
C_{\ell,0}
:=
\operatorname*{ess\,sup}_{(t,\omega)\in[0,T]\times\Omega}
\left(
\ell^-(t,\omega,0)
+
\frac{1}{c_\ell}
\left|\nabla\ell(t,\omega,0)\right|^2
\right)
\le
C_0+\frac{C_0^2}{c_\ell}.
\]
is finite, and in particular $-m(t,\omega)\le C_{\ell,0}$.
Set
\[
\bar a
:=
\min\left\{
1,
\frac{1}{2C_\ell},
\frac{\log 2}{T(C_{\ell,0}+1)}
\right\}.
\]
For every $0<a\le\bar a$,
\[
-a\ell(t,X,b)
+
\frac{a^2}{2}
\left|\nabla\ell(t,X,b)\right|^2
\le
-am(t,X)
\le
aC_{\ell,0}.
\]
Under $\dbQ^{h,n}$, we have
\[
\delta H_\ell(\dbQ^{h,n},X)
=
\int_0^T
\ell(t,X,\beta_t^{h,n})\,\dd t
+
\int_0^T
\nabla\ell(\beta_t^{h,n})
\cdot\dd W_t^{\beta^{h,n}}.
\]
Hence, for every $0<a\le\bar a$,
\begin{equation*}
e^{-a\delta H_\ell(\dbQ^{h,n},X)} = D_T \left( -a\int_0^\cdot \nabla\ell(\beta_t^{h,n}) \cdot\dd W_t^{\beta^{h,n}} \right) \times \exp\left( \int_0^T \left[ -a\ell(\beta_t^{h,n}) + \frac{a^2}{2} \left|\nabla\ell(\beta_t^{h,n})\right|^2 \right]\dd t \right).
\end{equation*}
The stochastic exponential in the last display is a nonnegative local
martingale and therefore a supermartingale. It follows that
\[
\left\langle
 e^{-a\delta H_\ell(\dbQ^{h,n},\cdot)},
 \dbQ^{h,n}
\right\rangle
\le
e^{aTC_{\ell,0}}.
\]
Define
\begin{equation}
\label{eq-choice-a}
a^{h,n}
:=
\frac{\bar a}{1+B^{h,n}}.
\end{equation}
Then $a^{h,n}\le\bar a$, and by the definition of $\bar a$,
\begin{equation}
\label{eq-Ma-bound}
M_{a^{h,n}}
:=
\left\langle
 e^{-a^{h,n}\delta H_\ell(\dbQ^{h,n},\cdot)},
 \dbQ^{h,n}
\right\rangle
\le2.
\end{equation}

\revsubstep{1c}{Second moment of the negative part}
For the negative part of
$V=-\delta H_\ell(\dbQ^{h,n},\cdot)$, we have \(V_- = \left( \delta H_\ell(\dbQ^{h,n},\cdot) \right)_+.\) By Assumption~\ref{ass-ini} and the upper Hessian bound on $\ell$,
\begin{equation}
\ell^+(t,X,b)
\le
C_{\ell,+}(1+|b|^2),
\qquad
\left|\nabla\ell(t,X,b)\right|^2
\le
C_{\nabla\ell}(1+|b|^2),
\label{eq-ell-+-nabla}
\end{equation}
where $C_{\ell,+}:=\max\{C_0+\frac{C_0^2}{2},\frac{1+C_\ell}{2}\}$ and $C_{\nabla\ell}:=2\max\{C_0^2,C_\ell^2\}$. The BMO energy inequality gives
\[
\dbE^{\dbQ^{h,n}}
\left[
\left(
\int_0^T|\beta_t^{h,n}|^2\,\dd t
\right)^2
\right]
\le
2(B^{h,n})^2,
\qquad
\dbE^{\dbQ^{h,n}}
\left[
\int_0^T|\beta_t^{h,n}|^2\,\dd t
\right]
\le
B^{h,n}.
\]
Therefore,
\begin{align}
N_2
&:=
\left\langle
\left(
\delta H_\ell(\dbQ^{h,n},\cdot)
\right)_+^2,
\dbQ^{h,n}
\right\rangle
\nonumber\\
&\le
2\dbE^{\dbQ^{h,n}}
\left[
\left(
\int_0^T
\ell^+(\beta_t^{h,n})\,\dd t
\right)^2
\right]
+
2\dbE^{\dbQ^{h,n}}
\left[
\left|
\int_0^T
\nabla\ell(\beta_t^{h,n})
\cdot\dd W_t^{\beta^{h,n}}
\right|^2
\right]
\nonumber\\
&\le
4C_{\ell,+}^2
\left(
T^2+2(B^{h,n})^2
\right)
+
2C_{\nabla\ell}
\left(
T+B^{h,n}
\right)
\nonumber\\
&\le
C_N(1+B^{h,n})^2,
\label{eq-N2-bound}
\end{align}
where \(C_N := 4C_{\ell,+}^2(T^2+2) + 2C_{\nabla\ell}(T+1).\) \revsubstep{1d}{Assemble the Bolley--Villani estimate}
Recalling that
\[
R_a=\frac{M_a}{ae}+\sqrt{N_2},
\]
we deduce from \eqref{eq-choice-a}, \eqref{eq-Ma-bound}, and
\eqref{eq-N2-bound} that
\begin{equation}
\label{eq-Ra-bound}
R_{a^{h,n}}
\le
C_R(1+B^{h,n}),
\end{equation}
where
\[
C_R
:=
\frac{2}{e\bar a}+\sqrt{C_N}.
\]
In particular,
\[
\frac{a^{h,n}R_{a^{h,n}}}{2}
\le
\frac{\bar a C_R}{2}.
\]

Substituting \eqref{eq-Ma-bound}, \eqref{eq-N2-bound}, and
\eqref{eq-Ra-bound} into the definition of $\sigma_a^2$, we obtain
\begin{equation}
\label{eq-sigma-bound}
\sigma_{a^{h,n}}^2
\le
C_{\mathrm{BV}}
(1+B^{h,n})^2,
\end{equation}
where one may take
\[
C_{\mathrm{BV}}
:=
e^{\bar a C_R/2}
\left(
2C_N
+
\frac{64e^{-2}}{\bar a^2}
+
6C_R^2
\right).
\]

Lemma~\ref{lem-MF-OsBV},
\eqref{eq-relative-entropy-comparison}, and
\eqref{eq-sigma-bound} now yield
\begin{equation}\label{eq-BV-final}
\bigl\langle-\delta H_\ell(\dbQ^{h,n}),\dbQ^{h,n+1}-\dbQ^{h,n}\bigr\rangle
\le (1+B^{h,n})\left[\sqrt{\frac{2C_{\mathrm{BV}}}{c_\psi}H_\psi(\dbQ^{h,n+1}\mid\dbQ^{h,n})}
+\frac{2H_\psi(\dbQ^{h,n+1}\mid\dbQ^{h,n})}{\bar a c_\psi}\right].
\end{equation}
\revkeepstep
\noindent
\textit{Step 2: The one-step estimate.}

\revsubstep{2a}{Entropy and total-variation increments}
Since $\|\delta F\|_\infty\le C_F$, Pinsker's inequality and
\eqref{eq-relative-entropy-comparison} give
\[
\left|
F(\dbQ^{h,n+1})-F(\dbQ^{h,n})
\right|
\le
C_F\sqrt{\frac{2}{c_\psi}}
\sqrt{
H_\psi\bigl(
\dbQ^{h,n+1}\mid\dbQ^{h,n}
\bigr)
}.
\]
Dropping the nonpositive term
\[
-hH_\ell\bigl(
\dbQ^{h,n+1}\mid\dbQ^{h,n}
\bigr)
\]
in \eqref{eq-MF-linearized-increment}, and using
\eqref{eq-BV-final}, we obtain
\begin{align*}
H_\psi\bigl(
\dbQ^{h,n+1}\mid\dbQ^{h,n}
\bigr)
&\le
h\frac{2}{\bar a c_\psi}
(1+B^{h,n})
H_\psi\bigl(
\dbQ^{h,n+1}\mid\dbQ^{h,n}
\bigr)
\\
&+
h
\left(
\sqrt{\frac{2C_{\mathrm{BV}}}{c_\psi}}
+
C_F\sqrt{\frac{2}{c_\psi}}
\right)
(1+B^{h,n})
\times
\sqrt{
H_\psi\bigl(
\dbQ^{h,n+1}\mid\dbQ^{h,n}
\bigr)
}.
\end{align*}
Here we used $1+B^{h,n}\ge1$ to absorb the contribution of $F$.
Consequently, whenever
\begin{equation}
\label{eq-local-smallness}
h\frac{2}{\bar a c_\psi}
(1+B^{h,n})
\le
\frac12,
\end{equation}
we have
\begin{align*}
\sqrt{
H_\psi\bigl(
\dbQ^{h,n+1}\mid\dbQ^{h,n}
\bigr)
}
&\le
2h
\left(
\sqrt{\frac{2C_{\mathrm{BV}}}{c_\psi}}
+
C_F\sqrt{\frac{2}{c_\psi}}
\right)
(1+B^{h,n}).
\label{eq-relative-H-bound}
\end{align*}
Therefore, by Pinsker's inequality,
\begin{align*}
\left\|
\dbQ^{h,n+1}-\dbQ^{h,n}
\right\|_{\TV}
&\le
2h\sqrt{\frac{2}{c_\psi}}
\left(
\sqrt{\frac{2C_{\mathrm{BV}}}{c_\psi}}
+
C_F\sqrt{\frac{2}{c_\psi}}
\right)
(1+B^{h,n}).
\label{eq-TV-bound}
\end{align*}

\revsubstep{2b}{Relate BMO to the conditional value}
For $\tau\in\cT$, define
\[
Y_\tau^{h,n}
:=
\dbE^{\dbQ^{h,n}}
\left[
\delta F(\dbQ^{h,n},X)
+
\int_\tau^T
\ell(\beta_t^{h,n})\,\dd t
\Bigm|\cF_\tau
\right],
\]
and set
\begin{equation}
\label{eq-def-K-n}
\mathcal K^{h,n}
:=
1+TC_{\ell,0}+C_F
+
\sup_{\tau\in\cT}
\left\|
\left(Y_\tau^{h,n}\right)_+
\right\|_\infty.
\end{equation}

On the other hand, by the $c_\ell$-strong convexity of $\ell$,
\[
\ell(t,X,b)
\ge
\ell(t,X,0)
+
\nabla\ell(t,X,0)\cdot b
+
\frac{c_\ell}{2}|b|^2.
\]
Using Young's inequality,
\[
\nabla\ell(t,X,0)\cdot b
\ge
-\frac{c_\ell}{4}|b|^2
-\frac{1}{c_\ell}
\left|\nabla\ell(t,X,0)\right|^2,
\]
and therefore
\[
\frac{c_\ell}{4}|b|^2
\le
\ell(t,X,b)
+
\ell^-(t,X,0)
+
\frac{1}{c_\ell}
\left|\nabla\ell(t,X,0)\right|^2.
\]
Taking $b=\beta_t^{h,n}$, integrating over $[\tau,T]$, and taking the
conditional expectation under $\dbQ^{h,n}$ yield
\begin{equation*}
\frac{c_\ell}{4} \dbE^{\dbQ^{h,n}} \left[ \int_\tau^T |\beta_t^{h,n}|^2\,\dd t \Bigm|\cF_\tau \right] \le Y_\tau^{h,n} - \dbE^{\dbQ^{h,n}} \left[ \delta F(\dbQ^{h,n},X) \Bigm|\cF_\tau \right] + TC_{\ell,0} \le \left(Y_\tau^{h,n}\right)_+ + C_F + TC_{\ell,0}.
\end{equation*}
Taking the essential supremum over $\tau\in\cT$ gives
\begin{equation}
\label{eq-B-by-E}
B^{h,n}
\le
\frac4{c_\ell}\mathcal K^{h,n},
\qquad
1+B^{h,n}
\le
\left(
1+\frac4{c_\ell}
\right)
\mathcal K^{h,n}.
\end{equation}

\revsubstep{2c}{One-step energy recursion}
Applying \eqref{eq-JKO-7} with
$
(\dbQ,\beta)=(\dbQ^{h,n},\beta^{h,n}),
$
we obtain
\begin{align*}
&\dbE^{\dbQ^{h,n+1}}
\left[
\delta F(\dbQ^{h,n+1},X)
+
\int_\tau^T
\left(
\ell(t,X,\beta_t^{h,n+1})
+
\frac1h
d_\psi\bigl(
\beta_t^{h,n+1},
\beta_t^{h,n}
\bigr)
\right)\dd t
\Bigm|\cF_\tau
\right]
\\
&\qquad\le
\dbE^{\dbQ^{h,n}}
\left[
\delta F(\dbQ^{h,n+1},X)
+
\int_\tau^T
\ell(t,X,\beta_t^{h,n})\,\dd t
\Bigm|\cF_\tau
\right],
\end{align*}
where
\[
d_\psi(b,b')
:=
\psi(b)-\psi(b')
-\nabla\psi(b')\cdot(b-b').
\]
Since \(d_\psi(b,b')\ge0,\) we may drop the Bregman term on the left-hand side. Using
Assumption~\ref{ass-pe}, we obtain, for every $\tau\in\cT$,
\[
Y_\tau^{h,n+1}
\le
Y_\tau^{h,n}
+
L_F
\left\|
\dbQ^{h,n+1}-\dbQ^{h,n}
\right\|_{\TV}.
\]
Therefore,
\[
\left(Y_\tau^{h,n+1}\right)_+
\le
\left(Y_\tau^{h,n}\right)_+
+
L_F
\left\|
\dbQ^{h,n+1}-\dbQ^{h,n}
\right\|_{\TV},
\]
and hence
\[
\mathcal K^{h,n+1}
\le
\mathcal K^{h,n}
+
L_F
\left\|
\dbQ^{h,n+1}-\dbQ^{h,n}
\right\|_{\TV}.
\]
Consequently, whenever \eqref{eq-local-smallness} holds,
\begin{equation}
\label{eq-energy-recursion}
\mathcal K^{h,n+1}
\le
(1+\Gamma h)\mathcal K^{h,n},
\end{equation}
where
\[
\Gamma
:=
2L_F\sqrt{\frac{2}{c_\psi}}
\left(
\sqrt{\frac{2C_{\mathrm{BV}}}{c_\psi}}
+
C_F\sqrt{\frac{2}{c_\psi}}
\right)
\left(
1+\frac4{c_\ell}
\right).
\]
\revkeepstep
\noindent
\textit{Step 3: Bootstrap.}

\revsubstep{3a}{Choose the bootstrap radius}
By Assumption~\ref{ass-ini} and Lemma~\ref{lem-BMO-reverse-holder}, we have
\[
\sup_{\tau\in\cT}
\left\|
\dbE^{\dbQ^0}
\left[
\int_\tau^T
\ell^+(t,X,\beta_t^0)\,\dd t
\Bigm|\cF_\tau
\right]
\right\|_\infty
<\infty.
\]
Together with $\|\delta F\|_\infty\le C_F$, this implies
 ${\mathcal K^{h,0}}<\infty$ . Note
\[
 {\mathcal K^0}
:=
{\mathcal K^{h,0}} .
\]
Fix $S>0$ and set \(K_S := 2e^{\Gamma S}{\mathcal K^0}.\) Choose
\[
h_S
{:=\frac12\min}\left\{h_0,
\frac{
\bar a c_\psi
}{
4
\left(
1+4/c_\ell
\right)
K_S
}\right\}.
\]

\revsubstep{3b}{Exclude the first exit}
Fix $0<h\le h_S$ and define the first exit index
\[
M_h
:=
\inf\left\{
n\in\{0,\ldots,\left\lfloor\frac{S}{h}\right\rfloor\}:
\mathcal K^{h,n}>K_S
\right\},
\]
with the convention that $M_h=\left\lfloor\frac{S}{h}\right\rfloor+1$ if the set is empty.

Since \(\mathcal K^{h,0} = \mathcal K^0 < K_S,\) we have $M_h\ge1$. Suppose, by contradiction, that
\[
 {M_h\le\lfloor S/h\rfloor.} 
\]
Then
\[
\mathcal K^{h,k}\le K_S,
\qquad
0\le k<M_h.
\]
By \eqref{eq-B-by-E},
\[
1+B^{h,k}
\le
\left(
1+\frac4{c_\ell}
\right)
K_S,
\qquad
0\le k<M_h.
\]
Hence, by the definition of $h_S$,
\[
h\frac{2}{\bar a c_\psi}
(1+B^{h,k})
\le
\frac12,
\qquad
0\le k<M_h.
\]
Thus, the local smallness condition
\eqref{eq-local-smallness} holds for every $k<M_h$, and
\eqref{eq-energy-recursion} can be iterated. We obtain
\begin{equation}
\mathcal K^{h,M_h}
\le
(1+\Gamma h)^{M_h}\mathcal K^{h,0}
\le
e^{\Gamma S}\mathcal K^0
=
\frac{K_S}{2}.
\label{eq-K-n-bound}
\end{equation}
This contradicts the definition of $M_h$, according to which
$\mathcal K^{h,M_h}>K_S$.
Therefore,
\[
M_h=
\left\lfloor\frac{S}{h}\right\rfloor+1.
\]
Consequently, for every $n$ satisfying $nh\le S$, the local
smallness condition \eqref{eq-local-smallness} holds and
\begin{equation*}
\mathcal K^{h,n}
\le
(1+\Gamma h)^n\mathcal K^{h,0}
\le
e^{\Gamma S}\mathcal K^0.
\label{eq-S-bounded-E} 
\end{equation*}
Finally, by \eqref{eq-B-by-E},
\[
\left\|
\beta^{h,n}
\right\|_{\BMO(\dbQ^{h,n})}^2
=
B^{h,n}
\le
\frac4{c_\ell}
e^{\Gamma S}\mathcal K^0:=C^2_S,
\qquad
nh\le S.
\]
Since $\mathcal K^0$ is independent of $h$ and $n$, this proves the uniform BMO bound.
\end{proof}

\Needspace{6\baselineskip}
\subsubsection{\texorpdfstring{The BMO flow and its discretisation}{The BMO flow and its discretisation}}

\begin{lem}[Evaluation on a BMO ball]\label{lem-fixed-policy-evaluation}
Let Assumption~\ref{ass-rc}, Assumption~\ref{ass-ini} and
Assumption~\ref{ass-rc-quadratic} hold. Let $\|\beta\|_{\BMO}\le M$ and
$\|k\|_\infty\le K$. The BSDE
\[
 \dd Y_t=-[\beta_t\cdot Z_t+\ell(t,X,\beta_t)]\,\dd t+Z_t\,\dd X_t,
 \qquad Y_T=k,
\]
has a unique solution in $\mathcal S^\infty\times\BMO$.
There is a finite constant, depending only on $M,K,T$ and the
running-cost bounds, such that
$\|Y\|_\infty+\|Z\|_{\BMO}\le C$.
\end{lem}
Here the policy is held fixed. This evaluation BSDE is different from the optimal BSDE of a penalised proximal step; comparing the two is essential for consistency.
\begin{proof}
\phantomsection\label{proof-fixed-policy-evaluation}

\textit{Step 1: conditional evaluation.}
Let $\dbQ= D(\beta)\dbP$. Reverse H\"older and BMO invariance give
$\|\beta\|_{\BMO(\dbQ)}\le M'$ with $M'$ depending only on $M$.
Since $|\ell(t,X,b)|\le C(1+|b|^2)$,
\[
 A:=\sup_{\tau\in\cT}\left\|
 \dbE^{\dbQ}\left[\int_\tau^T|\ell(t,X,\beta_t)|\,\dd t
                    \,\middle|\,\cF_\tau\right]\right\|_\infty
 \le C(T+(M')^2).
\]
Define $Y_t=\dbE^{\dbQ}[k+\int_t^T\ell(r,X,\beta_r)\,\dd r\mid\cF_t]$.
The BMO energy inequality gives square integrability of the total
terminal cost under $\dbQ$. Martingale representation supplies $Z$,
with $\dd Y=-\ell(\beta)\dd t+Z\dd W^\beta$ and
$\|Y\|_\infty\le K+A$.

\textit{Step 2: the martingale estimate.}
Apply It\^o's formula to $Y^2$ between a stopping time and a
localising sequence, then pass to the limit by the preceding
integrability and boundedness. It follows that
\[
 \|Z\|_{\BMO(\dbQ)}^2\le K^2+2(K+A)A.
\]
The inverse change of measure and the norm comparison yield
\[
 \|Z\|_{\BMO}^2
 \le (1+\tfrac{\sqrt2}{2}M)^2\,[K^2+2(K+A)A].
\]
A difference of two bounded solutions is a $\dbQ$-martingale with
zero terminal value, since their integrands are BMO under $\dbQ$.
It vanishes, and martingale representation gives uniqueness of $Z$.
\end{proof}

\begin{prop}[Uniform BMO increment estimate]
\label{prop-MF-BMO-continue}
We work under the same assumptions and notation as in Theorem~\ref{thm-natural-structure}. For every $S>0$, there exist $h_S>0$ and $L_S>0$ such
that, for every $0<h\le h_S$ and every $n\ge0$ satisfying $nh\le S$,
\[
\|\beta^{h,n+1}-\beta^{h,n}\|_{\BMO}
\le
L_Sh.
\]
\end{prop}

A bound on each policy is not enough to identify a time derivative. The first-order condition and conditional BSDE estimates give increments proportional to the learning step.
\begin{proof}
\phantomsection\label{proof-natural-increments}

Decrease the step threshold so that $h_S\le1$, and take the
bounds denoted here by $C_S,K_S$ from the bootstrap at horizon $S+1$.
Thus $nh\le S$ also permits every bound at index $n+1$. Only finite
constants are enlarged; the conclusion is still for the prescribed
horizon $S$.

\revstep{1}{Discrete optimality condition}
For every $0<h<h_S$ and $n\ge0$, define
\[
\psi_t^{h,n}(b)
:=
\psi(t,X,b)
-
\psi(t,X,\beta_t^{h,n})
-
\nabla\psi(t,X,\beta_t^{h,n})
\cdot
(b-\beta_t^{h,n}),
\]
and
\[
g^{h,n}(t,X,b)
:=
\ell(t,X,b)
+
\frac1h\psi_t^{h,n}(b).
\]
Let $(Y^{h,n+1},Z^{h,n+1})$ be the solution of the BSDE associated
with the {Bregman proximal} step. Recall that
\[
\beta_t^{h,n+1}
=
\argmin_{b\in\dbR^d}
\left\{
b\cdot Z_t^{h,n+1}
+
\ell(t,X,b)
+
\frac1h\psi_t^{h,n}(b)
\right\}.
\]
The first-order condition gives
\begin{equation}
\label{eq-MF-mirror-discrete-FOC}
\nabla\psi(\beta_t^{h,n+1})
-
\nabla\psi(\beta_t^{h,n})
=
-h
\left(
\nabla\ell(\beta_t^{h,n+1})
+
Z_t^{h,n+1}
\right).
\end{equation}

\revstep{2}{Uniform control of the drift and value}
Choose $p_S>1$ such that \(C_S<\Phi(p_S).\) By Lemma~\ref{lem-BMO-reverse-holder}.3 and Proposition~\ref{prop-MF-bounded-psi}, we obtain
\[
\|\beta^{h,n}\|_{\BMO}^2
\le
\kappa^{p_S}(C_S).
\]
so that
\begin{equation}
\sup_{\substack{0<h\le h_S\\ nh\le S}}
\|\beta^{h,n}\|_{\BMO}
\le
\sqrt{\kappa^{p_S}(C_S)}.
\label{eq-BMO-uni-P-bound}
\end{equation}
Since Assumption~\ref{ass-rc-quadratic} implies
\[
|\nabla\ell(t,X,b)|
\le
|\nabla\ell(t,X,0)|+C_\ell|b|,
\]
we have
\begin{equation}
\label{eq-MF-uniform-nabla-ell}
\sup_{\substack{0<h\le h_S\\nh\le S}}
\|\nabla\ell(\beta^{h,n})\|_{\BMO}
\le C_0\sqrt{T}+C_\ell \sqrt{\kappa^{p_S}(C_S)}
:=C^S_{\ell,\nabla}.
\end{equation}

We distinguish the penalised value $Y^{h,n+1}$ from evaluation
of the preceding policy. Put
\[
 \overline Y_\tau^{h,n}:=
 \dbE^{\dbQ^{h,n}}\left[
 \delta F(\dbQ^{h,n},X)+\int_\tau^T\ell(t,X,\beta_t^{h,n})\,\dd t
                    \,\middle|\,\cF_\tau\right].
\]
By \eqref{eq-def-K-n}--\eqref{eq-K-n-bound},
\begin{equation*}\label{eq-MF-Yn-positive-bound}
 \sup_\tau\|(\overline Y_\tau^{h,n})_+\|_\infty\le K_S/2.
\end{equation*}
Verification for the penalised problem, with $\beta^{h,n}$ as
competitor, uses $\psi^{h,n}(\beta^{h,n})=0$ and gives
\[
 Y_\tau^{h,n+1}
 \le \overline Y_\tau^{h,n}
      +L_F\|\dbQ^{h,n+1}-\dbQ^{h,n}\|_{\TV}
 \le K_S/2+2L_F.
\]
On the other hand, $g^{h,n}\ge\ell\ge-C_{\ell,-}$ and the terminal
value is bounded below by $-C_F$, so
$Y_\tau^{h,n+1}\ge-C_F-TC_{\ell,-}$. Consequently
\begin{equation}\label{eq-MF-uniform-Y}
 \sup_{0<h\le h_S,\,nh\le S}\|Y^{h,n+1}\|_{\mathcal S^\infty}
 \le\max\{K_S/2+2L_F,C_F+TC_{\ell,-}\}=:C_Y^S.
\end{equation}

\revstep{3}{Conditional energy estimate for the BSDE}
We next estimate $Z^{h,n+1}$. Recall that
\[
\mathcal H^{h,n}(t,X,z)
:=
\inf_{b\in\dbR^d}
\left\{
b\cdot z
+
\ell(t,X,b)
+
\frac1h\psi_t^{h,n}(b)
\right\}.
\]
Since \(\psi_t^{h,n}(\beta_t^{h,n})=0,\) testing the infimum at $b=\beta_t^{h,n}$ gives
\begin{equation}
\label{eq-H-upper-by-old-control}
\mathcal H^{h,n}(t,X,Z_t^{h,n+1})
\le
\beta_t^{h,n}\cdot Z_t^{h,n+1}
+
\ell(t,X,\beta_t^{h,n}).
\end{equation}
By \eqref{eq-ell-+-nabla} and \eqref{eq-MF-C-ell-minus}, 
\[
|\ell(t,X,b)|
\le
C_{\ell,-}
+
C_{\ell,+}(1+|b|^2).
\]
Hence, by \eqref{eq-BMO-uni-P-bound},
\begin{equation}
\label{eq-MF-ell-beta-uniform}
\sup_{\substack{0<h\le h_S\\ nh\le S}}\sup_{\tau\in\cT}\left\|\dbE\left[\int_\tau^T|\ell(t,X,\beta_t^{h,n})|\,\dd t\Bigm|\cF_\tau\right]\right\|_\infty\le T(C_{\ell,-}+C_{\ell,+})+C_{\ell,+}\kappa^{p_S}(C_S)=:C^S_\ell.
\end{equation}
Fix $\tau\in\cT$ and set \(U_t^{h,n+1} := Y_t^{h,n+1}+C_Y^S.\) By \eqref{eq-MF-uniform-Y},
\[
0\le U_t^{h,n+1}\le2C_Y^S.
\]
Moreover,
\[
|U_T^{h,n+1}|
=
\left|
\delta F(\dbQ^{h,n+1},X)+C_Y^S
\right|
\le
C_Y^S+C_F.
\]
For every fixed $(h,n)$, Lemma~\ref{lem-BSDE-fp} gives
$Z^{h,n+1}\in\BMO$, so It\^o's formula may be applied to
$|U_t^{h,n+1}|^2$ on $[\tau,T]$. Using the BSDE
\[
\dd Y_t^{h,n+1}
=
-\mathcal H^{h,n}
(t,X,Z_t^{h,n+1})\,\dd t
+
Z_t^{h,n+1}\,\dd X_t
\]
and taking conditional expectations, we obtain
\begin{equation*}
\dbE\left[ \int_\tau^T |Z_t^{h,n+1}|^2\,\dd t \Bigm|\cF_\tau \right] = \dbE\left[ |U_T^{h,n+1}|^2 - |U_\tau^{h,n+1}|^2 + 2\int_\tau^T U_t^{h,n+1} \mathcal H^{h,n} (t,X,Z_t^{h,n+1})\,\dd t \Bigm|\cF_\tau \right].
\end{equation*}
Since $U^{h,n+1}\ge0$, \eqref{eq-H-upper-by-old-control} and
\eqref{eq-MF-uniform-Y} imply
\begin{equation}\label{eq-MF-Z-Ito-pre}
\dbE\!\left[\int_\tau^T|Z_t^{h,n+1}|^2\,\dd t\,\middle|\,\cF_\tau\right]
\le (C_Y^S+C_F)^2+4C_Y^S\dbE\!\left[\int_\tau^T
\bigl(|\beta_t^{h,n}|\,|Z_t^{h,n+1}|+|\ell(t,X,\beta_t^{h,n})|\bigr)\,\dd t\,\middle|\,\cF_\tau\right].
\end{equation}
By Young's inequality,
\[
4C_Y^S
|\beta_t^{h,n}|
|Z_t^{h,n+1}|
\le
\frac12|Z_t^{h,n+1}|^2
+
8{C_Y^S}^2|\beta_t^{h,n}|^2.
\]
Substituting this estimate into
\eqref{eq-MF-Z-Ito-pre}, and using
\eqref{eq-BMO-uni-P-bound} and
\eqref{eq-MF-ell-beta-uniform}, gives
\begin{equation*}
\dbE\left[ \int_\tau^T |Z_t^{h,n+1}|^2\,\dd t \Bigm|\cF_\tau \right] \le (C_Y^S+C_F)^2 + \frac12 \dbE\left[ \int_\tau^T |Z_t^{h,n+1}|^2\,\dd t \Bigm|\cF_\tau \right] + {8{C_Y^S}^2\kappa^{p_S}(C_S)} + 4C_Y^SC^S_{\ell}.
\end{equation*}
Absorbing the second term on the right-hand side, we obtain
\[
\dbE\left[
\int_\tau^T
|Z_t^{h,n+1}|^2\,\dd t
\Bigm|\cF_\tau
\right]
\le
{C_Z^S}^2,
\]
where
\begin{equation*}
\label{eq-MF-Z-constant}
{C_Z^S}^2
:=
2(C_Y^S+C_F)^2
+
{16{C_Y^S}^2\kappa^{p_S}(C_S)}
+
8C_Y^S C_\ell^S.
\end{equation*}
Taking the essential supremum over $\tau\in\cT$ therefore yields
\begin{equation*}
\label{eq-MF-uniform-Z}
\sup_{\substack{0<h\le h_S\\ nh\le S}}
\|Z^{h,n+1}\|_{\BMO}
\le
C_Z^S.
\end{equation*}

\revstep{4}{Control the policy increment}
Combining this estimate with
\eqref{eq-MF-mirror-discrete-FOC} and
\eqref{eq-MF-uniform-nabla-ell}, we obtain
\[
\left\|
\nabla\psi(\beta^{h,n+1})
-
\nabla\psi(\beta^{h,n})
\right\|_{\BMO}
\le
(C^S_{\ell,\nabla}+C_Z^S)h.
\]
Finally, the uniform strong convexity of $\psi$ gives
\[
\left(
\nabla\psi(t,X,b)-\nabla\psi(t,X,b')
\right)
\cdot(b-b')
\ge
c_\psi|b-b'|^2,
\]
and hence
\[
|b-b'|
\le
\frac1{c_\psi}
\left|
\nabla\psi(t,X,b)-\nabla\psi(t,X,b')
\right|.
\]
Thus
\[
\|\beta^{h,n+1}-\beta^{h,n}\|_{\BMO}
\le
L_Sh,
\]
where $
L_S:=\tfrac{C^S_{\ell,\nabla}+C_Z^S}{c_\psi}$.
\end{proof}

\subsubsection{Construction, consistency and global continuation}

The flow is constructed locally in mirror coordinates. A separate consistency bound handles the penalised evaluation term, and the a priori estimates then prevent finite-learning-time breakdown.
\begin{proof}[Proof of Theorem~\ref{thm-natural-structure}, parts \textnormal{(ii)--(iii)}]
\phantomsection\label{proof-natural-existence}

We divide the proof into several steps.

\revkeepstep
\noindent
\textit{Step 1: the mirror coordinate and the limiting vector field.}

\revsubstep{1a}{Mirror coordinates and policy evaluation}
By Assumptions~\ref{ass-rc} and~\ref{ass-rc-quadratic}, for every
$(t,\omega)$, the map \(b\mapsto\nabla\psi(t,\omega,b)\) is strongly monotone and globally Lipschitz. Hence it is bijective on
$\dbR^d$, and its inverse satisfies
\begin{equation}
\label{eq-MF-inverse-psi-Lipschitz}
\left|
(\nabla\psi)^{-1}(q)
-
(\nabla\psi)^{-1}(\widetilde q)
\right|
\le
\frac1{c_\psi}|q-\widetilde q|.
\end{equation}
For any $q\in\BMO$, define
\[
\beta(q):=(\nabla\psi)^{-1}(q),
\qquad
\frac{\dd\dbQ(q)}{\dd\dbP}
:=
D(\beta(q)).
\]
Since $\nabla\psi(0)$ is bounded,
\eqref{eq-MF-inverse-psi-Lipschitz} implies that $\beta(q)\in\BMO$
whenever $q\in\BMO$.
Let $(Y(q),Z(q))$ denote the solution of
\begin{equation}
\label{eq-MF-BSDE-q}
\begin{cases}
\displaystyle
\dd Y_t(q)
=
-
\left(
\beta_t(q)\cdot Z_t(q)
+
\ell(t,X,\beta_t(q))
\right)\dd t
+
Z_t(q)\,\dd X_t,
\\[1ex]
\displaystyle
Y_T(q)
=
\delta F(\dbQ(q),X).
\end{cases}
\end{equation}
Equivalently, under $\dbQ(q)$,
\[
Y_t(q)
=
\delta F(\dbQ(q),X)
+
\int_t^T
\ell(r,X,\beta_r(q))\,\dd r
-
\int_t^T
Z_r(q)\,\dd W_r^{\beta(q)}.
\]
The quadratic growth of $\ell$ and the BMO property of $\beta(q)$
imply that
\[
(Y(q),Z(q))
\in
\mathcal S^\infty\times\BMO.
\]
Since $
\|\delta F(\dbQ,\cdot)\|_\infty\le C_F$,
{Lemma~\ref{lem-fixed-policy-evaluation} yields, for every fixed $R>0$,} 
\begin{equation}
{B^Z_R:=\sup_{\|q\|_\BMO\le R}\|Z(q)\|_{\BMO}}
<
+\infty,
\label{eq-B-Z-R}
\end{equation}

\revsubstep{1b}{Stability of the data}
We next establish the local Lipschitz continuity of $q\mapsto Z(q)$. Let $q,\widetilde q$ satisfy
\begin{equation}\label{eq-q-BMO-R}
{\max\{\|q\|_{\BMO},\|\widetilde q\|_{\BMO}\}\le R.}
\end{equation}
Write
\[
\beta:=\beta(q),
\qquad
\widetilde\beta:=\beta(\widetilde q),
\qquad
\dbQ:=\dbQ(q),
\qquad
\widetilde{\dbQ}:=\dbQ(\widetilde q),
\]
and set
\[
\Delta Y:=Y(q)-Y(\widetilde q),
\qquad
\Delta Z:=Z(q)-Z(\widetilde q),
\]
\[
\Delta\beta:=\beta-\widetilde\beta,
\qquad
\Delta\xi
:=
\delta F(\dbQ,X)
-
\delta F(\widetilde{\dbQ},X).
\]
By the strong monotonicity of $\nabla\psi$,
\[
\left|(\nabla\psi)^{-1}(0)\right|
\le
\frac{|\nabla\psi(0)|}{c_\psi}
\le
\frac{C_0}{c_\psi}.
\]
Hence, by \eqref{eq-MF-inverse-psi-Lipschitz},
\begin{equation*}\label{eq-MF-mirror-bounded}
\|\beta\|_{\BMO}
\vee
\|\widetilde\beta\|_{\BMO}
\le
\frac{C_0\sqrt T+R}{c_\psi}:=M_R.
\end{equation*}
Moreover,
\begin{equation}
\label{eq-MF-beta-q-Lipschitz}
\|\Delta\beta\|_{\BMO}
\le
\frac1{c_\psi}
{\|q-\widetilde q\|_{\BMO}\le \frac{2R}{c_\psi}.}
\end{equation}
By Lemma~\ref{lem-BMO-lip}, applied on the BMO ball of radius
\(
M_R,
\)
we obtain
\begin{equation*}
\label{eq-MF-Q-beta-Lipschitz}
\|\Delta\dbQ\|_{\TV}
\le
L_{M_R}
\|\Delta\beta\|_{\BMO}.
\end{equation*}
Consequently, Assumption~\ref{ass-pe} yields
\begin{equation}
\label{eq-MF-terminal-Lipschitz}
\left\|
\delta F(\dbQ,X)
-
\delta F(\widetilde{\dbQ},X)
\right\|_\infty
\le
L_F
L_{M_R}
\|\Delta\beta\|_{\BMO}.
\end{equation}

\revsubstep{1c}{Stability of the evaluation BSDE}
Writing both BSDEs under $\dbQ$, we obtain
\begin{equation}\label{eq-MF-difference-BSDE}
\Delta Y_t = \Delta\xi + \int_t^T \Bigl( \ell(r,X,\beta_r) - \ell(r,X,\widetilde\beta_r) + \Delta\beta_r\cdot Z_r(\widetilde q) \Bigr)\dd r - \int_t^T \Delta Z_r\,\dd W_r^\beta.
\end{equation}

{For clarity, all constants in the following calculation may depend
on the fixed data as well as $R$. On the ball in
\eqref{eq-q-BMO-R}, BMO invariance gives common bounds under $\dbQ$
for $\beta$, $\widetilde\beta$ and $Z(\widetilde q)$, the last by
Lemma~\ref{lem-fixed-policy-evaluation}. The upper Hessian bound implies
\[
 |\ell(\beta)-\ell(\widetilde\beta)|
 \le [C_0+C_\ell(|\beta|+|\widetilde\beta|)]|\Delta\beta|.
\]
Conditional Cauchy--Schwarz therefore bounds the conditional integral
of the absolute difference driver in \eqref{eq-MF-difference-BSDE}
by $A_R\|\Delta\beta\|_{\BMO(\dbQ)}$. Increase $A_R$ to also bound
$\|\Delta\xi\|_\infty$ by the same expression, using
\eqref{eq-MF-terminal-Lipschitz} and the inverse BMO comparison.
Conditional expectation then gives
$\|\Delta Y\|_\infty\le2A_R\|\Delta\beta\|_{\BMO(\dbQ)}$.
The square identity gives
\[
 \|\Delta Z\|_{\BMO(\dbQ)}^2
 \le \|\Delta\xi\|_\infty^2
      +2\|\Delta Y\|_\infty A_R\|\Delta\beta\|_{\BMO(\dbQ)}
 \le5A_R^2\|\Delta\beta\|_{\BMO(\dbQ)}^2.
\]
Thus, with finite constants of the indicated fixed-data dependence,}
\[
\|\Delta Y\|_{\mathcal S^\infty}
\le
C_Y(M_R,
c_\ell,C_\ell,L_F)
\|\Delta\beta\|_{\BMO(\dbQ)},
\]
Applying It\^o's formula to $|\Delta Y|^2$ in
\eqref{eq-MF-difference-BSDE} gives
\[
\|\Delta Z\|_{\BMO(\dbQ)}
\le
C_Z(M_R,
c_\ell,C_\ell,L_F)
\|\Delta\beta\|_{\BMO(\dbQ)},
\]
with the same dependence of the constant. 
Fix
\[
1<p<
\Phi^{-1}\left(M_R
\right).
\]
By Lemma~\ref{lem-BMO-reverse-holder}.2, we obtain that
\[
\|\beta\|_{\BMO(\dbQ)}
\le
\sqrt{\kappa^p(
\|\beta\|_{\BMO})}
\le
\sqrt{\kappa^p\left(M_R\right)}
\]
Then by Lemma~\ref{lem-BMO-reverse-holder}.3, we obtain that for all
\[
\|\Delta\beta\|_{\BMO(\dbQ)}
\le
{\left(
1+
\frac{\sqrt2}{2}
\sqrt{
\kappa^p\left(
M_R
\right)
}
\right)}
\|\Delta\beta\|_{\BMO},
\]
and, by applying the same Lemma with $\Delta Z$,
\[
\|\Delta Z\|_{\BMO}
\le
{\left(
1+
\frac{\sqrt2}{2}
M_R
\right)}
\|\Delta Z\|_{\BMO(\dbQ)}.
\]
The same estimates hold with $\dbQ$ replaced by
$\widetilde{\dbQ}$.
Finally we obtain
\[
\|\Delta Z\|_{\BMO}
\le
C'_Z(M_R,
c_\ell,C_\ell,L_F)
\|\Delta\beta\|_{\BMO}.
\]
Together with \eqref{eq-MF-beta-q-Lipschitz}, this yields
\begin{equation}
\label{eq-L-Z-R}
\|Z(q)-Z(\widetilde q)\|_{\BMO}
\le
L^Z_R
\|q-\widetilde q\|_{\BMO},
\end{equation}
where
\[
L^Z_R
=
L^Z_R\left(
R,T,C_0,
c_\ell,C_\ell,c_\psi,
C_F,L_F
\right)
<\infty.
\]

\revkeepstep
\noindent
  \revstep{2}{Local existence}
We can now define the limiting vector field. For \(q\in\BMO\), set
\begin{equation}
\label{eq-MF-Psi-definition}
\Psi(q)
:=
\nabla\ell\bigl(\beta(q)\bigr)+Z(q),
\qquad
\beta(q):=(\nabla\psi)^{-1}(q),
\end{equation}
where, more explicitly,
$$
[\Psi(q)]_t
=
\nabla\ell
\left(
t,X,
[\nabla\psi(t,X,\cdot)]^{-1}(q_t)
\right)
+
Z_t(q),
$$
where \(\Psi\) is locally Lipschitz on \(\BMO\).
Set \(q^0:=\nabla\psi(\beta^0).\) Since BMO process space modulo \(dt\otimes dP\)-a.e. equality is complete. By the Picard--Lindel\"of theorem in the Banach space $\BMO$, there
exists a unique maximal solution of
\begin{equation}
\label{eq-MF-continuous-q-flow}
\frac{\dd q^s}{\dd s}
=
-\Psi(q^s),
\qquad
q^{s=0}=q^0.
\end{equation}
Define
\[
\beta^s:=(\nabla\psi)^{-1}(q^s),
\qquad
\frac{\dd\dbQ^s}{\dd\dbP}
=
D(\beta^s).
\]
Then \eqref{eq-MF-continuous-q-flow} is equivalent to
\[
\nabla\psi(\beta^S)
-
\nabla\psi(\beta^0)
=
-
\int_0^S
\left(
\nabla\ell(\beta^s)+Z^s
\right)\dd s.
\]

We next show that the maximal solution is global. Let
$[0,S_*)$ be its maximal interval of existence and suppose, by
contradiction, that $S_*<\infty$. Fix $S<S_*$, then we have
\[
B^q_S:=\sup_{0\le s\le S}\|q^s\|_\BMO<+\infty
\]

\revkeepstep
\noindent
\textit{Step 3: consistency of the discrete BSDE.}

{\revsubstep{3a}{The penalization remainder}
Take $h\le\min(1,h_{S+1})$. By the bootstrap transferred back to
$\dbP$, fix a number $B_S\ge1$ with
\[
 \sup_{0<h\le\min(1,h_{S+1}),\,nh\le S+1}
       \|\beta^{h,n}\|_{\BMO}\le B_S.
\]
Use the increment constant from horizon $S+1$, denoted again by
$L_S$, and set $p_{B_S}=(1+\Phi^{-1}(B_S))/2$.
}
Set \(q^{h,n} := \nabla\psi(\beta^{h,n}) \), \( s_n:=nh\). By \eqref{eq-MF-mirror-discrete-FOC},
\begin{equation}
\label{eq-MF-discrete-q}
q^{h,n+1}-q^{h,n}
=
-h
\left(
\nabla\ell(\beta^{h,n+1})
+
Z^{h,n+1}
\right).
\end{equation}

Let \(\widehat Z^{h,n+1} := Z(q^{h,n+1}),\qquad \widehat Y^{h,n+1} := Y(q^{h,n+1})\) denote the corresponding solution of
\eqref{eq-MF-BSDE-q}. Both
$(Y^{h,n+1},Z^{h,n+1})$ and
$(\widehat Y^{h,n+1},\widehat Z^{h,n+1})$
have terminal value
\[
\delta F(\dbQ^{h,n+1},X).
\]
Under $\dbQ^{h,n+1}$, their only difference is the additional
Bregman term
\[
r_t^{h,n+1}
:=
\frac1h
\left(
\psi(\beta_t^{h,n+1})
-
\psi(\beta_t^{h,n})
-
\nabla\psi(\beta_t^{h,n})
\cdot
(\beta_t^{h,n+1}-\beta_t^{h,n})
\right).
\]
By the quadratic upper bound on $\psi$,
\begin{equation*}
\label{eq-MF-r-bound}
0
\le
r_t^{h,n+1}
\le
\frac{C_\psi}{2h}
|\beta_t^{h,n+1}-\beta_t^{h,n}|^2.
\end{equation*}
By Proposition~\ref{prop-MF-BMO-continue}, Lemma~\ref{lem-BMO-reverse-holder}.2 and Lemma~\ref{lem-BMO-reverse-holder}.3, 
\[
\sup_{\tau\in\cT}
\left\|
\dbE^{\dbQ^{h,n+1}}
\left[
\left(
\int_\tau^T r_t^{h,n+1}\,\dd t
\right)^2
\Bigm|\cF_\tau
\right]
\right\|_\infty^{1/2}
\le
2C_
\psi\left(1+\frac{\sqrt{2}}{2}\sqrt{\kappa^{p_{B_S}}(B_S)}\right)^{2}L_S^2h.
\]
\revsubstep{3b}{Conditional isometry and consistency}
Since, under $\dbQ^{h,n+1}$,
\[
\Delta Y_t
:=
Y_t^{h,n+1}-\widehat Y_t^{h,n+1}
=
\dbE^{\dbQ^{h,n+1}}
\left[
\int_t^T r_s^{h,n+1}\,\dd s
\Bigm|\cF_t
\right],
\]
Hence, by the conditional It\^o isometry,
\begin{equation*}
\dbE^{\dbQ^{h,n+1}}
\left[
\int_\tau^T
\left|
Z_t^{h,n+1}-\widehat Z_t^{h,n+1}
\right|^2
\,\dd t
\Bigm|\cF_\tau
\right]
\le
\dbE^{\dbQ^{h,n+1}}
\left[
\left(\int_\tau^T r_t^{h,n+1}\,\dd t
\right)^2
\Bigm|\cF_\tau
\right].
\end{equation*}
Therefore,
\[
\left\|
Z^{h,n+1}-\widehat Z^{h,n+1}
\right\|_{\BMO(\dbQ^{h,n+1})}
\le
2C_
\psi\left(1+\frac{\sqrt{2}}{2}\sqrt{\kappa^{p_{B_S}}(B_S)}\right)^{2}L_S^2h.
\]
Again by Lemma~\ref{lem-BMO-reverse-holder}.3,
\begin{equation}
\label{eq-MF-Z-consistency}
\left\|
Z^{h,n+1}-\widehat Z^{h,n+1}
\right\|_{\BMO}
\le
2C_
\psi
\left(1+\frac{\sqrt{2}}{2}\sqrt{\kappa^{p_{B_S}}(B_S)}\right)^{2}
\left(1+\frac{\sqrt{2}}{2}{B_S}\right)
L_S^2h
:=
L^1_{Z,S}h.
\end{equation}

\revkeepstep
\noindent
\textit{Step 4: error estimate at the grid points.}

{\revsubstep{4a}{Consistency at the grid}
Choose a single mirror-coordinate radius
\[
 R:=1+B_S^q+C_0\sqrt T+C_\psi B_S.
\]
It contains both $q^s$, $0\le s\le S$, and all the discrete
$q^{h,n}$ with $nh\le S+1$. All vector-field and evaluation
constants in this step are taken on this same ball.
}
Using \eqref{eq-MF-Psi-definition},
\eqref{eq-MF-discrete-q} can therefore be written as
\begin{equation}
\label{eq-MF-discrete-consistent-form}
q^{h,n+1}-q^{h,n}
=
-h\Psi(q^{h,n+1})
-h\left(
Z^{h,n+1}
-
\widehat Z^{h,n+1}
\right).
\end{equation}
Integrating \eqref{eq-MF-continuous-q-flow} over
$[s_n,s_{n+1}]$ yields
\begin{equation}
\label{eq-MF-continuous-implicit-Euler}
q^{s_{n+1}}-q^{s_n}
=
-h\Psi(q^{s_{n+1}})
+
\rho^{n+1},
\end{equation}
where
\[
\rho^{n+1}
:=
\int_{s_n}^{s_{n+1}}
\left(
\Psi(q^{s_{n+1}})
-
\Psi(q^r)
\right)\dd r.
\]
By Assumption~\ref{ass-rc}, Assumption~\ref{ass-rc-quadratic} and \eqref{eq-B-Z-R},
\[
B^{\Psi}_R:=\sup_{\|q\|_{\BMO}\le R}
\|\Psi(q)\|_{\BMO}
\le
C_0\sqrt T
+
\frac{C_\ell}{c_\psi}
\left(
C_0\sqrt T+R
\right)
+
{B^Z_R}
<\infty.
\]
Since
\[
\frac{\dd q^s}{\dd s}
=
-\Psi(q^s),
\]
it follows that, for every $0\le r\le s\le S$,
\[
\|q^s-q^r\|_{\BMO}
\le
B^{\Psi}_R
|s-r|.
\]
By Assumption~\ref{ass-rc}, Assumption~\ref{ass-rc-quadratic} and\eqref{eq-L-Z-R}, we then obtain
\begin{equation}
L_{\Psi,R}:={\operatorname{Lip}(\Psi;\{q:\|q\|_{\BMO}\le R\})}\le
\frac{C_\ell}{c_\psi}+L^Z_R.
\label{eq-Lip-Psi}
\end{equation}
Hence
\begin{equation}\label{eq-rho-n+1}
\|\rho^{n+1}\|_{\BMO} \le L_{\Psi,R} \int_{s_n}^{s_{n+1}} \|q^{s_{n+1}}-q^r\|_{\BMO} \,\dd r \le \frac{1}{2}L_{\Psi,R}B^\Psi_R h^2.
\end{equation}

\revsubstep{4b}{Discrete Gronwall estimate}
Set \(e_n:=q^{s_n}-q^{h,n}.\) Subtracting \eqref{eq-MF-discrete-consistent-form} from
\eqref{eq-MF-continuous-implicit-Euler}, using \eqref{eq-Lip-Psi}, \eqref{eq-MF-Z-consistency}, and
\eqref{eq-rho-n+1}, yields
\begin{equation*}
\|e_{n+1}\|_{\BMO} \le \|e_n\|_{\BMO} + hL_{\Psi,R}\|e_{n+1}\|_{\BMO} + \left[ L^1_{Z,S} + \frac{1}{2}L_{\Psi,R}B^\Psi_R \right]h^2.
\end{equation*}
Assume that $h$ small enough such that
\[
L_{\Psi,R}h\le\frac12.
\]
Since $e_0=0$, iteration gives
\[
\|e_n\|_{\BMO}
\le
\left(
L^1_{Z,S}
+
\frac{1}{2}L_{\Psi,R}B^\Psi_R
\right)
h^2
\sum_{k=1}^n
\left(1-L_{\Psi,R}h\right)^{-k}.
\]
For $nh\le S$,
\[
(1-L_{\Psi,R}h)^{-n}
\le
e^{2L_{\Psi,R}S},
\]
and consequently
\begin{equation*}
\label{eq-MF-q-global-error}
\max_{0\le nh\le S}
\|q^{s_n}-q^{h,n}\|_{\BMO}
\le
S {e^{2L_{\Psi,R}S}}
\left(
L^1_{Z,S}
+
\frac{1}{2}L_{\Psi,R}B^\Psi_R
\right)h.
\end{equation*}
By \eqref{eq-MF-inverse-psi-Lipschitz},
\begin{equation*}
\label{eq-MF-beta-grid-error}
\max_{0\le nh\le S}
\|\beta^{s_n}-\beta^{h,n}\|_{\BMO}
\le
\frac{1}{c_\psi}S {e^{2L_{\Psi,R}S}}
\left(
L^1_{Z,S}
+
\frac{1}{2}L_{\Psi,R}B^\Psi_R
\right)h:=L^1_{S}h.
\end{equation*}

\revkeepstep
\noindent
\textit{Step 5: interpolation and convergence of the measures.}
The local Lipschitz continuity of
$s\mapsto q^s$ and
\eqref{eq-MF-inverse-psi-Lipschitz} imply that
$s\mapsto\beta^s$ is locally Lipschitz in $\BMO$, namely
\[
 \|\beta^s-\beta^r\|_{\BMO}
 \le\frac{B_R^\Psi}{c_\psi}|s-r|,
 \qquad 0\le r,s\le S.
\]
Set $L_S^2=B_R^\Psi/c_\psi$; the bound is $L_S^2h$ when
$|s-r|\le h$.
Hence, for
$n(s):=\lfloor s/h\rfloor$,
\[
\|\beta^s-\beta^{h,s}\|_{\BMO}
\le
\|\beta^s-\beta^{s_{n(s)}}\|_{\BMO}
+
\|\beta^{s_{n(s)}}-\beta^{h,n(s)}\|_{\BMO}
\le
(L^1_S+L^2_S)h,
\]
uniformly for $s\in[0,S]$. {Enlarge the common drift ball, if necessary, to a radius
$\widetilde B_S$ containing the continuous policies as well.
Set $L'_S=(1+L_{\widetilde B_S})(L_S^1+L_S^2)$, where
$L_{\widetilde B_S}$ is the density-Lipschitz constant in
Lemma~\ref{lem-BMO-lip}. Then}
\begin{equation*}
\label{eq-MF-final-BMO-error}
\sup_{s\in[0,S]}
\|\beta^s-\beta^{h,s}\|_{\BMO}
\le
L'_Sh.
\end{equation*}
The BMO--Girsanov stability estimate then yields
\[
\sup_{s\in[0,S]}
\|\dbQ^s-\dbQ^{h,s}\|_{\TV}
\le
L'_Sh.
\]
\revstep{6}{Global continuation}
Proposition~\ref{prop-MF-bounded-psi}, applied with
the horizon $S_*$, together with the BMO--Girsanov equivalence, yields
a constant $L'_{S_*}$, independent of $S<S_*$, such that
\[
\sup_{\substack{0<h\le h_{S_*}\\nh\le S_*}}
\|\beta^{h,n}\|_{\BMO}
{\le
L'_{S_*}.}
\]
Thus, for every fixed $S<S_*$, letting $h\to0$ gives
\[
\sup_{s\in[0,S]}\|\beta^s\|_{\BMO}
\le
L'_{S_*}.
\]
Since the right-hand side is independent of $S$, we conclude that
\[
\sup_{s<S_*}\|\beta^s\|_{\BMO}<\infty.
\]
Using
\[
|\nabla\psi(t,X,b)|
\le
|\nabla\psi(t,X,0)|+C_\psi|b|,
\]
we further obtain
\[
R_q:=\sup_{s<S_*}\|q^s\|_{\BMO}<\infty.
\]

By the Lipschitz estimate for $\Psi$ on bounded subsets of $\BMO$,
there exists ${L_{R_q+1}}<\infty$ such that
\[
\|\Psi(q)-\Psi(\widetilde q)\|_{\BMO}
\le
{L_{R_q+1}}\|q-\widetilde q\|_{\BMO}
\]
whenever
{$\|q\|_{\BMO},\|\widetilde q\|_{\BMO}\le R_q+1$.}
Consequently,
\[
\sup_{s<S_*}\|\Psi(q^s)\|_{\BMO}
\le
\|\Psi(0)\|_{\BMO}+{L_{R_q+1}}R_q
<\infty.
\]
Since
\[
\frac{\dd q^s}{\dd s}=-\Psi(q^s),
\]
it follows that $s\mapsto q^s$ is uniformly Lipschitz on $[0,S_*)$.
Hence $(q^s)_{s<S_*}$ is Cauchy as $s\uparrow S_*$ and, by the
completeness of $\BMO$, there exists $q_*\in\BMO$ such that
\[
q^s\longrightarrow q_*
\qquad\text{in }\BMO
\quad\text{as }s\uparrow S_*.
\]
Applying the Picard--Lindel\"of theorem (see, e.g., \cite{Deimling1977}) in $\BMO$ with initial
condition $q_*$ at time $S_*$ and using uniqueness extends the
solution beyond $S_*$, contradicting maximality. Therefore
$S_*=\infty$, and the solution exists uniquely for every $s\ge0$.

\end{proof}

\Needspace{6\baselineskip}
\subsection{Common convergence consequences}
\label{sec-common-flow-decay-proof}\label{subsec-linear-exponential-convergence}
\readingguide{The proximal scheme already satisfies discrete Bregman contraction
and objective estimates. To pass these inequalities to continuous
learning time, we must show that the corresponding entropy, Bregman and
objective terms converge along the discrete approximations. We prove
this separately for the Newton flow and the general flow, because the
Newton argument avoids the upper-Hessian assumption used in the general
BMO-continuity estimate.}
\label{sec-common-limits}
\noindent The general-flow and particle limit passages use the following continuity estimate. It is placed here because it is needed precisely when the discrete estimates are passed to continuous learning time; the Newton limit is treated separately below and does not use the upper-Hessian bounds.

\begin{lem}[Continuity on a BMO ball]\label{lem-BMO-energy-continuity}
Let $(\dbQ^k,\beta^k)$ and $(\dbQ,\beta)$ be admissible pairs such that
$\|\beta^k-\beta\|_{\BMO}\to0$. Retain the measurability and differentiability conditions and the bounds
at zero in Assumption~\ref{ass-rc}.
\begin{enumerate}[label=\textnormal{(\roman*)}]
\item If $0\le\nabla^2\psi\le C_\psi I_d$ and
$(\dbQ',\beta')$ is a fixed admissible pair with $\beta'\in\BMO$, then
\[
 H_\psi(\dbQ'\mid\dbQ^k)\longrightarrow H_\psi(\dbQ'\mid\dbQ).
\]
\item If also $0\le\nabla^2\ell\le C_\ell I_d$ and $\delta F$ is bounded,
then $J(\dbQ^k)\to J(\dbQ)$.
\end{enumerate}
The same statements hold on every fixed product canonical space.
\end{lem}
\begin{proof}
\phantomsection\label{proof-common-continuity}
\revstep{1}{The fixed reference law}
Since $\beta'\in\BMO$, the BMO--Girsanov equivalence in
Lemma~\ref{lem-BMO-reverse-holder} gives
\[
 \|\beta^k-\beta\|_{\BMO(\dbQ')}\longrightarrow0.
\]
In particular,
\begin{equation}\label{eq-common-reference-L2}
 \dbE^{\dbQ'}\!\int_0^T|\beta_t^k-\beta_t|^2\,\dd t\longrightarrow0.
\end{equation}
\revstep{2}{Continuity in the second Bregman argument}
For fixed $(t,\omega)$, define
\[
d_\psi(a,b)
:=
\psi(t,\omega,a)
-
\psi(t,\omega,b)
-
\nabla\psi(t,\omega,b)\cdot(a-b).
\]
By the Hessian bound in part (i),
\[
0
\le
d_\psi(a,b)
\le
\frac{C_\psi}{2}|a-b|^2,
\]
and
\[
|\nabla\psi(t,\omega,b)-\nabla\psi(t,\omega,c)|
\le
C_\psi|b-c|.
\]
Moreover,
\[
d_\psi(a,b)-d_\psi(a,c)
=
d_\psi(c,b)
+
\bigl(
\nabla\psi(t,\omega,c)
-
\nabla\psi(t,\omega,b)
\bigr)
\cdot(a-c).
\]
Consequently,
\begin{equation}
\label{eq-MF-bregman-second-variable}
|d_\psi(a,b)-d_\psi(a,c)|
\le
\frac{C_\psi}{2}|b-c|^2
+
C_\psi|b-c|\,|a-c|.
\end{equation}
Applying \eqref{eq-MF-bregman-second-variable} with \(a=\beta_t'\), \( b=\beta_t^k \), \( c=\beta_t\), and integrating with respect to
$\dd t\otimes\dd\dbQ'$, we obtain
\begin{align*}
|
H_\psi(\dbQ'\,|\,\dbQ^k)
&- 
H_\psi(\dbQ'\,|\,\dbQ)
|
\le
\frac{C_\psi}{2}
\dbE^{\dbQ'}
\left[
\int_0^T
|\beta_t^k-\beta_t|^2\,\dd t
\right]
\\
&+
C_\psi
\left(
\dbE^{\dbQ'}
\left[
\int_0^T
|\beta_t^k-\beta_t|^2\,\dd t
\right]
\right)^{1/2}
\times
\left(
\dbE^{\dbQ'}
\left[
\int_0^T
|\beta_t'-\beta_t|^2\,\dd t
\right]
\right)^{1/2}.
\end{align*}
The last factor is finite, whereas the first two factors converge to
zero by \eqref{eq-common-reference-L2}. Therefore,
\begin{equation*}
\label{eq-common-relative-H-limit}
H_\psi(\dbQ'\,|\,\dbQ^k)
\longrightarrow
H_\psi(\dbQ'\,|\,\dbQ).
\end{equation*}

\revstep{3}{The objective and the varying law}
All these drifts lie in a common BMO ball. Reverse H\"older gives a
common density bound in $L^{p_0}(\dbP)$ for some $p_0>1$, and
Lemma~\ref{lem-BMO-lip} gives convergence in $L^1(\dbP)$. For
$1<p<p_0$, interpolation yields
\[
 \|D(\beta^k)-D(\beta)\|_p
 \le\|D(\beta^k)-D(\beta)\|_1^\vartheta
      \|D(\beta^k)-D(\beta)\|_{p_0}^{1-\vartheta}
 \longrightarrow0,
 \qquad \frac1p=\vartheta+\frac{1-\vartheta}{p_0}.
\]
The quadratic upper bound on $\ell$ gives
\[
|\ell(t,X,b)-\ell(t,X,b')|
\le
\left(
|\nabla\ell(t,X,0)|+C_\ell(|b|+|b'|)
\right)|b-b'|.
\]
By Cauchy--Schwarz and the BMO energy inequality, for every finite $q>1$,
\[
 \left\|\int_0^T|\ell(\beta_t^k)-\ell(\beta_t)|\,\dd t\right\|_{L^q(\dbP)}
 \le C_q\bigl(\sqrt T+\|\beta^k\|_{\BMO}+\|\beta\|_{\BMO}\bigr)
                   \|\beta^k-\beta\|_{\BMO}\longrightarrow0.
\]
Choose $q=p/(p-1)$ in this estimate. Splitting the energy difference as
\begin{equation*}
H_\ell(\dbQ^k)-H_\ell(\dbQ) =\dbE\!\left[D(\beta^k)\int_0^T (\ell(\beta_t^k)-\ell(\beta_t))\,\dd t\right] +\dbE\!\left[(D(\beta^k)-D(\beta)) \int_0^T\ell(\beta_t)\,\dd t\right]
\end{equation*}
and applying H\"older proves convergence of $H_\ell$. Finally,
\[
 |F(\dbQ^k)-F(\dbQ)|\le C_F\|\dbQ^k-\dbQ\|_{\TV}\longrightarrow0.
\]
All estimates are unchanged in form in dimension $Nd$; the BMO radii
and resulting constants may depend on $N$.
\end{proof}

\begin{proof}[Proof of Corollary~\ref{cor-common-flow-decay}]
\phantomsection\label{proof-common-flow-decay}
\revstep{1}{The Newton energy limit}
\phantomsection\label{proof-newton-decay}
The density interpolation in the proof of
Theorem~\ref{thm-newton-structure} gives strong convergence in some
$L^p(\dbP)$, while \eqref{eq-main-dual-affine} gives convergence of the
centred entropy coordinate in
every finite $L^q(\dbP)$. Decrease $p>1$ if necessary so that the density
of $\dbQ^*$ has the same reverse H\"older bound, and set $q=p/(p-1)$.
No upper Hessian bound is used in the following argument.

\revsubstep{1a}{Lower and upper energy bounds}
Fix $s\ge0$. Let $(h_k)_{k\ge1}$ satisfy \(h_k\to0^+.\) For brevity, write
\[
\dbQ_k:=\dbQ^{h_k,s},
\qquad
D_k:=D(\beta^{h_k,s}),
\qquad
D:=D(\beta^s),
\qquad
\rho_k:=\overline\rho^{h_k,s}.
\]
We have proved that
\[
D_k\rightharpoonup D
\qquad
\text{weakly in }L^p(\dbP).
\]
Moreover, by \eqref{eq-uni-8},
\[
\rho_k\longrightarrow\overline\rho^s
\qquad
\text{strongly in }L^q(\dbP).
\]
Lemma~\ref{lem-H-weak-lsc} yields
\begin{equation}
\label{eq-uni-12}
H_\ell(\dbQ^s)
\le
\liminf_{k\to\infty}H_\ell(\dbQ_k).
\end{equation}
On the other hand, since
$\rho_k\in\partial H_\ell(\dbQ_k)$,
\[
H_\ell(\dbQ^s)-H_\ell(\dbQ_k)
\ge
\left\langle
\rho_k,\dbQ^s-\dbQ_k
\right\rangle.
\]
\revsubstep{1b}{The vanishing pairing}
We claim that
\[
\left\langle
\rho_k,\dbQ^s-\dbQ_k
\right\rangle
\longrightarrow0.
\]
Indeed,
\begin{equation*}
\left\langle \rho_k,\dbQ^s-\dbQ_k \right\rangle = \dbE\left[\rho_k(D-D_k)\right] = \dbE\left[ (\rho_k-\overline\rho^s)(D-D_k) \right] + \dbE\left[ \overline\rho^s(D-D_k) \right].
\end{equation*}
The first term converges to zero by H\"older's inequality, the
strong $L^q(\dbP)$ convergence of $\rho_k$, and the uniform
$L^p(\dbP)$ bound on $(D_k)_{k\ge1}$. The second term converges
to zero by the weak $L^p(\dbP)$ convergence of $D_k$.
It follows that
\[
\limsup_{k\to\infty}H_\ell(\dbQ_k)
\le
H_\ell(\dbQ^s).
\]
Together with \eqref{eq-uni-12}, this proves
\begin{equation}
\label{eq-uni-13}
H_\ell(\dbQ^{h_k,s})
\longrightarrow
H_\ell(\dbQ^s).
\end{equation}

\revsubstep{1c}{Convergence of the potential}
Finally, since $F$ is linear,
\[
F(\dbQ_k)-F(\dbQ^s)
=
\left\langle
f,\dbQ_k-\dbQ^s
\right\rangle
=
\dbE\left[
f(D_k-D)
\right].
\]
Since $f$ is bounded, it belongs to $L^q(\dbP)$.
The weak $L^p(\dbP)$ convergence of $D_k$ therefore gives
\[
F(\dbQ_k)\longrightarrow F(\dbQ^s).
\]
Combining this convergence with
\eqref{eq-uni-13}, we conclude that
\begin{equation*}
J(\dbQ^{h_k,s})
\longrightarrow
J(\dbQ^s).
\end{equation*}
Since $(h_k)$ was arbitrary, the energy and objective limits hold as
$h\to0^+$ for every fixed $s$.

\revstep{2}{The Newton Bregman pairing}
We next consider the relative entropy with fixed first argument.
Using the Bregman identity,
\begin{align*}
H_\ell(\dbQ^*\mid\dbQ^{h,s})
&=
H_\ell(\dbQ^*)
-
H_\ell(\dbQ^{h,s})
-
\left\langle
\overline\rho^{h,s},
\dbQ^*-\dbQ^{h,s}
\right\rangle.
\end{align*}
It remains to pass to the limit in the last pairing. We have
\begin{equation*}
\left\langle\overline\rho^{h,s},\dbQ^*-\dbQ^{h,s}\right\rangle-\left\langle\overline\rho^s,\dbQ^*-\dbQ^s\right\rangle
=\dbE\left[(\overline\rho^{h,s}-\overline\rho^s)(D(\beta^*)-D(\beta^{h,s}))\right]+\dbE\left[\overline\rho^s(D(\beta^s)-D(\beta^{h,s}))\right].
\end{equation*}
The second term
converges to zero by the weak $L^p(\dbP)$ convergence
$D(\beta^{h,s})\rightharpoonup D$, since $\overline\rho^s\in L^q(\dbP)$. And by the strong $L^q(\dbP)$ convergence of $\overline\rho^{h,s}$ and
the boundedness of $(D(\beta^{h,s}))_h$ in $L^p(\dbP)$. 
\begin{align*}
\left|
\dbE\left[
(\overline\rho^{h,s}-\overline\rho^s)(D(\beta^*)-D(\beta^{h,s}))
\right]
\right|
&\le
\|\overline\rho^{h,s}-\overline\rho^s\|_{L^q(\dbP)}
\left(
\|D(\beta^*)\|_{L^p(\dbP)}
+
\|D(\beta^{h,s})\|_{L^p(\dbP)}
\right)
\longrightarrow0,
\end{align*}
The first term converges to zero. Therefore,
\[
\left\langle
\overline\rho^{h,s},
\dbQ^*-\dbQ^{h,s}
\right\rangle
\longrightarrow
\left\langle
\overline\rho^s,
\dbQ^*-\dbQ^s
\right\rangle.
\]
Together with
\eqref{eq-uni-13}, this yields
\begin{equation*}
\label{eq-exp-3}
H_\ell(\dbQ^*\mid\dbQ^{h,s})
\longrightarrow
H_\ell(\dbQ^*\mid\dbQ^s).
\end{equation*}

\revstep{3}{The general-flow limits}
\phantomsection\label{proof-natural-decay}
On every bounded learning interval, the preceding construction bounds
both discrete and continuous drifts in a common BMO ball and gives
$\|\beta^{h,s}-\beta^s\|_{\BMO}\to0$. The optimal drift $\beta^*$ is
also BMO by Proposition~\ref{prop-bregman-proximal}. Applying both parts
of Lemma~\ref{lem-BMO-energy-continuity}, with the fixed reference
$\dbQ'=\dbQ^*$ in part (i), gives
\begin{equation*}\label{eq-MF-relative-H-limit}\label{eq-MF-energy-limit}
 H_\psi(\dbQ^*\mid\dbQ^{h,s})\to H_\psi(\dbQ^*\mid\dbQ^s),
 \qquad J(\dbQ^{h,s})\to J(\dbQ^s).
\end{equation*}

\revstep{4}{Passage of the discrete decay estimates to the limit}
Fix $s\ge0$ and set \(M_h:=\lfloor s/h\rfloor\), \( \mathcal E_0:=H_\psi(\dbQ^*\mid\dbQ^0)\). The discrete estimates of Proposition~\ref{prop-bregman-proximal} give
\[
 H_\psi(\dbQ^*\mid\dbQ^{h,s})
 \le
 (1+\lambda_{\ell|\psi}h)^{-M_h}\mathcal E_0,
\]
and, for $s>0$ and $h$ sufficiently small so that $M_h\ge1$,
\[
 0\le J(\dbQ^{h,s})-J(\dbQ^*)
 \le
 \frac{\lambda_{\ell|\psi}\mathcal E_0}
 {(1+\lambda_{\ell|\psi}h)^{M_h}-1}.
\]
For the Newton flow we take $\lambda_{\ell|\psi}=1$. Since
\[
 (1+\lambda_{\ell|\psi}h)^{M_h}
 \longrightarrow
 e^{\lambda_{\ell|\psi}s},
\]
the limit passages established in Steps~1--2 for the Newton flow and
Step~3 for the general flow {give the fixed-first-argument bound in \eqref{eq-common-flow-bregman-decay} and the objective bound in \eqref{eq-common-flow-value-decay}.} Likewise, for
$0\le s_1\le s_2$, the discrete monotonicity
\[
 J(\dbQ^{h,s_2})\le J(\dbQ^{h,s_1})
\]
passes to the limit and proves that $s\mapsto J(\dbQ^s)$ is non-increasing.

\revlead{The forward divergence}
For $s>0$, the limiting law belongs to $\cA_\ell\cap\cA_\psi$.
The stationarity identity at the optimizer and Lemma~\ref{lem-forward-coercivity} therefore imply
\[
 \lambda_{\ell|\psi}H_\psi(\dbQ^s\mid\dbQ^*)
 \le J(\dbQ^s)-J(\dbQ^*)
 \le\frac{\lambda_{\ell|\psi}\mathcal E_0}
              {e^{\lambda_{\ell|\psi}s}-1}.
\]
Division by $\lambda_{\ell|\psi}>0$ proves the forward-divergence bound in
\eqref{eq-common-flow-value-decay}. This argument uses the objective
limit already established and requires neither convergence of
$H_\psi(\dbQ^{h,s}\mid\dbQ^*)$ nor an upper Hessian bound for the Newton flow.
For $s\ge\log2/\lambda_{\ell|\psi}$, use
$e^{\lambda_{\ell|\psi}s}-1\ge\tfrac12 e^{\lambda_{\ell|\psi}s}$.
Finally, Pinsker's inequality and \eqref{eq-relative-ell-psi} give
\[
 \|\dbQ^s-\dbQ^*\|_{\TV}^2
 \le\frac{2}{c_\psi}H_\psi(\dbQ^*\mid\dbQ^s),
\]
which proves the total-variation consequence in the unnumbered remark
following Corollary~\ref{cor-common-flow-decay}.
\end{proof}

\section{Particle approximation and Markov case} \label{sec-particle-proofs}

We first transfer the flow construction to each fixed particle number, then isolate the empirical-potential defect and propagate the discrete contraction estimate.

\subsection{Proofs for the finite-particle system}\label{sec-finite-particle-proofs}
\label{subsec-particle-proofs}

\subsubsection{Well-posedness and continuous-time limit}
At fixed particle number the control problem is still a weak-control problem, now in dimension $Nd$. The argument records which constants may depend on $N$ rather than claiming uniformity.
\begin{proof}[Proof of Theorem~\ref{thm-particle-flow}]
\phantomsection\label{proof-particle-flow}

\revstep{1}{Tensorized running costs}
Fix $N\ge1$ and regard the particle problem as a stochastic control
problem on the $Nd$-dimensional canonical space
$(\Omega^N,\mathbb F^N,\dbP^N)$.

By
\eqref{eq-particle-ell-Hessian}--\eqref{eq-particle-relative-convexity},
the running cost $\ell^N$ and Bregman potential $\psi^N$ satisfy
the same Hessian lower and upper bounds and the same
relative-convexity constant as in
Assumptions~\ref{ass-rc} and~\ref{ass-rc-quadratic}.
The additive gradient bounds, however, become
\[
 |\nabla\ell^N(\mathbf b)|
 \le L_\ell(\sqrt N+|\nabla\psi^N(\mathbf b)|),\qquad
 |\nabla\psi^N(\mathbf b)|
 \le L_\psi(\sqrt N+|\nabla\ell^N(\mathbf b)|).
\]
Indeed, apply the coordinate inequalities and the Euclidean triangle
inequality to the vectors of their norms. Thus their additive
constants, unlike the Hessian constants, depend on $N$.
In particular,
\begin{equation}
H_\ell^N(\dbP^N\mid\dbQ^N)
\ge
\lambda_{\ell|\psi}
H_\psi^N(\dbP^N\mid\dbQ^N).
\label{eq-particle-proof-relative-convexity}
\end{equation}

\revstep{2}{Potential and initialization}
Moreover,
\[
\mathcal F_N(\dbQ^N)
=
\langle V_N,\dbQ^N\rangle,
\qquad
V_N(\mathbf x)
=
NF(\mu^N(\mathbf x)),
\]
is linear in $\dbQ^N$, with linear derivative \(\delta\mathcal F_N(\dbQ^N,\mathbf x) = V_N(\mathbf x),\) which is independent of $\dbQ^N$. To check boundedness, fix
$\overline{\dbQ}\in\mathcal P(\Omega)$. By the boundedness of
$\delta F$,
\[
|F(\dbQ)-F(\overline{\dbQ})|
\le
2C_F,
\qquad
\dbQ\in\mathcal P(\Omega),
\]
and therefore $V_N$ is bounded for every fixed $N$. Its
measure-Lipschitz constant is zero.

Finally, Assumption~\ref{ass-particle-initial} gives the required
initial BMO and integrability conditions. Thus all the assumptions
of Theorem~\ref{thm-natural-structure} hold on $\Omega^N$ for the data
$(\ell^N,\psi^N,\mathcal F_N)$. {For clarity, the proximal existence threshold can still be
$h_0=(2L_\ell)^{-1}$. In the absorption estimate
\eqref{eq-JKO-9}, the tensorised first-order condition gives
\[
 (1-hL_\ell)\|\nabla\psi^N(\boldsymbol\beta^{h,n+1})\|_{\BMO}
 \le\|\nabla\psi^N(\boldsymbol\beta^{h,n})\|_{\BMO}
       +h\|Z^{h,n+1}\|_{\BMO}+hL_\ell\sqrt{NT}.
\]
Consequently, for $h<h_0$,
\[
 \|\nabla\psi^N(\boldsymbol\beta^{h,n+1})\|_{\BMO}
 \le2\|\nabla\psi^N(\boldsymbol\beta^{h,n})\|_{\BMO}
       +2h\|Z^{h,n+1}\|_{\BMO}+2hL_\ell\sqrt{NT}.
\]
The other fixed-$N$ estimates in the proof of Theorem~\ref{thm-natural-structure}
apply with the tensorised bounds at zero, the bounded terminal cost
$V_N$, and the prescribed initial bounds. Their constants are allowed
to depend on $N$. Corollary~\ref{cor-common-flow-decay}, applied to this
same product-space problem, gives exponential decay towards its own
minimiser. This proves Theorem~\ref{thm-particle-flow}, with possibly
$N$-dependent $h_{S,N}$ and $C_{S,N}$, but with the unchanged contraction
exponent $\lambda_{\ell|\psi}$.}
\end{proof}

\subsubsection{The finite-particle defect}\label{sec-particle-defect-proof}

Let $(\dbQ^*,\beta^*)$ be the unique minimizer of \(J(\dbQ)=H_\ell(\dbQ)+F(\dbQ).\) For later use, define
\begin{equation*}
    \delta F^{\otimes N}(\dbQ^*,\mathbf x)
    :=\sum_{i=1}^N\delta F(\dbQ^*,x^i).
    \label{eq-particle-delta-F-tensor}
\end{equation*}
The first-order optimality condition reads
\begin{equation*}
    \delta H_\ell(\dbQ^*,\cdot)+\delta F(\dbQ^*,\cdot)=c_*
    \label{eq-particle-mf-optimality-proof}
\end{equation*}
for some deterministic constant $c_*\in\dbR$. By tensorization,
\begin{equation*}
    \delta H_\ell^N\left((\dbQ^*)^{\otimes N},\mathbf x\right)
    =\sum_{i=1}^N\delta H_\ell(\dbQ^*,x^i).
    \label{eq-particle-Hell-tensorization-proof}
\end{equation*}
Consequently,
\begin{equation}
    \delta H_\ell^N\left((\dbQ^*)^{\otimes N},\cdot\right)
    +\delta F^{\otimes N}(\dbQ^*,\cdot)=Nc_*.
    \label{eq-particle-mf-optimality-N-proof}
\end{equation}

\begin{lem}[Finite-particle potential defect]
\label{lem-particle-potential-defect-psi}
For every $\dbQ^N\in\mathcal P(\Omega^N)$,
\begin{equation}
    \left\langle
        V_N-\delta F^{\otimes N}(\dbQ^*,\cdot),
        (\dbQ^*)^{\otimes N}-\dbQ^N
    \right\rangle
    \leq\varepsilon_N.
    \label{eq-particle-potential-defect-proof}
\end{equation}
Moreover,
\[
    0\leq\varepsilon_N\leq C_{\mathrm{emp}},
\]
where $C_{\mathrm{emp}}$ is independent of $N$.
In addition, for every
$(\dbQ^N,\boldsymbol\beta)\in\mathcal A_\ell^N\cap\mathcal A_\psi^N$,
\[
    \lambda_{\ell|\psi}
    H_\psi^N\left(\dbQ^N\mid(\dbQ^*)^{\otimes N}\right)
    \leq
    H_\ell^N\left(\dbQ^N\mid(\dbQ^*)^{\otimes N}\right)
    \leq
    J_N(\dbQ^N)-N J(\dbQ^*).
\]
\end{lem}

\begin{proof}
\phantomsection\label{proof-particle-defect}

\revstep{1}{Convexity and the average marginal}
Let $\dbQ^{N,i}$ denote the $i$-th marginal of $\dbQ^N$, and set
\[
    \overline{\dbQ}^N:=\frac1N\sum_{i=1}^N\dbQ^{N,i}.
\]
Since $\dbE^{\dbQ^N}[\mu^N]=\overline{\dbQ}^N$, convexity gives
\[
    \dbE^{\dbQ^N}[F(\mu^N)]
    \geq F(\overline{\dbQ}^N)
    \geq F(\dbQ^*)
    +\left\langle
        \delta F(\dbQ^*,\cdot),\overline{\dbQ}^N-\dbQ^*
    \right\rangle.
\]
Also,
\[
    \left\langle\delta F^{\otimes N}(\dbQ^*,\cdot),\dbQ^N\right\rangle
    =N\left\langle\delta F(\dbQ^*,\cdot),\overline{\dbQ}^N\right\rangle.
\]
Therefore,
\begin{equation}
    \left\langle
        V_N-\delta F^{\otimes N}(\dbQ^*,\cdot),\dbQ^N
    \right\rangle
    \geq
    N F(\dbQ^*)
    -N\left\langle\delta F(\dbQ^*,\cdot),\dbQ^*\right\rangle.
    \label{eq-particle-defect-lower-comparison}
\end{equation}
For the product law, the definition of $\varepsilon_N$ yields
\begin{equation*}
\left\langle V_N-\delta F^{\otimes N}(\dbQ^*,\cdot), (\dbQ^*)^{\otimes N} \right\rangle = N F(\dbQ^*) -N\left\langle\delta F(\dbQ^*,\cdot),\dbQ^*\right\rangle +\varepsilon_N.
\end{equation*}
Subtracting \eqref{eq-particle-defect-lower-comparison} proves
\eqref{eq-particle-potential-defect-proof}.

\revstep{2}{Leave-one-out expansion}
Let $X^1,\ldots,X^N$ be independent with common law $\dbQ^*$, and set
\[
    \mu_k^N
    :=\frac1N\sum_{i=1}^k\delta_{X^i}
      +\frac{N-k}{N}\dbQ^*,
    \qquad k=0,\ldots,N.
\]
Thus $\mu_0^N=\dbQ^*$, $\mu_N^N=\mu^N$, and
\[
    \Delta_k^N:=\mu_k^N-\mu_{k-1}^N
    =\frac1N(\delta_{X^k}-\dbQ^*).
\]
Linear differentiability gives
\begin{equation}
    F(\mu_k^N)-F(\mu_{k-1}^N)
    =
    \left\langle\delta F(\mu_{k-1}^N,\cdot),\Delta_k^N\right\rangle
    +R_k^N,
    \label{eq-particle-leave-one-out-proof}
\end{equation}
where
\[
    R_k^N
    :=
    \int_0^1
    \left\langle
        \delta F(\mu_{k-1}^N+r\Delta_k^N,\cdot)
        -\delta F(\mu_{k-1}^N,\cdot),
        \Delta_k^N
    \right\rangle\,\dd r.
\]
By the total-variation Lipschitz bound on $\delta F$,
\[
    |R_k^N|
    \leq\frac{L_F}{2}\|\Delta_k^N\|_{\TV}^2
    \leq\frac{C}{N^2},
\]
with $C$ independent of $k$ and $N$.

\revstep{3}{Cancellation and summation}
Conditionally on $X^1,\ldots,X^{k-1}$, the measure $\mu_{k-1}^N$
is fixed and $X^k$ has law $\dbQ^*$. Hence
\[
    \dbE\left[
        \left.
        \left\langle
            \delta F(\mu_{k-1}^N,\cdot),\Delta_k^N
        \right\rangle
        \right|X^1,\ldots,X^{k-1}
    \right]=0.
\]
Taking expectations in \eqref{eq-particle-leave-one-out-proof}
and summing over $k=1,\ldots,N$ gives
\[
    \dbE^{(\dbQ^*)^{\otimes N}}[F(\mu^N)]-F(\dbQ^*)
    \leq\frac{C_{\mathrm{emp}}}{N}.
\]
On the other hand, Jensen's inequality and
$\dbE^{(\dbQ^*)^{\otimes N}}[\mu^N]=\dbQ^*$ imply that the left-hand
side is nonnegative. Thus
$0\leq\varepsilon_N\leq C_{\mathrm{emp}}$.

\revstep{4}{Objective comparison with the product optimum}
By tensorization of the running cost and the definition of
$\varepsilon_N$,
\begin{equation*}
J_N\left((\dbQ^*)^{\otimes N}\right) = N H_\ell(\dbQ^*) +N\dbE^{(\dbQ^*)^{\otimes N}}[F(\mu^N)] =N J(\dbQ^*)+\varepsilon_N.
\end{equation*}
Fix
$(\dbQ^N,\boldsymbol\beta)\in\mathcal A_\ell^N\cap\mathcal A_\psi^N$.
The Bregman identity of Lemma~\ref{lem-linear-derivative}, applied on
the product space, and \eqref{eq-particle-mf-optimality-N-proof} give
\begin{align*}
    J_N(\dbQ^N)-J_N\left((\dbQ^*)^{\otimes N}\right)
    &=
    H_\ell^N\left(\dbQ^N\mid(\dbQ^*)^{\otimes N}\right) +
\left\langle
        V_N-\delta F^{\otimes N}(\dbQ^*,\cdot),
        \dbQ^N-(\dbQ^*)^{\otimes N}
    \right\rangle\\
    &\geq
    H_\ell^N\left(\dbQ^N\mid(\dbQ^*)^{\otimes N}\right)
    -\varepsilon_N.
\end{align*}
The inequality follows from
\eqref{eq-particle-potential-defect-proof}. Consequently,
\[
    J_N(\dbQ^N)-N J(\dbQ^*)
    \geq H_\ell^N\left(\dbQ^N\mid(\dbQ^*)^{\otimes N}\right).
\]
The remaining inequality follows from
\eqref{eq-particle-proof-relative-convexity}.
\end{proof}

\subsubsection{Discrete defective contraction}\label{sec-particle-contraction-proof}

For $0<h<h_0$ and $n\geq0$, set
\begin{equation*}
    E_{h,n}^N
    :=H_\psi^N\left((\dbQ^*)^{\otimes N}\mid\dbQ_{h,n}^N\right), \quad R_{h,n}^N
    :=H_\psi^N\left(\dbQ_{h,n+1}^N\mid\dbQ_{h,n}^N\right).
    \label{eq-particle-proof-discrete-energy}
\end{equation*}
We also write
\[
    \rho_{h,n}^{N,\psi}:=\delta H_\psi^N(\dbQ_{h,n}^N,\cdot),
    \qquad
    \rho_{h,n}^{N,\ell}:=\delta H_\ell^N(\dbQ_{h,n}^N,\cdot).
\]

\begin{prop}[Discrete defective contraction]
\label{prop-particle-discrete-contraction-psi}
For every $0<h<h_0$ and every $n\geq0$,
\begin{equation}\label{eq-particle-proof-discrete-dissipation}
E_{h,n+1}^N-E_{h,n}^N \leq -R_{h,n}^N -h\lambda_{\ell|\psi}E_{h,n+1}^N -h\lambda_{\ell|\psi} H_\psi^N\left(\dbQ_{h,n+1}^N\mid(\dbQ^*)^{\otimes N}\right) +h\varepsilon_N.
\end{equation}
In particular,
\begin{equation}
    (1+\lambda_{\ell|\psi}h)E_{h,n+1}^N
    \leq E_{h,n}^N+h\varepsilon_N, \quad
    \label{eq-particle-proof-discrete-recursion}
\end{equation}
and
\begin{equation}
    E_{h,n}^N
    \leq
    (1+\lambda_{\ell|\psi}h)^{-n}E_{h,0}^N
    +\frac{\varepsilon_N}{\lambda_{\ell|\psi}}
    \left[1-(1+\lambda_{\ell|\psi}h)^{-n}\right].
    \label{eq-particle-proof-discrete-decay}
\end{equation}
Moreover, for every $n\geq1$,
\[
    0\leq J_N(\dbQ_{h,n}^N)-N J(\dbQ^*)
    \leq
    \frac{\lambda_{\ell|\psi}E_{h,0}^N}
         {(1+\lambda_{\ell|\psi}h)^n-1}
    +\varepsilon_N,
\]
and
\[
    H_\psi^N\left(\dbQ_{h,n}^N\mid(\dbQ^*)^{\otimes N}\right)
    \leq
    \frac{E_{h,0}^N}{(1+\lambda_{\ell|\psi}h)^n-1}
    +\frac{\varepsilon_N}{\lambda_{\ell|\psi}}.
\]
\end{prop}

\begin{proof}
\phantomsection\label{proof-particle-discrete-decay}

\revstep{1}{Three-point identity and stationarity}
The three-point Bregman identity gives
\begin{equation}\label{eq-particle-proof-three-point}
E_{h,n+1}^N-E_{h,n}^N =-R_{h,n}^N + \left\langle \rho_{h,n}^{N,\psi}-\rho_{h,n+1}^{N,\psi}, (\dbQ^*)^{\otimes N}-\dbQ_{h,n+1}^N \right\rangle.
\end{equation}
The first-order optimality condition for the proximal step is
\[
    \rho_{h,n+1}^{N,\ell}+V_N
    +\frac1h\left(\rho_{h,n+1}^{N,\psi}-\rho_{h,n}^{N,\psi}\right)
    =c_{h,n}^N
\]
for some deterministic constant $c_{h,n}^N\in\dbR$.
Pairing with $(\dbQ^*)^{\otimes N}-\dbQ_{h,n+1}^N$ eliminates
this constant and yields
\begin{equation}\label{eq-particle-proof-JKO-pairing}
\left\langle \rho_{h,n}^{N,\psi}-\rho_{h,n+1}^{N,\psi}, (\dbQ^*)^{\otimes N}-\dbQ_{h,n+1}^N \right\rangle = h\left\langle \rho_{h,n+1}^{N,\ell}+V_N, (\dbQ^*)^{\otimes N}-\dbQ_{h,n+1}^N \right\rangle.
\end{equation}

\revstep{2}{Separate entropy and potential}
By \eqref{eq-particle-mf-optimality-N-proof},
\begin{multline}
    \left\langle
        \rho_{h,n+1}^{N,\ell}+V_N,
        (\dbQ^*)^{\otimes N}-\dbQ_{h,n+1}^N
    \right\rangle
    =
    \left\langle
        \rho_{h,n+1}^{N,\ell}
        -\delta H_\ell^N\left((\dbQ^*)^{\otimes N},\cdot\right),
        (\dbQ^*)^{\otimes N}-\dbQ_{h,n+1}^N
    \right\rangle
     \\
     +
    \left\langle
        V_N-\delta F^{\otimes N}(\dbQ^*,\cdot),
        (\dbQ^*)^{\otimes N}-\dbQ_{h,n+1}^N
    \right\rangle.
    \label{eq-particle-proof-pairing-decomposition}
\end{multline}
By the symmetrized Bregman identity, the first pairing on the
right-hand side equals
\[
    -H_\ell^N\left((\dbQ^*)^{\otimes N}\mid\dbQ_{h,n+1}^N\right)
    -H_\ell^N\left(\dbQ_{h,n+1}^N\mid(\dbQ^*)^{\otimes N}\right).
\]
Relative convexity bounds this from above by
\[
    -\lambda_{\ell|\psi}E_{h,n+1}^N
    -\lambda_{\ell|\psi}
    H_\psi^N\left(\dbQ_{h,n+1}^N\mid(\dbQ^*)^{\otimes N}\right).
\]
The second pairing in
\eqref{eq-particle-proof-pairing-decomposition} is at most
$\varepsilon_N$ by Lemma~\ref{lem-particle-potential-defect-psi}.
Combining these bounds with
\eqref{eq-particle-proof-three-point} and
\eqref{eq-particle-proof-JKO-pairing} proves
\eqref{eq-particle-proof-discrete-dissipation}.

\revstep{3}{Iterate the defective contraction}
Since both Bregman divergences are nonnegative, dropping their
negative contributions gives
\eqref{eq-particle-proof-discrete-recursion}.
Set
\[
    a_h:=\frac1{1+\lambda_{\ell|\psi}h}.
\]
Then $E_{h,n+1}^N\leq a_hE_{h,n}^N+a_hh\varepsilon_N$, and iteration
gives
\[
    E_{h,n}^N
    \leq a_h^nE_{h,0}^N
       +h\varepsilon_N\sum_{k=1}^n a_h^k
    =a_h^nE_{h,0}^N
       +\frac{\varepsilon_N}{\lambda_{\ell|\psi}}(1-a_h^n).
\]
This is \eqref{eq-particle-proof-discrete-decay}.

\revstep{4}{Objective decay and forward divergence}
The minimizing property of the proximal step implies
\[
    J_N(\dbQ_{h,n+1}^N)+\frac1hR_{h,n}^N
    \leq J_N(\dbQ_{h,n}^N).
\]
In particular, the objective is non-increasing along the scheme.
Since $\mathcal F_N$ is linear, the Bregman identity also gives
\begin{equation*}
\left\langle \rho_{h,n+1}^{N,\ell}+V_N, (\dbQ^*)^{\otimes N}-\dbQ_{h,n+1}^N \right\rangle = J_N\left((\dbQ^*)^{\otimes N}\right)-J_N(\dbQ_{h,n+1}^N) -H_\ell^N\left((\dbQ^*)^{\otimes N}\mid\dbQ_{h,n+1}^N\right).
\end{equation*}
Use
$J_N((\dbQ^*)^{\otimes N})=N J(\dbQ^*)+\varepsilon_N$
and relative convexity in
\eqref{eq-particle-proof-three-point}--\eqref{eq-particle-proof-JKO-pairing}.
This yields
\begin{equation*}
(1+\lambda_{\ell|\psi}h)E_{h,n+1}^N+R_{h,n}^N +h\left[J_N(\dbQ_{h,n+1}^N)-N J(\dbQ^*)\right] \leq E_{h,n}^N+h\varepsilon_N.
\end{equation*}
For $n\geq1$, apply this inequality at indices $k-1$, multiply by
$(1+\lambda_{\ell|\psi}h)^{k-1}$, and sum over $k=1,\ldots,n$.
After dropping the nonnegative step divergences, we obtain
\begin{equation*}
h\sum_{k=1}^n(1+\lambda_{\ell|\psi}h)^{k-1} \left[J_N(\dbQ_{h,k}^N)-N J(\dbQ^*)\right] \leq E_{h,0}^N-(1+\lambda_{\ell|\psi}h)^nE_{h,n}^N +\frac{\varepsilon_N}{\lambda_{\ell|\psi}} \left[(1+\lambda_{\ell|\psi}h)^n-1\right].
\end{equation*}
For $k\leq n$, monotonicity gives
$J_N(\dbQ_{h,k}^N)\geq J_N(\dbQ_{h,n}^N)$. Therefore,
\begin{equation*}
\frac{(1+\lambda_{\ell|\psi}h)^n-1}{\lambda_{\ell|\psi}} \left[J_N(\dbQ_{h,n}^N)-N J(\dbQ^*)\right] \leq E_{h,0}^N +\frac{\varepsilon_N}{\lambda_{\ell|\psi}} \left[(1+\lambda_{\ell|\psi}h)^n-1\right].
\end{equation*}
Dividing by the positive prefactor gives the asserted objective
estimate. Finally, Lemma~\ref{lem-particle-potential-defect-psi} yields
\[
    \lambda_{\ell|\psi}
    H_\psi^N\left(\dbQ_{h,n}^N\mid(\dbQ^*)^{\otimes N}\right)
    \leq J_N(\dbQ_{h,n}^N)-N J(\dbQ^*),
\]
which proves both non-negativity of the objective difference and
the forward-divergence estimate.
\end{proof}

\subsubsection{Passage to continuous learning time}\label{sec-particle-continuation-proof}

The strong BMO convergence of the particle controls allows us to
pass to the limit in both the fixed-reference divergence and the
objective.

\begin{lem}
\label{lem-particle-psi-Bregman-convergence}
For every fixed $N\geq1$ and $s\geq0$, let $n_h:=\left\lfloor\frac{s}{h}\right\rfloor$.
Then
\begin{equation}
    H_\psi^N\left((\dbQ^*)^{\otimes N}\mid\dbQ_{h,n_h}^N\right)
    \longrightarrow
    H_\psi^N\left((\dbQ^*)^{\otimes N}\mid\dbQ_s^N\right)
    \qquad\text{as }h\to0^+.
    \label{eq-particle-proof-Bregman-convergence}
\end{equation}
Moreover,
\[
    J_N(\dbQ_{h,n_h}^N)\longrightarrow J_N(\dbQ_s^N)
    \qquad\text{as }h\to0^+.
\]
\end{lem}

\begin{proof}
\phantomsection\label{proof-particle-limit}
Let
\[
    \boldsymbol\beta^{*,N}
    :=(\beta^*(X^1),\ldots,\beta^*(X^N)),\qquad
    \boldsymbol\beta_h:=\boldsymbol\beta_{h,n_h}^N,\qquad
    \boldsymbol\beta:=\boldsymbol\beta_s^N.
\]
By Theorem~\ref{thm-particle-flow},
$\|\boldsymbol\beta_h-\boldsymbol\beta\|_{\BMO}\to0$.
For fixed $N$, the reference drift $\boldsymbol\beta^{*,N}$ is BMO,
and $\nabla^2\psi^N\leq C_\psi I_{Nd}$ by
\eqref{eq-particle-psi-Hessian}.
Lemma~\ref{lem-BMO-energy-continuity}(i), applied on $\Omega^N$
with fixed reference law $(\dbQ^*)^{\otimes N}$, proves
\eqref{eq-particle-proof-Bregman-convergence}.

The same BMO convergence places $\boldsymbol\beta_h$ and
$\boldsymbol\beta$ in a common BMO ball for all sufficiently small
$h$. Moreover, $\nabla^2\ell^N\leq C_\ell I_{Nd}$ by
\eqref{eq-particle-ell-Hessian}, and
$\delta\mathcal F_N=V_N$ is bounded for each fixed $N$.
Lemma~\ref{lem-BMO-energy-continuity}(ii), applied to the
product-space objective $J_N$, therefore gives
$J_N(\dbQ_{h,n_h}^N)\to J_N(\dbQ_s^N)$.
\end{proof}

\subsubsection{Proof of the defective long-time estimate}\label{sec-particle-long-time-proof}

We pass the discrete divergence and objective estimates to the
continuous flow, and then use the objective comparison to control
the forward divergence.

\begin{proof}[Proof of Theorem~\ref{thm-particle-long-time}]
\phantomsection\label{proof-particle-long-time}

\revstep{1}{The fixed-first-argument estimate}
Fix $N\geq1$ and $s\geq0$, and let
$n_h:=\lfloor s/h\rfloor$.
By Proposition~\ref{prop-particle-discrete-contraction-psi},
\begin{equation}
    E_{h,n_h}^N
    \leq
    (1+\lambda_{\ell|\psi}h)^{-n_h}E_{h,0}^N
    +\frac{\varepsilon_N}{\lambda_{\ell|\psi}}
        \left[1-(1+\lambda_{\ell|\psi}h)^{-n_h}\right].
    \label{eq-particle-proof-discrete-at-s}
\end{equation}
Since $n_hh\to s$,
\[
    (1+\lambda_{\ell|\psi}h)^{-n_h}
    \longrightarrow e^{-\lambda_{\ell|\psi}s}.
\]
Also,
\[
    E_{h,0}^N
    =H_\psi^N\left((\dbQ^*)^{\otimes N}\mid\dbQ_0^N\right)
\]
is independent of $h$. Lemma~\ref{lem-particle-psi-Bregman-convergence}
therefore allows us to pass to the limit in
\eqref{eq-particle-proof-discrete-at-s}, giving
\begin{equation*}
H_\psi^N\left((\dbQ^*)^{\otimes N}\mid\dbQ_s^N\right) \leq e^{-\lambda_{\ell|\psi}s} H_\psi^N\left((\dbQ^*)^{\otimes N}\mid\dbQ_0^N\right) + \frac{1-e^{-\lambda_{\ell|\psi}s}}{\lambda_{\ell|\psi}} \varepsilon_N.
\end{equation*}
This proves \eqref{eq-particle-continuous-decay}.

\revstep{2}{The objective and forward-divergence estimates}
Now fix $s>0$. For all sufficiently small $h$, we have $n_h\geq1$,
so Proposition~\ref{prop-particle-discrete-contraction-psi} gives
\[
    0\leq J_N(\dbQ_{h,n_h}^N)-N J(\dbQ^*)
    \leq
    \frac{\lambda_{\ell|\psi}E_{h,0}^N}
         {(1+\lambda_{\ell|\psi}h)^{n_h}-1}
    +\varepsilon_N.
\]
By Lemma~\ref{lem-particle-psi-Bregman-convergence},
$J_N(\dbQ_{h,n_h}^N)\to J_N(\dbQ_s^N)$.
Passing to the limit yields
\begin{equation}\label{eq-particle-objective-defect-decay}
    0\leq J_N(\dbQ_s^N)-N J(\dbQ^*)
    \leq
    \frac{\lambda_{\ell|\psi}
        H_\psi^N\bigl((\dbQ^*)^{\otimes N}\mid\dbQ_0^N\bigr)}
        {e^{\lambda_{\ell|\psi}s}-1}
    +\varepsilon_N.
\end{equation}
We use \eqref{eq-particle-objective-defect-decay} in the forward-coercivity estimate below.
Since $(\dbQ_s^N,\boldsymbol\beta_s^N)$ belongs to
$\mathcal A_\ell^N\cap\mathcal A_\psi^N$,
Lemma~\ref{lem-particle-potential-defect-psi} also gives
\[
    \lambda_{\ell|\psi}
    H_\psi^N\left(\dbQ_s^N\mid(\dbQ^*)^{\otimes N}\right)
    \leq J_N(\dbQ_s^N)-N J(\dbQ^*).
\]
Combining the last two inequalities proves
\[
    H_\psi^N\left(\dbQ_s^N\mid(\dbQ^*)^{\otimes N}\right)
    \leq
    \frac{
        H_\psi^N\left((\dbQ^*)^{\otimes N}\mid\dbQ_0^N\right)
    }{
        e^{\lambda_{\ell|\psi}s}-1
    }
    +\frac{\varepsilon_N}{\lambda_{\ell|\psi}},
\]
which is \eqref{eq-particle-reverse-continuous-decay}.

\revstep{3}{Long-time and per-particle estimates}
The initial divergence is finite by admissibility and the
fixed-$N$ BMO change-of-measure estimates. Letting $s\to\infty$
in the two divergence bounds above and dividing by $N$ gives
\begin{align*}
    \limsup_{s\to\infty}\frac1N
    H_\psi^N\left((\dbQ^*)^{\otimes N}\mid\dbQ_s^N\right)
    &\leq
    \frac{\varepsilon_N}{\lambda_{\ell|\psi}N}
    \leq
    \frac{C_{\mathrm{emp}}}{\lambda_{\ell|\psi}N},\\
    \limsup_{s\to\infty}\frac1N
    H_\psi^N\left(\dbQ_s^N\mid(\dbQ^*)^{\otimes N}\right)
    &\leq
    \frac{\varepsilon_N}{\lambda_{\ell|\psi}N}
    \leq
    \frac{C_{\mathrm{emp}}}{\lambda_{\ell|\psi}N}.
\end{align*}
Here we used Lemma~\ref{lem-particle-potential-defect-psi}.
These are \eqref{eq-particle-long-time-defect} and
\eqref{eq-particle-reverse-long-time-defect}.
\end{proof}

\subsubsection{\texorpdfstring{Marginal projection of the forward divergence}{Marginal projection of the forward divergence}}
\label{sec-particle-marginal-projection}
\begin{lem}[{Marginal projection of the forward divergence}]
\label{lem-particle-marginal-projection}
{Assume Assumptions~\ref{ass-rc}, \ref{ass-pe} and
\ref{ass-rc-quadratic}. Let
$(\dbQ^N,\boldsymbol\beta)\in\mathcal A_\ell^N\cap\mathcal A_\psi^N$,
and let $\dbQ^*$ be the minimizer of $J$.} Let $\dbQ^{N,i}$ denote the $i$-th marginal of $\dbQ^N$, and let
$\boldsymbol\beta=(\beta^1,\ldots,\beta^N)$ be the drift associated
with $\dbQ^N$ in the joint filtration.
{The marginal laws have finite drift energy and finite $\ell$-cost, and the following divergences are finite:}
\[
    \sum_{i=1}^N
    H_\psi(\dbQ^{N,i}\mid\dbQ^*)
    \leq
    H_\psi^N
    \left(
        \dbQ^N
        \mid
        (\dbQ^*)^{\otimes N}
    \right).
\]
\end{lem}

\begin{proof}
\phantomsection\label{proof-particle-marginal-projection}
{\revstep{1}{Projection and integrability}}
Fix $i\in\{1,\ldots,N\}$, and let
\[
    \mathcal F_t^i
    :=
    \sigma(X_r^i:0\leq r\leq t),
    \qquad t\in[0,T],
\]
with the usual augmentation.
Define
\[
    \widehat\beta_t^i
    :=
    \dbE^{\dbQ^N}
    \left[
        \beta_t^i
        \,\middle|\,
        \mathcal F_t^i
    \right],
\]
where a predictable version is chosen in the
$\dd t\otimes\dd\dbQ^N$ sense.
We identify $\widehat\beta^i$ with its corresponding progressively
measurable functional of the $i$-th coordinate.
The following integrability bounds justify conditional expectations and Fubini throughout the proof.
Indeed,
\[
    0
    \leq
    \psi(t,x,b)-\psi(t,x,a)
    -\nabla\psi(t,x,a)\cdot(b-a)
    \leq
    \frac{C_\psi}{2}|b-a|^2,
\]
and conditional Jensen's inequality also gives
\[
    \dbE^{\dbQ^N}
    \int_0^T|\widehat\beta_t^i|^2\,\dd t
    \leq
    \dbE^{\dbQ^N}
    \int_0^T|\beta_t^i|^2\,\dd t
    <\infty.
\]
The reference drift $\beta^*(X^i)$ is square-integrable under
$\dbQ^N$ by the BMO change-of-measure estimates.

{\revstep{2}{The marginal drift}}
We first identify $\widehat\beta^i$ as the drift of the marginal
law $\dbQ^{N,i}$.
Under $\dbQ^N$, in the joint filtration,
\[
    \dd X_t^i
    =
    \beta_t^i\,\dd t+\dd W_t^i.
\]
For $0\leq u<t\leq T$, the tower property gives
\begin{align*}
    &\dbE^{\dbQ^N}
    \left[
        X_t^i-X_u^i
        -\int_u^t\widehat\beta_r^i\,\dd r
        \,\middle|\,
        \mathcal F_u^i
    \right]
    =
    \int_u^t
    \dbE^{\dbQ^N}
    \left[
        \beta_r^i-\widehat\beta_r^i
        \,\middle|\,
        \mathcal F_u^i
    \right]\,\dd r
    =
    \int_u^t
    \dbE^{\dbQ^N}
    \left[
        \dbE^{\dbQ^N}
        \left[
            \beta_r^i-\widehat\beta_r^i
            \,\middle|\,
            \mathcal F_r^i
        \right]
        \,\middle|\,
        \mathcal F_u^i
    \right]\,\dd r
    =0.
\end{align*}
Consequently, $
    X_t^i-X_0^i-\int_0^t\widehat\beta_r^i\,\dd r$, $0 \leq t \leq T$,
is a continuous martingale in the coordinate filtration, with
quadratic variation $tI_d$.
It is therefore a Brownian motion in that filtration.
Thus $\widehat\beta^i$ is the canonical drift of $\dbQ^{N,i}$.
Since $\dbQ^N\sim\dbP^N$, its marginal is equivalent to
Wiener measure. The quadratic bound
$|\ell(t,\omega,b)|\le C(1+|b|^2)$ and Step~1 give finite $\ell$-cost
under that marginal law. Thus the generalized entropies used below
are well defined, without requiring a separate BMO assertion for
$\widehat\beta^i$ under Wiener measure.

Conditional convexity and summation
The reference law $\dbQ^*$ has drift $\beta^*$.
Since $\dbQ^{N,i}$ is the law of $X^i$ under $\dbQ^N$, the definition
of the relative generalized entropy gives
\begin{equation*}
H_\psi(\dbQ^{N,i}\mid\dbQ^*) = \dbE^{\dbQ^N} \int_0^T \Big[ \psi(t,X^i,\widehat\beta_t^i) -\psi(t,X^i,\beta_t^*(X^i)) -\nabla\psi(t,X^i,\beta_t^*(X^i)) \cdot \bigl( \widehat\beta_t^i-\beta_t^*(X^i) \bigr) \Big]\,\dd t.
\end{equation*}
By progressive measurability, the random function
$\psi(t,X^i,\cdot)$ and the reference drift $\beta_t^*(X^i)$
depend only on the history of $X^i$ up to time $t$.
In particular, $\psi(t,X^i,\beta_t^*(X^i))$ and $\nabla\psi(t,X^i,\beta_t^*(X^i))$
are $\mathcal F_t^i$-measurable.
The convexity of $\psi$ in its drift variable therefore allows us
to apply conditional Jensen's inequality:
\begin{equation*}
\psi(t,X^i,\widehat\beta_t^i) = \psi \left( t,X^i, \dbE^{\dbQ^N} [\beta_t^i\mid\mathcal F_t^i] \right) \leq \dbE^{\dbQ^N} \left[ \psi(t,X^i,\beta_t^i) \,\middle|\, \mathcal F_t^i \right].
\end{equation*}
Moreover, the affine term satisfies
\begin{equation*}
\dbE^{\dbQ^N} \left[ \nabla\psi(t,X^i,\beta_t^*(X^i)) \cdot \bigl( \beta_t^i-\beta_t^*(X^i) \bigr) \,\middle|\, \mathcal F_t^i \right] = \nabla\psi(t,X^i,\beta_t^*(X^i)) \cdot \bigl( \widehat\beta_t^i-\beta_t^*(X^i) \bigr).
\end{equation*}
Combining these two relations yields
\begin{multline*}
\psi(t,X^i,\widehat\beta_t^i)
    -\psi(t,X^i,\beta_t^*(X^i))
    -\nabla\psi(t,X^i,\beta_t^*(X^i))
    \cdot
    \bigl(
        \widehat\beta_t^i-\beta_t^*(X^i)
    \bigr)
    \\
    \leq
    \dbE^{\dbQ^N}
    \Big[
        \psi(t,X^i,\beta_t^i)
        -\psi(t,X^i,\beta_t^*(X^i))
        -\nabla\psi(t,X^i,\beta_t^*(X^i))
        \cdot
        \bigl(
            \beta_t^i-\beta_t^*(X^i)
        \bigr)
        \,\Big|\,
        \mathcal F_t^i
    \Big].
\end{multline*}
Taking expectations and integrating over $t\in[0,T]$, we obtain
\begin{equation*}
H_\psi(\dbQ^{N,i}\mid\dbQ^*) \leq \dbE^{\dbQ^N} \int_0^T \Big[ \psi(t,X^i,\beta_t^i) -\psi(t,X^i,\beta_t^*(X^i)) -\nabla\psi(t,X^i,\beta_t^*(X^i)) \cdot \bigl( \beta_t^i-\beta_t^*(X^i) \bigr) \Big]\,\dd t.
\end{equation*}
Finally, the drift of $(\dbQ^*)^{\otimes N}$ in the joint
filtration is
$(
        \beta_t^*(X^1),\ldots,\beta_t^*(X^N)
)$.
By the additive definition of $\psi^N$, summing the preceding
inequality over $i=1,\ldots,N$ gives
\[
    \sum_{i=1}^N
    H_\psi(\dbQ^{N,i}\mid\dbQ^*)
    \leq
    H_\psi^N
    \left(
        \dbQ^N
        \mid
        (\dbQ^*)^{\otimes N}
    \right).
\]

\end{proof}

\subsubsection{Product initial condition}\label{sec-particle-product-proof}
\begin{proof}[Proof of Corollary~\ref{cor-particle-product-initial}]
\phantomsection\label{proof-particle-product-initial}
Write $\beta_0$ for the drift of $\dbQ_0$.
The canonical drifts of the product laws $\dbQ_0^{\otimes N}$ and
$(\dbQ^*)^{\otimes N}$ are, respectively,
$(\beta_{0,t}(X^i))_{i=1}^N$ and $(\beta_t^*(X^i))_{i=1}^N$.
By the additive definition of $\psi^N$,
\begin{align*}
&H_\psi^N
    \bigl((\dbQ^*)^{\otimes N}\mid\dbQ_0^{\otimes N}\bigr)\\
&=
    \sum_{i=1}^N
    \dbE^{(\dbQ^*)^{\otimes N}}\!\int_0^T
    \Bigl[
        \psi(t,X^i,\beta_t^*(X^i))
        -
        \psi(t,X^i,\beta_{0,t}(X^i))
        -\nabla\psi(t,X^i,\beta_{0,t}(X^i))
        \cdot
        \bigl(\beta_t^*(X^i)-\beta_{0,t}(X^i)\bigr)
    \Bigr]\,\dd t\\
&=
    N H_\psi(\dbQ^*\mid\dbQ_0).
\end{align*}
Indeed, every summand is a functional of one coordinate, whose law
under $(\dbQ^*)^{\otimes N}$ is $\dbQ^*$.
The integrals are finite by the quadratic growth and BMO
change-of-measure estimates.
This proves \eqref{eq-particle-product-initial-entropy}.

Next, the product initial law
$\dbQ_0^N=\dbQ_0^{\otimes N}$ is exchangeable.
The admissible class and $J_N$ are invariant under coordinate
permutations, while $H_\psi^N$ is invariant under simultaneous
permutations of its two arguments.
By uniqueness of the minimizer at each JKO step, induction shows
that every $\dbQ_{h,n}^N$ is exchangeable.
Passing to the total variation limit, we conclude that
$\dbQ_s^N$ is exchangeable for every $s\geq0$.
Consequently, all its one-particle marginals coincide, and
Lemma~\ref{lem-particle-marginal-projection} gives, for every
$i=1,\ldots,N$,
\[
    H_\psi\left(\dbQ_s^{N,i}\mid\dbQ^*\right)
    =
    \frac1N\sum_{j=1}^N
    H_\psi\left(\dbQ_s^{N,j}\mid\dbQ^*\right)
    \leq
    \frac1N H_\psi^N
    \left(\dbQ_s^N\mid(\dbQ^*)^{\otimes N}\right).
\]
Finally, substitute \eqref{eq-particle-product-initial-entropy}
into \eqref{eq-particle-continuous-decay} and
\eqref{eq-particle-reverse-continuous-decay}, and divide by $N$.
Together with the marginal estimate above and
$0\leq\varepsilon_N\leq C_{\mathrm{emp}}$, this yields
\eqref{eq-particle-final-estimate} for $s\geq0$ and
\eqref{eq-particle-reverse-final-estimate} for $s>0$.
\end{proof}

\subsection{Markov case and reflection coupling }
\label{sec-reflection-proof}
The coupling lemma is proved first as a statement about a fixed bounded drift. 
The coupling acts in physical time. Its meeting probability regularises bounded functions, while its mean separation controls Lipschitz data; neither bound differentiates the drift.
\begin{proof}[Proof of Lemma~\ref{prop-reflection-evaluation}]
\phantomsection\label{proof-reflection-evaluation}

\textit{Step 1: the distance before coupling.}
First suppose that $b$ is smooth and bounded. Start two diffusions
at $x,y$ at time $t$. Until their meeting time $\tau$, reflect the
Brownian increment of the second in the hyperplane perpendicular to
$X_r-Y_r$; after meeting use the same Brownian increments. This is
the reflection construction of \cite{LindvallRogers1986}. If
$R_r=|X_r-Y_r|$ and $e_r=(X_r-Y_r)/R_r$, its radial equation before
meeting is
\[
 \dd R_r=2\,\dd B_r+
 e_r\cdot[b(r,X_r)-b(r,Y_r)]\,\dd r,
 \qquad e_r\cdot[b(r,X_r)-b(r,Y_r)]\le2B.
\]
There is no radial It\^o correction because the noise in the
separation is purely radial. Thus, up to the respective hitting
times of zero, $R$ is dominated by
$|x-y|+2(B_r-B_t)+2B(r-t)$, driven by the same one-dimensional
Brownian motion. In particular its first meeting occurs no later
than the first zero of this upper process.

\textit{Step 2: meeting and mean-distance estimates.}
Write $z=|x-y|$, $u=r-t$ and let $\Phi_0$ be the standard normal
distribution function. The reflection principle for Brownian motion
with constant drift gives the survival probability of the upper
process:
\[
 S(u,z)=\Phi_0\!\left(\frac z{2\sqrt u}+B\sqrt u\right)
 -e^{-Bz}\Phi_0\!\left(-\frac z{2\sqrt u}+B\sqrt u\right).
\]
This formula can also be checked from the absorbing heat kernel
$e^{B(w-z)/2-B^2u/2}
[(8\pi u)^{-1/2}e^{-(w-z)^2/(8u)}
 -(8\pi u)^{-1/2}e^{-(w+z)^2/(8u)}]$, integrated over $w>0$.
Since $S(u,0)=0$ and differentiation gives
$\partial_zS(u,z)\le(2\pi u)^{-1/2}+B$, we obtain
\begin{equation*}\label{eq-reflection-meeting}
 \dbP(\tau>r)\le \min\{1,z\eta_B(u)\}.
\end{equation*}
Stopping the radial equation at meeting and taking expectation,
using bounded drift and localisation of the Brownian integral, gives
\[
 \dbE R_r\le z+2B\int_0^u\dbP(\tau>t+v)\,\dd v
 \le z\left(1+2B\sqrt{2u/\pi}+2B^2u\right)=z m_B(u).
\]
For bounded $f$, the coupled processes agree after meeting, so
\[
 |P^b_{t,r}f(x)-P^b_{t,r}f(y)|
 \le\osc(f)\dbP(\tau>r)\le z\eta_B(u)\osc(f).
\]
For Lipschitz $k$, instead use
$|k(X_r)-k(Y_r)|\le\operatorname{Lip}(k)R_r$.
These are the two estimates in \eqref{eq-reflection-semigroup}.

\textit{Step 3: bounded measurable drifts and evaluation.}
For a bounded Borel $b$, smooth bounded approximations may be chosen
with the same bound, converging almost everywhere in time and space.
For every fixed starting point, Brownian transition densities and
dominated convergence give convergence of the stochastic logarithms
in $L^2$, hence in probability of their terminal exponentials.
Boundedness of all the drifts gives a uniform $L^p$ bound on these
exponentials for each finite $p$, by the exponential-martingale
identity. Thus their convergence is in $L^1$. Girsanov identifies
the limiting law with the unique weak diffusion having drift $b$;
the inverse bounded-drift change of measure gives uniqueness as well.
Expectations of every bounded measurable terminal function converge,
so the first estimate passes to the limit. For a linearly growing
Lipschitz function, truncate and use the uniform second moments of
$X_r=x+W_{r-t}+\int_t^r b\,\dd u$ to pass to the limit. The second
estimate follows. Applying these estimates to the terminal cost and
the time-integrated source proves \eqref{eq-reflection-evaluation}
by Fubini and the assumed integrability of its right-hand side.
\end{proof}

The order of the estimates is important: restarted static problems first bound the feedbacks; reflection then gives smoothing and stability, followed by the learning equation and its convergence estimates.
\begin{proof}[Proof of Theorem~\ref{thm-markov-regularity}]
\phantomsection\label{proof-markov-regularity}

We first work with smooth approximations of the data. Bounds in
Steps 1--5 depend only on items (i)--(iii); the passage to the stated
regularity is detailed in Step 6. All spatial suprema are over
$\R^d$, and constants may depend on the fixed physical horizon.

\par\noindent\textit{Step 1: a uniform bound on feedbacks.}
Use the static minimisers constructed in
Theorem~\ref{thm-newton-structure}. Set
\[
 k^0(t,x)=q^0(t,x)\cdot\beta^0(t,x)-\ell(t,x,\beta^0(t,x))
        =H(t,x,-q^0(t,x)).
\]
Both $k^0$ and $D_xk^0$ are bounded. For $a\in[0,1]$, the static
running and terminal costs are
\begin{equation*}\label{eq-markov-static-cost}
 g_a(t,x,b)=\ell(t,x,b)-a q^0(t,x)\cdot b+a k^0(t,x),
 \qquad (1-a)g(x).
\end{equation*}
Let $U_a(t,x)$ be the value of this same static control problem
restarted at $(t,x)$. This is an application of the variational
construction, not an additional evolution equation in a parameter.
At $a=0$, $U_0=v^*$, and the control with $a=e^{-s}$ is $\beta^s$.

We verify bounds uniform in $a$ and in the restart. If
$p_b=-\nabla_b\ell(t,x,b)$, convexity of $H$ gives
\[
 c|p_b|^2-2C\le H(t,x,p_b)-H(t,x,0)
 \le D_pH(t,x,p_b)\cdot p_b\le |b|\,|p_b|.
\]
It follows that $|p_b|\le C(1+|b|)$. The envelope identity then yields
\begin{equation}\label{eq-markov-ell-x}
 |D_x\ell(t,x,b)|=|D_xH(t,x,p_b)|\le C(1+|b|).
\end{equation}
Moreover the lower Hessian bound and the bounds at zero imply
$g_a(t,x,b)\ge c'|b|^2-C'$ with $c'>0$, uniformly in $a$.
The zero-control competitor and bounded terminal data bound $U_a$
from above, while this coercivity bounds it below. Consequently any
control within one unit of the optimal value satisfies
\begin{equation*}\label{eq-markov-restarted-energy}
 \dbE_{t,x}\int_t^T|b_r|^2\,\dd r\le C_T,
\end{equation*}
with a constant independent of $a,t,x$.
Run the same open-loop control and the same Brownian path from
$x$ and $y$. Their states differ by $x-y$. By
\eqref{eq-markov-ell-x} and the bounded derivatives of $q^0,k^0$,
$|D_xg_a(t,x,b)|\le C(1+|b|)$. The costs under this coupling differ
by at most
\[
 |x-y|\left[\operatorname{Lip}(g)+CT+
                C\dbE\int_t^T|b_r|\,\dd r\right]
 \le C_T|x-y|.
\]
Use an arbitrarily close-to-optimal control, then interchange $x,y$,
to obtain $\operatorname{Lip}(U_a(t,\cdot))\le C_T$.
This argument can first be made with $|b|\le M$. Its energy and
Lipschitz bounds do not depend on $M$. The minimiser of the
unconstrained Hamiltonian on the resulting gradient ball is bounded
uniformly in $M,a,t,x$; for $M$ larger than this bound the constraint
is inactive. Verification identifies the obtained value and feedback
with the unique static minimiser. In particular the optimal feedback is
\begin{equation}\label{eq-markov-feedback-static}
 \widehat\beta^a(t,x)
 =-D_pH(t,x,\nabla_xU_a(t,x)-a q^0(t,x)),\qquad
 \nabla_b\ell(\widehat\beta^a)=a q^0-\nabla_xU_a,
\end{equation}
and $\sup_a\|\widehat\beta^a\|_\infty\le B$ for a fixed $B$.
We enlarge $B$ once to bound the feedbacks obtained from every
line segment between the bounded gradients and mirror coordinates
appearing below. This preliminary argument is essential: a
coupling estimate whose constant depends on $\|\beta^s\|_\infty$
would otherwise not close the feedback bound.

\par\noindent\textit{Step 2: smoothing by reflection.}
The envelope identity for the optimal value gives, for each spatial
coordinate $j$,
\[
 (\partial_t+\tfrac12\Delta+\widehat\beta^a\cdot\nabla)
       \partial_jU_a
 =-\partial_{x_j}g_a(t,x,\widehat\beta^a),
 \qquad \partial_jU_a(T)=(1-a)\partial_jg.
\]
There is no derivative of $\widehat\beta^a$ on the right: its
coefficient cancels by the drift first-order condition. The source
is uniformly bounded by Step 1. The bounded-terminal-function
estimate in \eqref{eq-reflection-semigroup} therefore yields
\begin{equation*}\label{eq-markov-U-hessian}
 \sup_a\|D_x^2U_a(t)\|_\infty
 \le C\left[\eta_B(T-t)+\int_t^T\eta_B(r-t)\,\dd r\right]
 \le C_T\bigl(1+(T-t)^{-1/2}\bigr).
\end{equation*}
Differentiating \eqref{eq-markov-feedback-static}, using the local
bounds on $H_{px},H_{pp}$, proves the same bound for
$D_x\widehat\beta^a$.

The policy-evaluation identity \eqref{eq-markov-evaluation} has
bounded running and terminal costs, so $\|v^s\|_\infty\le C$.
Apply \eqref{eq-reflection-evaluation} with drift $\beta^s$ and
source $\ell(t,x,\beta^s)$ to obtain
$\sup_s\|\nabla_xv^s\|_\infty\le C_T$.
Differentiating the evaluation PDE gives
\[
 (\partial_t+\tfrac12\Delta+\beta^s\cdot\nabla)\partial_jv^s
 =-\partial_{x_j}\ell(t,x,\beta^s)
  -(\partial_j\beta^s)\cdot
        [\nabla_b\ell(t,x,\beta^s)+\nabla_xv^s].
\]
Here the partial derivative $\partial_{x_j}\ell$ holds $b$ fixed.
The bracket is uniformly bounded, and the source is bounded by
$C(1+(T-t)^{-1/2})$. Reflection applied once more proves
\eqref{eq-markov-smoothing-low}; the only possibly singular source
integral is controlled by
\[
 \int_t^T(r-t)^{-1/2}(T-r)^{-1/2}\,\dd r=\pi.
\]
This proves the asserted uniform low-order estimates.

\par\noindent\textit{Step 3: feedback stability and learning-time rates.}
Consider two of the static problems, with indices $a,a'$. Their
Hamiltonians are
$\cH_{g_a}(t,x,z)=-H(t,x,z-aq^0)+ak^0$.
Linearising their difference in the gradient gives a bounded drift
$\lambda$, and
\[
 (\partial_t+\tfrac12\Delta+\lambda\cdot\nabla)(U_a-U_{a'})
       =-f_{a,a'},\qquad
 (U_a-U_{a'})(T)=-(a-a')g.
\]
All gradients lie in the ball fixed in Step 1. Since
\[
 \partial_a\cH_{g_a}(t,x,z)
       =k^0(t,x)-q^0(t,x)\cdot b^a(t,x,z),
\]
we have $\|\lambda\|_\infty\le B$ and
$\|f_{a,a'}\|_\infty\le C|a-a'|$.
Conditional expectation and the reflection evaluation estimate give
\begin{equation}\label{eq-markov-static-sup-stability}
 \|U_a-U_{a'}\|_\infty+
 \|\nabla_xU_a-\nabla_xU_{a'}\|_\infty
 +\|\widehat\beta^a-\widehat\beta^{a'}\|_\infty
 \le C_T|a-a'|.
\end{equation}
For the feedback term, use the inverse-gradient bound
$c_\ell^{-1}$ in \eqref{eq-markov-feedback-static}.
Taking $a=e^{-s}$ and $a'=0$ proves
\eqref{eq-markov-feedback-rate}. The discrete identification in
Theorem~\ref{thm-newton-structure} and its elementary bound
$\sup_s|(1+h)^{-\lfloor s/h\rfloor}-e^{-s}|\le Ch$ prove
\eqref{eq-markov-feedback-discrete}.

\par\noindent\textit{Step 4: a quadratic value error.}
Optimality at $\beta^*$ and local upper bounds on the Hessian of
$\ell$ imply
\begin{equation*}
0\le d_s(t,x):=\ell(t,x,\beta^s)-\ell(t,x,\beta^*) +(\beta^s-\beta^*)\cdot\nabla_xv^* \le C|\beta^s-\beta^*|^2\le Ce^{-2s}.
\end{equation*}
To justify the uniform upper bound, a bounded $b$ corresponds to a
bounded $p_b$ by Step 1, and Legendre inversion gives
$\ell_{bb}(t,x,b)=[H_{pp}(t,x,p_b)]^{-1}\le\theta_R^{-1}I_d$.
Subtracting the optimal HJB equation from the policy-evaluation
PDE gives
\[
 (\partial_t+\tfrac12\Delta+\beta^s\cdot\nabla)(v^s-v^*)=-d_s,
 \qquad (v^s-v^*)(T)=0.
\]
Thus $v^s-v^*=\int_t^T P^{\beta^s}_{t,r}d_s(r)\,\dd r\ge0$.
This proves the value bound, and reflection applied to the bounded
source proves its gradient bound in \eqref{eq-markov-value-rate}.

\par\noindent\textit{Step 5: the direct policy-improvement equation.}
The Markov laws just identified satisfy the centred entropy equation
\eqref{eq-main-dual-affine}. Write
$q^s(t,x)=\nabla_b\ell(t,x,\beta^s(t,x))$. By
\eqref{eq-markov-feedback-static} and
\eqref{eq-markov-static-sup-stability}, $q^s$ is uniformly bounded
and locally Lipschitz in $s$, pointwise and in the sup norm.
For almost every $s$, differentiation of the entropy-derivative
formula in $L^2(\dbP)$ gives
\[
 \partial_s\delta H_\ell(\dbQ^s)
      =\int_0^T\partial_sq^s(t,X_t)\cdot\dd W_t^{\beta^s}.
\]
Indeed the terms $q^s\cdot\partial_s\beta^s$ from differentiating
the running cost and the drift correction cancel. Pointwise
Rademacher differentiation, Fubini and the bounded difference
quotients justify differentiation of the time integral and, by
It\^o isometry, of the stochastic integral. Equivalence with bounded
drift gives the same integrability under $\dbQ^s$.
The actual policy-evaluation BSDE gives
\[
 \delta H_\ell(\dbQ^s)+g(X_T)
   =v^s(0,0)+\int_0^T
            [q^s(t,X_t)+\nabla_xv^s(t,X_t)]\cdot\dd W_t^{\beta^s}.
\]
The centred entropy equation says that the sum of the two left-hand
sides is deterministic. On the right the stochastic integral has
mean zero under $\dbQ^s$. Its conditional isometry therefore implies
\[
 \partial_sq^s(t,X_t)=-q^s(t,X_t)-\nabla_xv^s(t,X_t)
             \quad\dd t\otimes\dd\dbP\text{-a.e.}
\]
Integration in $s$, followed by continuity of the spatial versions
for $t<T$ and positivity of the Brownian transition density for
$t>0$, gives \eqref{eq-markov-direct-flow} in integral form for all
these $(t,x)$. Continuity extends it to $t=0$.

For uniqueness, on a bounded ball of mirror coordinates the inverse
$\nabla_b\ell$ and the running cost on the resulting feedback ball
are uniformly Lipschitz. Evaluate two such feedbacks $b,b'$.
Their value difference, written with drift $b$, has zero terminal
value and source
\[
 \ell(t,x,b)-\ell(t,x,b')+(b-b')\cdot\nabla_xv^{b'}.
\]
Its sup norm is at most $C\|b-b'\|_\infty$, since the evaluation
gradients are uniformly bounded by Lemma~\ref{prop-reflection-evaluation}.
Applying that proposition again shows
$\|\nabla_xv^b-\nabla_xv^{b'}\|_\infty\le C_T\|b-b'\|_\infty$.
Thus the right-hand side of the equation for $q$ is locally Lipschitz
in the sup norm. The integral equation and Gronwall give uniqueness
on every interval where two solutions are bounded. The constructed
solution is uniformly bounded for all $s$, proving the claim.

\par\noindent\textit{Step 6: passage to the stated data.}
The preceding differentiations may be justified first with smooth
$H,q^0,g$ and bounded controls. Smooth the Hamiltonian in time,
space and momentum and define its running cost by Legendre duality.
Convexity, the quadratic growth bounds and $|H_x|\le C(1+|p|)$ are
preserved up to fixed constants; local lower bounds on $H_{pp}$
are retained on slightly enlarged momentum balls. The bound $H_{pp}\le c_\ell^{-1}I_d$ implied by the uniform convexity of $\ell$ is preserved as well. Smooth the
bounded Lipschitz terminal datum and $q^0$; their relevant bounds
remain uniform.

Step 1 proves bounds independent of the control truncation and
smoothing. On compact subintervals of $[0,T)$ the differentiated
value equation in Step 2 is a linear uniformly parabolic equation
with bounded drift and bounded source. Local parabolic estimates,
or the heat-kernel integral equation with the bounded drift term
absorbed on short intervals, give compactness of the gradients.
The estimates of Step 2 and local uniform continuity of the
coefficients identify the limiting derivatives and feedbacks.
The corresponding bounded-drift densities converge in $L^1$ by the
Girsanov argument in the preceding proposition. Passing in the
bounded running costs and the verification identities identifies
the limit with the unique static minimiser of the original data.
The same argument identifies policy evaluation. Terminal convergence
is uniform for values, because $g$ is uniformly continuous; its
bounded weak gradient is sufficient for the smoothing estimates.
All constants in Steps 1--5 depend only on the unsmoothed bounds.
Their inequalities pass to the limit, and the integral policy
improvement identity passes by bounded convergence. This proves the
theorem without a smoothness assumption on the terminal gradient.
\end{proof}

\subsubsection*{Funding}
Z.\,R.'s research is supported by the Finance For Energy Market Research Centre,
the France 2030 grant (ANR-21-EXES-0003), and the PEPR PDE-AI project.

\printbibliography[heading=bibintoc,title={References}]

@Article{CCRW2026,
  author  = {Claisse, Julien and Conforti, Giovanni and Ren, Zhenjie
             and Wang, Songbo},
  title   = {Mean field optimization problem regularized by Fisher information},
  journal = {The Annals of Applied Probability},
  year    = {2026},
  volume  = {36},
  number  = {2},
  pages   = {1002--1047},
  doi     = {10.1214/25-AAP2225}
}

@Article{CLRW2024,
  author  = {Chen, Fan and Lin, Yiqing and Ren, Zhenjie and Wang, Songbo},
  title   = {Uniform-in-time propagation of chaos for kinetic mean field Langevin dynamics},
  journal = {Electronic Journal of Probability},
  year    = {2024},
  volume  = {29},
  pages   = {1--43},
  note    = {Paper No. 17},
  doi     = {10.1214/24-EJP1079}
}

@Article{CRW2025,
  author  = {Chen, Fan and Ren, Zhenjie and Wang, Songbo},
  title   = {Uniform-in-time propagation of chaos for mean field Langevin dynamics},
  journal = {Annales de l'Institut Henri Poincar{\'e},
             Probabilit{\'e}s et Statistiques},
  year    = {2025},
  volume  = {61},
  number  = {4},
  pages   = {2357--2404},
  doi     = {10.1214/24-AIHP1499}
}

@article{AlvarezBolteBrahic2004,
 author={Alvarez, Felipe and Bolte, J{\'e}r{\^o}me and Brahic, Olivier},
 title={Hessian {Riemannian} gradient flows in convex programming},
 journal={SIAM Journal on Control and Optimization},
 year={2004},volume={43},number={2},pages={477--501},
 doi={10.1137/S0363012902419977}
}

@article{Amari1998,
 author={Amari, Shun-ichi},
 title={Natural Gradient Works Efficiently in Learning},
 journal={Neural Computation},
 year={1998},volume={10},number={2},pages={251--276},
 doi={10.1162/089976698300017746}
}

@article{BoueDupuis1998,
 author={Bou{\'e}, Michelle and Dupuis, Paul},
 title={A variational representation for certain functionals of {Brownian} motion},
 journal={The Annals of Probability},
 year={1998},volume={26},number={4},pages={1641--1659},
 doi={10.1214/aop/1022855876}
}

@misc{Chaintron2024,
 author={Chaintron, Louis-Pierre},
 title={Quasi-continuity method for mean-field diffusions: large deviations and central limit theorem},
 year={2024},
 eprint={2410.04935},
 eprinttype={arxiv}
}

@misc{ChikvinidzeMania2014,
 author={Chikvinidze, Besik and Mania, Michael},
 title={New proofs of some results on {BMO} martingales using {BSDEs}},
 year={2012},
 eprint={1205.1249},
 eprinttype={arxiv}
}

@inproceedings{Fazel2018,
 author={Fazel, Maryam and Ge, Rong and Kakade, Sham M. and Mesbahi, Mehran},
 title={Global Convergence of Policy Gradient Methods for the Linear Quadratic Regulator},
 booktitle={Proceedings of the 35th International Conference on Machine Learning},
 series={Proceedings of Machine Learning Research},
 volume={80},
 year={2018},
 pages={1467--1476}
}

@inproceedings{Kakade2002,
 author={Kakade, Sham M.},
 title={A Natural Policy Gradient},
 booktitle={Advances in Neural Information Processing Systems},
 volume={14},
 year={2002},
 publisher={MIT Press}
}

@article{KerimkulovFR2026,
 author={Kerimkulov, Bekzhan and Leahy, James-Michael and {\v S}i{\v s}ka, David and Szpruch, {\L}ukasz and Zhang, Yufei},
 title={A {Fisher--Rao} Gradient Flow for Entropy-Regularised {Markov} Decision Processes in {Polish} Spaces},
 journal={Foundations of Computational Mathematics},
 year={2026},
 volume={26},
 pages={2395--2469},
 doi={10.1007/s10208-025-09729-3}
}

@article{KerimkulovHoward2020,
 author={Kerimkulov, Bekzhan and {\v S}i{\v s}ka, David and Szpruch, {\L}ukasz},
 title={Exponential Convergence and Stability of {Howard}'s Policy Improvement Algorithm for Controlled Diffusions},
 journal={SIAM Journal on Control and Optimization},
 year={2020},
 volume={58},
 number={3},
 pages={1314--1340},
 eprint={1812.07846},
 eprinttype={arxiv}
}

@article{KerimkulovMirror2025,
 author={Kerimkulov, Bekzhan and {\v S}i{\v s}ka, David and Szpruch, {\L}ukasz and Zhang, Yufei},
 title={Mirror descent for stochastic control problems with measure-valued controls},
 journal={Stochastic Processes and their Applications},
 year={2025},
 volume={190},
 pages={104765},
 doi={10.1016/j.spa.2025.104765}
}

@article{LuFreundNesterov2018,
 author={Lu, Haihao and Freund, Robert M. and Nesterov, Yurii},
 title={Relatively Smooth Convex Optimization by First-Order Methods, and Applications},
 journal={SIAM Journal on Optimization},
 year={2018},
 volume={28},
 number={1},
 pages={333--354},
 doi={10.1137/16M1099546}
}

@article{Mueller2025,
 author={M{\"u}ller, Johannes and {\c C}ayc{\i}, Semih and Mont{\'u}far, Guido},
 title={Fisher--Rao Gradient Flows of Linear Programs and State-Action Natural Policy Gradients},
 journal={SIAM Journal on Optimization},
 year={2025},
 volume={35},
 number={2},
 doi={10.1137/24M1653422},
 eprint={2403.19448},
 eprinttype={arxiv}
}

@article{RaskuttiMukherjee2015,
 author={Raskutti, Garvesh and Mukherjee, Sayan},
 title={The information geometry of mirror descent},
 journal={IEEE Transactions on Information Theory},
 year={2015},
 volume={61},
 number={3},
 pages={1451--1457},
 doi={10.1109/TIT.2015.2388583}
}

@article{SethiAnnealing2025,
 author={Sethi, Deven and {\v S}i{\v s}ka, David and Zhang, Yufei},
 title={Entropy Annealing for Policy Mirror Descent in Continuous Time and Space},
 journal={SIAM Journal on Control and Optimization},
 year={2025},
 volume={63},
 number={4},
 pages={3006--3041},
 doi={10.1137/24M166591X}
}

@misc{SethiConstrained2026,
 author={Sethi, Deven and {\v S}i{\v s}ka, David},
 title={Mirror descent for constrained stochastic control problems},
 year={2026},
 eprint={2506.02564},
 eprinttype={arxiv},
 note={Version 2, 8 September 2026; first posted in 2025}
}

@misc{SethiMSA2023,
 author={Sethi, Deven and {\v S}i{\v s}ka, David},
 title={The Modified {MSA}, a Gradient Flow and Convergence},
 year={2023},
 eprint={2212.05784},
 eprinttype={arxiv},
 note={Version 2}
}

@article{SiskaSzpruch2024,
 author={{\v S}i{\v s}ka, David and Szpruch, {\L}ukasz},
 title={Gradient Flows for Regularized Stochastic Control Problems},
 journal={SIAM Journal on Control and Optimization},
 year={2024},
 volume={62},
 pages={2036--2070},
 doi={10.1137/20M1373645}
}

@article{Rockafellar1976,
 author={Rockafellar, R. Tyrrell},
 title={Monotone Operators and the Proximal Point Algorithm},
 journal={SIAM Journal on Control and Optimization},
 year={1976},
 volume={14},
 number={5},
 pages={877--898},
 doi={10.1137/0314056}
}

@article{Teboulle1992,
 author={Teboulle, Marc},
 title={Entropic Proximal Mappings with Applications to Nonlinear Programming},
 journal={Mathematics of Operations Research},
 year={1992},
 volume={17},
 number={3},
 pages={670--690},
 doi={10.1287/moor.17.3.670}
}

@article{BeckTeboulle2003,
 author={Beck, Amir and Teboulle, Marc},
 title={Mirror descent and nonlinear projected subgradient methods for convex optimization},
 journal={Operations Research Letters},
 year={2003},
 volume={31},
 number={3},
 pages={167--175},
 doi={10.1016/S0167-6377(02)00231-6}
}

@book{Beck2017,
 author={Beck, Amir},
 title={First-Order Methods in Optimization},
 publisher={Society for Industrial and Applied Mathematics},
 year={2017},
 series={MOS-SIAM Series on Optimization},
 volume={25},
 doi={10.1137/1.9781611974997}
}

@book{Amari2016,
 author={Amari, Shun-ichi},
 title={Information Geometry and Its Applications},
 publisher={Springer},
 year={2016},
 series={Applied Mathematical Sciences},
 volume={194},
 doi={10.1007/978-4-431-55978-8}
}

@book{AyJostLeSchwachhofer2017,
 author={Ay, Nihat and Jost, J{\"u}rgen and L{\^e}, H{\^o}ng V{\^a}n and Schwachh{\"o}fer, Lorenz},
 title={Information Geometry},
 publisher={Springer},
 year={2017},
 series={Ergebnisse der Mathematik und ihrer Grenzgebiete. 3. Folge},
 volume={64},
 doi={10.1007/978-3-319-56478-4}
}

@article{Lehec2013,
 author={Lehec, Joseph},
 title={Representation formula for the entropy and functional inequalities},
 journal={Annales de l'Institut Henri Poincar{\'e}, Probabilit{\'e}s et Statistiques},
 year={2013},
 volume={49},
 number={3},
 pages={885--899},
 doi={10.1214/11-AIHP464}
}

@book{Deimling1977,
  author    = {Deimling, Klaus},
  title     = {Ordinary Differential Equations in Banach Spaces},
  series    = {Lecture Notes in Mathematics},
  volume    = {596},
  publisher = {Springer},
  address   = {Berlin, Heidelberg},
  year      = {1977},
  doi       = {10.1007/BFb0091636}
}

@article{BolleyVillani2005,
  author  = {Bolley, Fran{\c{c}}ois and Villani, C{\'e}dric},
  title   = {Weighted {Csisz{\'a}r--Kullback--Pinsker} inequalities
             and applications to transportation inequalities},
  journal = {Annales de la Facult{\'e} des sciences de Toulouse :
             Math{\'e}matiques},
  volume  = {14},
  number  = {3},
  pages   = {331--352},
  year    = {2005},
  doi     = {10.5802/afst.1095}
}

@book{billingsley1999convergence,
  author    = {Billingsley, Patrick},
  title     = {Convergence of Probability Measures},
  edition   = {2},
  publisher = {John Wiley \& Sons},
  address   = {New York},
  year      = {1999}
}

@article{cobzas2006fixed,
  author  = {Cobza{\c{s}}, Stefan},
  title   = {Fixed Point Theorems in Locally Convex Spaces---The Schauder Mapping Method},
  journal = {Fixed Point Theory and Applications},
  year    = {2006},
  volume  = {2006},
  pages   = {1--13},
  doi     = {10.1155/FPTA/2006/57950}
}

@book{karatzas1991brownian,
  title     = {Brownian Motion and Stochastic Calculus},
  author    = {Karatzas, Ioannis and Shreve, Steven E.},
  edition   = {2},
  publisher = {Springer},
  year      = {1991}
}

@book{rockafellar1998variational,
  title     = {Variational Analysis},
  author    = {Rockafellar, R. Tyrrell and Wets, Roger J.-B.},
  series    = {Grundlehren der mathematischen Wissenschaften},
  volume    = {317},
  publisher = {Springer},
  address   = {Berlin},
  year      = {1998}
}

@article{briand2008quadratic,
  author  = {Philippe Briand and Ying Hu},
  title   = {Quadratic BSDEs with convex generators and unbounded terminal conditions},
  journal = {Probability Theory and Related Fields},
  volume  = {141},
  number  = {3--4},
  pages   = {543--567},
  year    = {2008},
  doi     = {10.1007/s00440-007-0093-y}
}

@article{delbaen2011uniqueness,
  author  = {Freddy Delbaen and Ying Hu and Adrien Richou},
  title   = {On the uniqueness of solutions to quadratic BSDEs with convex generators and unbounded terminal conditions},
  journal = {Annales de l'Institut Henri Poincar{\'e}, Probabilit{\'e}s et Statistiques},
  volume  = {47},
  number  = {2},
  pages   = {559--574},
  year    = {2011},
  doi     = {10.1214/10-AIHP372}
}

@book{kazamaki1994continuous,
  author    = {Norihiko Kazamaki},
  title     = {Continuous Exponential Martingales and BMO},
  series    = {Lecture Notes in Mathematics},
  volume    = {1579},
  publisher = {Springer-Verlag},
  address   = {Berlin},
  year      = {1994},
  doi       = {10.1007/BFb0073585}
}

@article{nielsen2020elementary,
  title={An elementary introduction to information geometry},
  author={Nielsen, Frank},
  journal={Entropy},
  volume={22},
  number={10},
  pages={1100},
  year={2020},
  publisher={MDPI}
}

@article{aubin2022mirror,
  title={Mirror descent with relative smoothness in measure spaces, with application to Sinkhorn and EM},
  author={Aubin-Frankowski, Pierre-Cyril and Korba, Anna and L{\'e}ger, Flavien},
  journal={Advances in Neural Information Processing Systems},
  volume={35},
  pages={17263--17275},
  year={2022}
}

@article{aubin2025evolution,
 author={Aubin-Frankowski, Pierre-Cyril and Sodini, Giacomo Enrico and Stefanelli, Ulisse},
 title={Evolution variational inequalities with general costs},
 journal={Journal of Functional Analysis},year={2026},volume={291},
 number={1},pages={111469},
 doi={10.1016/j.jfa.2026.111469},
 eprint={2505.00559},
 eprinttype={arxiv}
}

@article{backhoff2020nonexponential,
 author={Backhoff-Veraguas, Julio and Lacker, Daniel and Tangpi, Ludovic},
 title={Nonexponential {Sanov} and {Schilder} theorems on {Wiener} space: {BSDEs}, {Schr{\"o}dinger} problems and control},
 journal={The Annals of Applied Probability},year={2020},volume={30},number={3},pages={1321--1367}
 ,doi={10.1214/19-AAP1531}
 }

@article{lacker2020non,
  title={A non-exponential extension of Sanov’s theorem via convex duality},
  author={Lacker, Daniel},
  journal={Advances in Applied Probability},
  volume={52},
  number={1},
  pages={61--101},
  year={2020},
  publisher={Cambridge University Press}
}

@Article{LackerHier,
 Author = {Lacker, Daniel},
 Title = {Hierarchies, entropy, and quantitative propagation of chaos
          for mean field diffusions},
 FJournal = {Probability and Mathematical Physics},
 Journal = {Probab. Math. Phys.},
 ISSN = {2690-0998},
 Volume = {4},
 Number = {2},
 Pages = {377--432},
 Year = {2023},
 Language = {English},
 DOI = {10.2140/pmp.2023.4.377},
 zbMATH = {7709553},
 Zbl = {1515.82109}
}

@Book{FlemingSoner,
 Author = {Fleming, Wendell H. and Soner, H. Mete},
 Title = {Controlled {Markov} processes and viscosity solutions},
 Edition = {2nd ed.},
 FSeries = {Stochastic Modelling and Applied Probability},
 Series = {Stoch. Model. Appl. Probab.},
 ISSN = {0172-4568},
 Volume = {25},
 Year = {2006},
 Publisher = {New York, NY: Springer},
 Language = {English},
 DOI = {10.1007/0-387-31071-1},
 zbMATH = {5032360},
 Zbl = {1105.60005}
}

@Article{PorrettaPriolaGlobalLip,
 Author = {Porretta, Alessio and Priola, Enrico},
 Title = {Global {Lipschitz} regularizing effects for linear and nonlinear parabolic equations},
 FJournal = {Journal de Math{\'e}matiques Pures et Appliqu{\'e}es. Neuvi{\`e}me S{\'e}rie},
 Journal = {J. Math. Pures Appl. (9)},
 ISSN = {0021-7824},
 Volume = {100},
 Number = {5},
 Pages = {633--686},
 Year = {2013},
 Language = {English},
 DOI = {10.1016/j.matpur.2013.01.016},
 zbMATH = {6448859},
 Zbl = {1322.35066}
}

@Article{Conforti2023,
 Author = {Conforti, Giovanni},
 Title = {Coupling by reflection for controlled diffusion processes: turnpike property and large time behavior of {Hamilton}--{Jacobi}--{Bellman} equations},
 FJournal = {The Annals of Applied Probability},
 Journal = {Ann. Appl. Probab.},
 ISSN = {1050-5164},
 Volume = {33},
 Number = {6A},
 Pages = {4608--4644},
 Year = {2023},
 Language = {English},
 DOI = {10.1214/22-AAP1927},
 zbMATH = {7789643}
}

@Article{Eberle2016,
 Author = {Eberle, Andreas},
 Title = {Reflection couplings and contraction rates for diffusions},
 FJournal = {Probability Theory and Related Fields},
 Journal = {Probab. Theory Relat. Fields},
 ISSN = {0178-8051},
 Volume = {166},
 Number = {3-4},
 Pages = {851--886},
 Year = {2016},
 Language = {English},
 DOI = {10.1007/s00440-015-0673-1},
 zbMATH = {6657178},
 Zbl = {1367.60099}
}

@Book{HenryGeometricParabolic,
 Author = {Henry, Dan},
 Title = {Geometric theory of semilinear parabolic equations},
 FSeries = {Lecture Notes in Mathematics},
 Series = {Lect. Notes Math.},
 ISSN = {0075-8434},
 Volume = {840},
 Year = {1981},
 Publisher = {Springer},
 Address = {Cham},
 Language = {English},
 zbMATH = {3714138},
 Zbl = {0456.35001}
}

@Article{LindvallRogers1986,
 Author = {Lindvall, Torgny and Rogers, L. Christopher G.},
 Title = {Coupling of multidimensional diffusions by reflection},
 FJournal = {The Annals of Probability},
 Journal = {Ann. Probab.},
 ISSN = {0091-1798},
 Volume = {14},
 Pages = {860--872},
 Year = {1986},
 Language = {English},
 DOI = {10.1214/aop/1176992442},
 zbMATH = {3953996},
 Zbl = {0593.60076}
}
\end{document}